\documentclass[12pt]{amsart}

 \usepackage{epsfig}

\usepackage{amssymb}

\usepackage[utf8]{inputenc}
\usepackage[english]{babel}
\usepackage{xcolor}
\usepackage{environ}
\NewEnviron{mynote}{%
    \marginpar{\footnotemark}
    \footnotetext{{\color{black}\BODY}}}[\par]

\usepackage[margin=1in]{geometry}  
\usepackage{graphicx}              
\usepackage{amsmath}               
\usepackage{amsfonts}              
\usepackage{ marvosym } 

\usepackage{setspace}

\usepackage{amsthm}                
\usepackage{amssymb}              
\usepackage{tikz}                       
\usetikzlibrary{plotmarks}
\usetikzlibrary{shapes}
\usetikzlibrary{arrows}

\usepackage{listings} 
\usepackage{3dplot}                  
\usepackage{upgreek}
\usepackage{capt-of}
\usepackage{ stmaryrd }

\usepackage{caption}
\usepackage{url}

\usepackage{color}

\usepackage[pdftex]{hyperref}
\hypersetup{
colorlinks=false,
 linkcolor=blue!30!black,
    citecolor=green!50!black,
    urlcolor=blue!80!black,
    filecolor =yellow
}

\theoremstyle{plain}

\newtheorem{thm}{Theorem}[section]
\newtheorem{lem}[thm]{Lemma}

\newtheorem{prop}[thm]{Proposition}
\newtheorem{cor}[thm]{Corollary}

 \newtheoremstyle{TheoremNum}
        {\topsep}{\topsep}              
        {\itshape}                      
        {}                              
        {\bfseries}                     
        {.}                             
        { }                             
        {\thmname{#1}\thmnote{ \bfseries #3}}
    \theoremstyle{TheoremNum}
    \newtheorem{thmn}{Theorem}

\theoremstyle{definition}

\newtheorem{rem}{Remark}
\newtheorem{ex}{Example}

\newtheorem{defn}[thm]{Definition}

\DeclareMathOperator{\conv}{conv}
\DeclareMathOperator{\Spec}{Spec}

\DeclareMathOperator{\Relint}{Relint}
\DeclareMathOperator{\Edr}{Edr}
\DeclareMathOperator{\TV}{TV}

\DeclareMathOperator{\G}{G}
\DeclareMathOperator{\di}{dim}
\DeclareMathOperator{\casei}{(i)}
\DeclareMathOperator{\caseii}{(ii)}
\DeclareMathOperator{\caseiii}{(iii)}
\DeclareMathOperator{\caseiv}{(iv)}
\DeclareMathOperator{\casev}{(v)}

\DeclareMathOperator{\Def}{Def}

\DeclareMathOperator{\Val}{Val}

\newcommand{\RR}{\mathbb{R}}      
\newcommand{\QQ}{\mathbb{Q}}      
\newcommand{\ZZ}{\mathbb{Z}}      
\newcommand{\CC}{\mathbb{C}}      

\newcommand{\Ione}{\mathcal{I}_G^{(1)}}
\newcommand{\Itwo}{\mathcal{I}_G^{(2)}}
\newcommand{\Ithree}{\mathcal{I}_G^{(3)}}

\pdfoutput=1

\begin{document}

\title{On rigidity of toric varieties arising from bipartite graphs}

\author{Irem Portakal}

\address{Otto-von-Guericke-Universität Magdeburg\\ Department of Mathematics, Magdeburg, Germany}
\email{irem.portakal@ovgu.de}
\keywords{
bipartite graph, edge cone, toric variety, rigidity}
\subjclass[2010]{14B07, 14M25, 52B20 (Primary). 05C69 (Secondary).}

\begin{abstract}
One can associate to a bipartite graph a so-called edge ring whose spectrum is an affine normal toric variety. We characterize the faces of the (edge) cone associated to this toric variety in terms of some independent sets of the bipartite graph. By applying to this characterization the combinatorial study of deformations of toric varieties by Altmann, we present certain criteria for their rigidity purely in terms of graphs.  

\end{abstract}



\maketitle

\section{Introduction}\label{introduction}
We want to investigate rigidity of a certain family of affine toric varieties by utilizing combinatorial tools from bipartite graphs. While the investigation of rigidity in general is difficult, we are able to present simple criteria for the rigidity in the case where the toric variety is constructed from a graph. 

\subsection{Toric varieties arising from graphs}
Let $G$ be a simple graph. We denote its vertex set as $V(G)$ and its edge set as $E(G)$. One defines the {\it{edge ring}} associated to $G$ as $$\Edr(G):=\mathbb{C}[t_{i}t_{j} \ | \ \{i,j\} \in E(G), \ i,j \in V(G)].$$ 

\noindent Consider the surjective ring morphism
\begin{eqnarray*}
\mathbb{C}[x_e \ | \ e \in E(G)] & \rightarrow & \Edr(G) \\
x_e & \mapsto & t_i t_j
\end{eqnarray*}
where $e = \{i,j\} \in E(G)$.  The kernel $I_G$ of this morphism is called the \textit{edge ideal}. The associated toric variety to the graph $G$ is defined as $$\TV(G):=\Spec(\mathbb{C}[x_e \ | \ e \in E(G)]/I_G) = \Spec(\mathbb{C}[\sigma_G^{\vee} \cap M])$$ where $\sigma_G^{\vee}$ is called the \emph{dual edge cone}. The edge ring $\Edr(G)$ is an integrally closed domain and hence $\TV(G)$ is a normal variety. For more details on the theory, we refer the reader to \cite{binomialbook}. In this paper, we focus on the case where $G$ is a bipartite graph. We examine the first order deformations of $\TV(G)$, more precisely we investigate certain criteria for the bipartite graph $G$ such that the first order deformations of $\TV(G)$ are all trivial, that is~$\TV(G)$ is {\it rigid}. 

\subsection{Examples of rigid varieties}
The first example of a rigid singularity is the cone over the Segre embedding $\mathbb{P}^r \times \mathbb{P}^1$ in $\mathbb{P}^{2r +1}$ ($r \geq 1)$ which has been introduced by Grauert and Kerner in \cite{firstrigidpaper}. We will observe that this is in fact the toric variety associated to the complete bipartite graph $K_{r+1,2}$. In \hyperref[rigidcomplete]{Theorem \ref*{rigidcomplete}}, we also give an alternative proof to the classical fact in \cite{segreler} that $\TV(K_{m,n})$ is rigid except for $m=n = 2$. 

Another well-known class of rigid varieties is introduced by Schlessinger in \cite{schlessingerrigid}, which are isolated quotient singularities with dimension greater than three.\\

\noindent Many aspects of the infinitesimal deformations of affine normal toric varieties have been studied by K. Altmann in \cite{minkowskialtmann},\cite{invencionesaltmann} and \cite{altmann96}. In these papers, it has been shown that the first order deformations of affine normal toric varieties are multi-graded. The homogeneous pieces  $T_{X}^1(-R)$ are given by a so-called \textit{deformation degree} $R \in M$, where $M$ is the character group of the torus. For the computation of $T_{X}^1(-R)$, one first considers the \emph{crosscut} in degree $R$, which is the intersection of the associated cone of the toric variety with an affine space defined by $R$. Next, one examines the two-dimensional faces of this intersection and  loosely speaking, how these two-dimensional faces are connected to each other. In [\cite{minkowskialtmann}, Corollary 6.5.1], it was proved that if the affine toric variety is $\mathbb{Q}$-Gorenstein and $\QQ$-factorial in codimension 3 then it is rigid. Another approach for the deformations of affine toric varieties has been discussed in \cite{filip18} via Hochschild cohomology which provides concrete calculations for the homogenous piece $T^1_X(-R)$. We follow the technique by Altmann which we present in \hyperref[deformationtheory]{Section \ref*{deformationtheory}} and we introduce many additional interesting families (not necessarily $\QQ$-Gorenstein nor an isolated singularity) of rigid toric varieties throughout this paper.

\subsection{Studying rigidity in terms of graphs}
It suffices for our purposes to assume that $G \subseteq K_{m,n}$ is a connected bipartite graph. This has the following reason: if $G = G_1 \sqcup G_2 \subset K_{m,n}$ is not connected, then one calculates the edge cones $\sigma^{\vee}_{G_1} \subseteq M^1_{\QQ}$ and $\sigma^{\vee}_{G_2} \subseteq M^2_{\QQ}$ for the connected components $G_1$ and $G_2$ and obtains $\sigma^{\vee} _G = \sigma^{\vee}_{G_1} + \sigma^{\vee}_{G_2} \subseteq  M^1_{\QQ} \oplus  M^2_{\QQ}$. Hence, the associated toric variety is simply $\TV(G) = \TV(G_1) \times \TV(G_2)$. If one of these toric varieties is not rigid, then $\TV(G)$ is also not rigid. If every connected component of $G$ yields a rigid associated toric variety, then $\TV(G)$ is rigid.\\

\noindent The first attempt on investigating the rigidity of $\TV(G)$ has been done in \cite{herzog2015} where one considers the connected bipartite graph $G\subsetneq K_{n,n}$ with one edge removed from the complete bipartite graph $K_{n,n}$. Here $n$ denotes the number of vertices of the disjoint sets of $K_{n,n}$. They prove in Proposition 7.1 that $\TV(G)$ is rigid for $n\geq 4$ and $\TV(G)$ is not rigid for $n=3$. The proof is done by techniques from commutative algebra which we do not utilize.  \\

\noindent In order to apply the constructions in \hyperref[deformationtheory]{Section \ref*{deformationtheory}}, we first describe the  \textit{edge cone} $\sigma_G$ associated to the toric variety $\TV(G)$. A characterization for the extremal ray generators of the edge cone $\sigma_G$ was given by C.H. Valencia and R.H. Villarreal in \cite{valenciavillareal}. In \hyperref[11thm]{Theorem \ref*{11thm}}, we refine this approach and present a one-to-one correspondence between the set of extremal ray generators of the cone $\sigma_G$ and the so-called {\it first independent sets}. The advantage of using this approach is that we are able to determine the faces of $\sigma_G$. For this, one defines a spanning subgraph $\G\{A\} \subseteq G$ associated to the first independent set $A$.  In \hyperref[facetheorem]{Theorem \ref*{facetheorem}}, we prove that a set $S$ of $d$ first independent sets, equivalently a set of $d$ extremal rays of $\sigma_G$, spans a face of dimension $d$ if and only if $\bigcap_{A \in S} \G\{A\}$ has $d+1$ connected components. In particular, we describe the supporting hyperplane of a face by the degree sequence of this intersection graph.  \\

\noindent This result allows us to prove that $\TV(G)$ is smooth in codimension 2 (\hyperref[smooth2co]{Theorem \ref*{smooth2co}}). We provide a detailed exposition of three dimensional faces of $\sigma_G$ in \hyperref[threefacesection]{Section \ref*{threefacesection}} motivated by \hyperref[altmann96]{Theorem \ref*{altmann96}}. In particular, we determine when the edge cone $\sigma_G$ admits a non-simplicial three-dimensional face (\hyperref[existenceofadeg]{Theorem \ref*{existenceofadeg}}). In that case, we prove the following.

\begin{thmn}[\ref{4tuplenonrigid}]
Let $G \subseteq K_{m,n}$ be a connected bipartite graph. Assume that the edge cone $\sigma_G$ admits a three-dimensional non-simplicial face. Then $\TV(G)$ is not rigid.
\end{thmn}

\noindent As an application to the general investigation of the rigidity of $\TV(G)$ in \hyperref[generalch]{Section \ref*{generalch}}, we present the characterization of rigid affine toric varieties $\TV(G)$, where $G$ has exactly one two-sided first independent set $C=C_1 \sqcup C_2$.  In other words, we consider the connected bipartite graphs $G \subset K_{m,n}$ where we remove all the edges between two vertex sets $\emptyset \neq C_1\subsetneq U_1$ and $\emptyset \neq C_2 \subsetneq U_2$. Here, $U_1$ and $U_2$ denote the disjoint sets of the bipartite graph $G$. For the case where $|C_1|=1$ and $|C_2|=1$, we recover the result in [\cite{herzog2015}~Proposition 7.1] without the assumption of $m=n$. 

\begin{thmn} [\ref{onetwosided}]
Let $G \subsetneq K_{m,n}$ be a connected bipartite graph with exactly one two-sided first independent set $C \in \Ione$. Then 
\begin{enumerate}
\item $\TV(G)$ is not rigid, if $|C_1|=1$ and $|C_2| = n-2$ or if $|C_1|=m-2$ and $|C_2|=1$.
\item $\TV(G)$ is rigid, otherwise.
\end{enumerate} 
\end{thmn}

\noindent Throughout this work, many examples have been checked using the software \texttt{Polymake} \cite{polymake} and the computer algebra system \texttt{Singular} \cite{singular}. In \cite{portakal19}, we present the function which takes as an input the dual edge cone and outputs the information about rigidity of the associated toric variety. In particular it draws the representative picture of the crosscut for any given deformation degree $R\in M$.

\setcounter{tocdepth}{1}
\tableofcontents

\section{Characterization of the faces of an edge cone}\label{facechar}

\subsection{Edge Cones}\label{edgeideals}
The edge ideals $I_G$ for finite connected graphs having no loop and no multiple edge were studied in \cite{hibiwalk}. They are characterized explicitly in terms of primitive even closed walks and in particular of cycles without a chord in the bipartite case. Throughout this paper, we focus on the bipartite case and investigate the corresponding edge cones. Let $G\subseteq K_{m,n}$ be a connected bipartite graph and denote its disjoint sets by $U_1$ and $U_2$. We label the vertices in $U_1$ as $\{1, \ldots, m\}$ and the vertices in $U_2$ as $\{m+1, \ldots, m+n\}$. Let $e^{i}$ denote a canonical basis element of $\mathbb{Z}^{m} \times 0$ and  $f^{j}$ denote a canonical basis element of $0 \times \ZZ^n$. By construction of the edge ideal, one obtains that the dual edge cone $\sigma^{\vee}_G$ is generated by the ray generators $e^i + f^j \in \ZZ^{m+n}$, for $\{i,j\} \in E(G)$.

\begin{prop}\label{bubaslangic}
Let $G \subseteq K_{m,n}$ be a connected bipartite graph. Then the dimension of the dual edge cone $\sigma_G^{\vee}$ is $m+n-1$. 
\end{prop}

\begin{proof}
Let $A_G$ be the (incidence) matrix whose columns are the ray generators of $\sigma_G^{\vee}$. Suppose that $x \in \QQ^{m+n}$ is an element of $\text{coker}(A_G)$. Then $x_i + x_j = 0$ whenever there is a path from vertex $i$ to vertex $j$. Since $G$ is connected, we obtain that the corank of $A_G$ is at most one. However the rows of $A_G$ are linearly dependent and therefore the rank of $A_G$ is smaller than or equal to $m+n-1$. It follows that $\dim \sigma_G^{\vee} = m+n-1$.
\end{proof}

\begin{rem} If $G$ is not a tree, then the generators of the dual edge cone $\sigma^{\vee}_G$ in $\QQ^{m+n}$ are linearly dependent. The relations are formed by the cycles without a chord of $G$. If $G$ is a tree, $\sigma_G^{\vee}$ has $m+n-1$ generators. In both cases, the dual cone $\sigma_G^{\vee}$ is not a full dimensional cone in the vector space $\QQ^{m+n}$. Equivalently, the edge cone $\sigma_G \subseteq \QQ^{m+n}$ is not strongly convex.

\noindent We calculate $(\sigma_G^{\vee})^{\bot}$ as
 \[\{a \in \QQ^{m+n}\ | \ \langle b,a \rangle =0 \text{ for all } b  \in \sigma_G^{\vee} \} = \big\langle \sum_{i=1}^m e_i - \sum_{j=1}^n f_j \big\rangle.\]
 
\noindent The one-dimensional subspace $(\sigma_G^{\vee})^{\bot}$ is the minimal face of $\sigma_G \subseteq \QQ^{m+n}$. We denote it by $\overline{(1,-1)}$. Hence we consider the cone $\sigma_G/ \overline{(1,-1)} \subseteq \QQ^{m+n} / \overline{(1,-1)}$ which is a strongly convex polyhedral cone. Therefore we set the lattices we use for the \emph{edge}  and \emph{dual edge cone} as follows:
\[N:= \ZZ^{m+n}/\overline{(1,-1)} \text{ and } M:=\ZZ^{m+n} \cap \overline{(1,-1)}^{\bot}.\]
We denote their associated vector space as $N_{\QQ}:=N \otimes_{\ZZ}{\mathbb{Q}}$ and $M_{\QQ}:=M \otimes_{\ZZ}{\mathbb{Q}}$. In order to distinguish the elements of these vector spaces, we denote the ones in $N_{\QQ}$ by normal brackets and the ones in $M_{\QQ}$ by square brackets. For the same reason, we denote the canonical basis elements as $e_i \in N_{\QQ}$ and $e^i \in M_{\QQ}$.  In particular, we follow the notation of \cite{toricbook} for cones and toric varieties.
\end{rem}

\subsection{Description of the extremal rays of an edge cone}\label{extremalrayssection}
We start with certain definitions from Graph Theory. Although these definitions do also work for an arbitrary abstract graph $G$, we preserve our assumption of $G \subseteq K_{m,n}$ being connected and bipartite. 
\begin{defn} 
\begin{enumerate}
\item[]
\item A nonempty subset $A$ of $V(G)$ is called an \textit{independent set} if it contains no adjacent vertices.
\item The \textit{neighbor set} of $A \subseteq V(G)$ is defined as $$N(A) : = \{v \in V(G) \ | \ v \text{ is adjacent to some vertex in } A\}.$$
\item The \textit{supporting hyperplane} of the dual edge cone $\sigma^{\vee}_G \subseteq M_{\QQ}$ associated to an independent set $\emptyset \neq A $ is defined as $$\mathcal{H}_A : = \{x \in M_{\mathbb{Q}} \ | \ \sum_{v_i \in A} x_i = \sum_{v_i \in N(A)} x_i\}.$$
\end{enumerate}
\end{defn}
\noindent Note that since no pair of vertices of an independent set $A$ is adjacent, we obtain that $A \cap N(A) =\emptyset$.
\begin{defn}
\begin{enumerate}
\item[]
\item  A subgraph of $G$ with the same vertex set as $G$ is called a \textit{spanning subgraph} (or \textit{full subgraph}) of $G$.
\item Let $A\subseteq V(G)$ be a subset of the vertex set of $G$. The \textit{induced subgraph} of $A$ is defined as the subgraph of $G$ formed from the vertices of $A$ and all of the edges connecting pairs of these vertices. We denote it as $\G[A]$ and we have the convention $G[\emptyset]=\emptyset$.
\end{enumerate}
\end{defn}

\noindent The next Proposition shows that every facet of $\sigma_G^{\vee}$ can be constructed by an independent set satisfying certain conditions. Note that there is a one-to-one correspondence between the facets of $\sigma_G^{\vee}$ and the extremal rays of $\sigma_G$. The face $\tau \preceq \sigma_G^{\vee}$ is a facet of $\sigma_G^{\vee}$ if and only if $\tau^{*}: = \tau^{\bot} \cap \sigma_G$ is an extremal ray of $\sigma_G$. We will interpret the next result in \hyperref[11thm]{Theorem \ref*{11thm}} and give an alternative one-to-one description between the extremal ray generators of $\sigma_G$ and certain independent sets. This description allows us to study the faces of $\sigma_G$.

\begin{prop}\label{villaprop} ({\cite{valenciavillareal}, \ Proposition \ 4.1, 4.6})
Let $A \subsetneq U_i$ be an independent set. Then $\mathcal{H}_A\cap \sigma_G^{\vee}$ is a proper face of $\sigma_G^{\vee}$. In particular, if $A\subsetneq U_1$, then $\mathcal{H}_A \cap \sigma_G^{\vee}$ is a facet of $\sigma_G^{\vee}$ if and only if $\G[A \sqcup N(A)]$ and $\G[(U_1 \backslash A) \sqcup (U_2\backslash N(A))]$ are connected and their union is a spanning subgraph of $G$. Furthermore, any facet of $\sigma_G^{\vee}$ has the form $\mathcal{H}_A \cap \sigma_G^{\vee}$ for some $A\subsetneq U_i$, $i=1$ or $i=2$.
\end{prop}

\begin{ex}\label{kucuk}
Let $G \subsetneq K_{2,2}$ be the connected bipartite graph with disjoint sets $U_1 = \{1,2\}$ and $U_2 = \{3,4\}$ and the edge set $E(G) = E(K_{2,2}) \backslash \{1,3\}$. Recall that we have the edge cone $\sigma_G$ in $N_{\QQ} \cong \QQ^4 / (1,1,-1-1) \cong \QQ^3$ and the dual edge cone $\sigma_G^{\vee}$ in $M_{\QQ} \cong \QQ^4 \cap (1,1,-1,-1)^{\bot} \cong \QQ^3$. The three-dimensional cone $\sigma_G^{\vee} \subset M_{\RR}$ is generated by the extremal rays $[1,0,0,1]$, $[0,1,1,0]$ and $[0,1,0,1]$. The independent sets inducing the facets of $\sigma_G^{\vee}$ in \hyperref[villaprop]{Proposition \ref*{villaprop}} are colored in yellow. Here, the independent set $\{3\}$ is not considered, since we have $\mathcal{H}_{\{3\}}\cap \sigma_G^{\vee}= \mathcal{H}_{\{1\}}\cap \sigma_G^{\vee}$. The blue color represents the induced subgraph $\G[(U_1 \backslash A) \sqcup (U_2 \backslash N(A))]$ and the black color represents the induced subgraph $\G[A \sqcup N(A)]$. The graphs are labeled by their associated facets of $\sigma_G^{\vee}$. 
\begin{center}
\colorbox{white}{
\begin{minipage}[b]{0.27\linewidth}
\begin{tikzpicture}[baseline=1cm,scale=1.5,every path/.style={>=latex},every node/.style={draw,circle,fill=white,scale=0.6}]
  \node            (a) at (0,0)  {2};
  \node       [rectangle]     (b) at (1.5,0)  {4};
  \node            (c) at (0,1.5) {1};
  \node       [rectangle]     (d) at (1.5,1.5) {3};
\node [draw =none, fill=none,scale=1.6](e) at (0.75, -0.4) {$G$};

    \draw[-] (c) edge (b);
    \draw[-] (b) edge (a);
    \draw[-] (a) edge (d);
 
\end{tikzpicture}
\end{minipage}

\begin{minipage}[b]{0.24\linewidth}
\begin{tikzpicture}[baseline=1cm,scale=1.5,every path/.style={>=latex},every node/.style={draw,circle,fill=white,scale=0.6}]
  \node     [fill=yellow]          (a) at (0,0)  {2};
  \node        [rectangle]    (b) at (1.5,0)  {4};
  \node       [draw = cyan]     (c) at (0,1.5) {1};
  \node       [rectangle]     (d) at (1.5,1.5) {3};
\node [draw =none, fill=none,scale=1.6](e) at (0.75, -0.4) {$a_1$};

    \draw[-] (b) edge (a);
    \draw[-] (a) edge (d);

\end{tikzpicture}
\end{minipage}

\begin{minipage}[b]{0.24\linewidth}
\begin{tikzpicture}[baseline=1cm,scale=1.5,every path/.style={>=latex},every node/.style={draw,circle,fill=white,scale=0.6}]
  \node            (a) at (0,0)  {2};
  \node       [fill=yellow,rectangle]        (b) at (1.5,0)  {4};
  \node         (c) at (0,1.5) {1};
  \node      [rectangle, draw=cyan]      (d) at (1.5,1.5) {3};
\node [draw =none, fill=none,scale=1.6](e) at (0.75, -0.4) {$a_2$};

    \draw[-] (c) edge (b);
    \draw[-] (b) edge (a);
 
\end{tikzpicture}
\end{minipage}

\begin{minipage}[b]{2.4\linewidth}
\begin{tikzpicture}[baseline=1cm,scale=1.5,every path/.style={>=latex},every node/.style={draw,circle,fill=white,scale=0.6}]
  \node         [draw=cyan]   (a) at (0,0)  {2};
  \node           [rectangle] (b) at (1.5,0)  {4};
  \node       [fill=yellow]        (c) at (0,1.5) {1};
  \node       [draw=cyan,rectangle]        (d) at (1.5,1.5) {3};
\node [draw =none, fill=none,scale=1.6](e) at (0.75, -0.4) {$a_3$ };

    \draw[-] (c) edge (b);
    \draw[-,cyan] (a) edge (d);

\end{tikzpicture}
\end{minipage}

 }

\end{center}
~
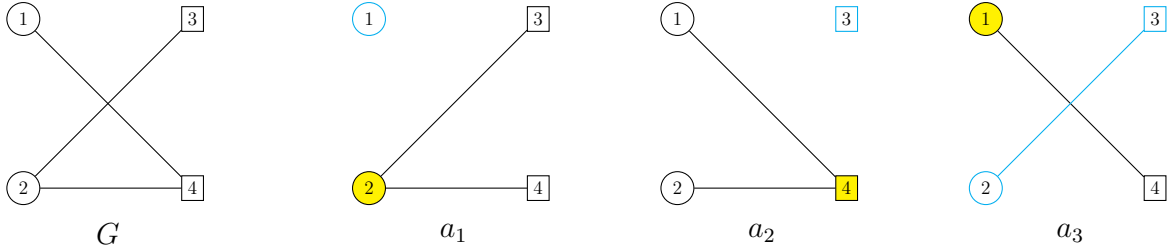
\captionof{figure}{The representation of the facets $a_1$, $a_2$, $a_3$ of the dual edge cone $\sigma_G^{\vee}$ of the connected bipartite graph $G$.} \label{ilkornek}
\vspace{0.5cm}
 Let us calculate the facet $a_1$ of $\sigma_G^{\vee}$ given by the independent set $A_1= \{2\}$. Equivalently, we calculate the extremal ray generator $a^{*}_1$ of $\sigma_G$. The supporting hyperplane associated to $A_1$ is $\mathcal{H}_{A_1} = \{ (x_1,x_2,x_3,x_4) \in M_{\RR} \ | \ x_2 = x_3 + x_4\}$. Therefore the facet $a_1$ is spanned by $[0,1,1,0]$ and $[0,1,0,1]$. In the same way, one obtains that $a_2$ is generated by $[1,0,0,1]$ and $[0,1,0,1]$ and $a_3$ is generated by $[1,0,0,1]$ and $[0,1,1,0]$. Moreover we obtain $a^*_1 = e_1$, $a^*_2 = e_3$, and $a^*_3 = e_2-e_3$.
\end{ex}

\begin{rem}\label{independentsarenoteverything}
The disjoint sets $U_1$ and $U_2$ of a bipartite graph $G \subseteq K_{m,n}$ are also independent sets and they induce the dual edge cone itself. For instance, in  \hyperref[kucuk]{Example \ref*{kucuk}}, if we consider the independent set $U_1= \{1,2\}$, then we obtain $\G[\{1,2\} \sqcup N(\{1,2\})] = G$, i.e.\ $\mathcal{H}_{\{1,2\}} \cap \sigma_G^{\vee}= \sigma_G^{\vee}$. On the other hand, not all faces of the edge cone can be induced by the independent sets $A\neq U_i$ as in \hyperref[villaprop]{Proposition \ref*{villaprop}}.  Let us consider the one-dimensional face $ a_1 \cap  a_2  = \langle [0,1,0,1] \rangle \prec \sigma_G^{\vee}$. It is represented by the edge $\{ 2,4\} \in E(G)$, but there exists no independent set $A$ such that $\langle [0,1,0,1] \rangle =\mathcal{H}_A \cap \sigma_G^{\vee}$. It is because one obtains that $\mathcal{H}_{U_2} \cap \sigma_G^{\vee}= \sigma_G^{\vee}$ and $\mathcal{H}_{\{1\}} \cap \sigma_G^{\vee}= \mathcal{H}_{\{3\}} \cap \sigma_G^{\vee} = a_3$.
\end{rem}
\noindent In order to state our alternative description in \hyperref[11thm]{Theorem \ref*{11thm}}, we present the central definitions that we need.
\begin{defn}
An independent set $A\subsetneq V(G)$ is called a \textit{maximal independent set} if there is no other independent set containing it. We say that an independent set is \textit{one-sided} if it is contained either in $U_1$ or in $U_2$. In a similar way, $A=A_1 \sqcup A_2$ is called a \textit{two-sided independent set} if $\emptyset \neq A_1 \subsetneq U_1$ and $ \emptyset \neq A_2 \subsetneq U_2$. 
\end{defn}

\noindent Let us consider the following independent sets:
\begin{eqnarray*}
\mathcal{I}_G^{(*)}:=\{ \text{Two-sided maximal independent sets}\} \sqcup \{ \text{One-sided independent sets } U_i \backslash \{\bullet\} \text{ not} \\ 
\text{      }  \text{        contained} \text{ in any two-sided independent set}\}
\end{eqnarray*}         
\noindent Here $\{\bullet\}$ stands for a single vertex in $U_i$.

\begin{defn}
Let $G[[A]]$ be the subgraph of $G$ associated to the independent set $A$ defined as 
$$ G[[A]]:=
  \left\{
                \begin{array}{ll}
                G[A \sqcup N(A)] \sqcup G[(U_1\backslash A) \sqcup (U_2\backslash N(A))], \text{if } A \subseteq U_1 \text{ is one-sided.}\\
                G[A \sqcup N(A)] \sqcup G[(U_2\backslash A) \sqcup (U_1\backslash N(A))], \text{if }A\subseteq U_2\text{ is one-sided.} \\
               G[A_1 \sqcup N(A_1)] \sqcup G[A_2 \sqcup N(A_2)], \text{if }A=A_1\sqcup A_2 \text{ is two-sided.} 
    
                \end{array}
              \right.  $$

\noindent We define the \emph{associated bipartite subgraph} $\G \{A\}  \subseteq G$ to the independent set $A$ as the spanning subgraph $G[[A]] \sqcup \big(V(G) \backslash V(G[[A]])\big)$.
\end{defn}

\noindent Now, we come to our characterization of the independent sets, defining a facet of $\sigma_G^{\vee}$. 
\begin{defn}\label{firstindependentsets}

We say that $A \in \mathcal{I}_G^{(*)}$ is a first independent set if the associated subgraph $\G\{A\}$ has two connected components. We denote the set of first independent sets by  $\mathcal{I}_G^{(1)}$.           
\end{defn}

\begin{ex}
Let $G\subsetneq K_{2,2}$ be the connected bipartite graph from \hyperref[kucuk]{Example \ref*{kucuk}}. We observe that $\{1\} \sqcup \{3\}$ is a two-sided maximal independent set and the associated subgraph $\G\{\{1\} \sqcup \{3\}\}$ is the fourth bipartite graph in \hyperref[ilkornek]{Figure \ref*{ilkornek}}. Likewise, the second and third graphs are the associated subgraphs $\G\{\{2\}\}$ and $\G\{\{4\}\}$ to the one-sided independent sets $\{2\} \subset U_1$ and $\{4\} \subset U_2$ which are not contained in any two-sided independent set. Note that all these associated subgraphs have two connected components and thus these independent sets are first independent sets. In particular, we have that $G =\G\{\{1,2\}\}=\G\{\{3,4\}\}$.
\end{ex}

\begin{thm}\label{11thm}
There is a one-to-one correspondence between the set of extremal generators of the cone $\sigma_G$ and the first independent set $\mathcal{I}_G^{(1)}$ of $G$. In particular, the map is given as
\begin{eqnarray*}
 \pi \colon \mathcal{I}_G^{(1)} &\longrightarrow &\sigma_G^{(1)} \\
A &\mapsto&  \mathfrak{a} := (\mathcal{H}_{A_i} \cap \sigma_G^{\vee})^{*}
\end{eqnarray*}
for a fixed $i \in \{1,2\}$ with $A_i \neq \emptyset$.
\end{thm}

\noindent Note that we preserve the curly notation $\mathfrak{a}:=\pi(A)$ for an extremal ray $\mathfrak{a}\preceq\sigma_G$ associated to $A \in \mathcal{I}_G^{(1)}$ for the rest of this paper. We also identify an extremal ray with its primitive ray generator.

\subsection{Proof of \hyperref[11thm]{Theorem \ref*{11thm}}.}
This section is devoted to the proof of the one-to-one correspondence between the set of extremal ray generators of $\sigma_G$ and the first independent sets of $G$. For this end, we prove certain combinatorial properties of independent sets and consequently we describe the supporting hyperplane for the extremal ray $\mathfrak{a}$ by the degree sequence of the graph $G\{A\}$. 
\begin{prop}\label{prop3}
Let $A=A_1 \sqcup A_2 \in \mathcal{I}_G^{(*)}$ be a two-sided maximal independent set. Then, one has $N(A_2) = U_1 \backslash A_1$ and $A_2 = U_2 \backslash N(A_1)$. 
\end{prop}

\begin{proof}
Let $x \in N(A_2)$. By definition there exists a vertex $y \in A_2$ such that $\{x, y\} \in E(G)$. Since $A$ is an independent set, $x$ can not be in $A_1$. Conversely, let $x \in U_1 \backslash A_1$. Since $G$ is connected, there exists a vertex $y \in U_2$ such that $\{x, y\} \in E(G)$. Suppose that $x \notin N(A_2)$. This means that for any $a_2 \in A_2$, $\{x, a_2\} \notin E(G)$. This implies that $x \in A_1$ by maximality of the independent set A. The other equality follows similarly. 
\end{proof}

\begin{lem}\label{lemonesid}
Let $A \in \mathcal{I}_G^{(*)}$ be a one-sided independent set not contained in any two-sided independent set. Then $N(A)$ is equal to one of the disjoint sets of $G$. If $\mathcal{H}_A \cap \sigma_G^{\vee}$ is a facet of $\sigma_G^{\vee}$, then $A=U_i \backslash \{u_i\}$ for some $u_i \in U_i$. Moreover, one obtains the following equality $\mathcal{H}_A \cap \sigma_G^{\vee} = H_{e_i} \cap \sigma_G^{\vee}$ where $H_{e_i}$ denotes the supporting hyperplane of $\sigma^{\vee}_G$ associated to $e_i$. 
\end{lem}

\begin{proof}
Let $A \subsetneq U_1$ be a one-sided independent set. Suppose that $N(A) \neq U_2$, then $A \sqcup (U_2 \backslash N(A))$ is a two-sided independent set containing $A$. Hence, if $A$ is a one-sided independent set not contained in any two-sided independent set, then $N(A)=U_2$. Suppose now that $\mathcal{H}_A \cap \sigma_G^{\vee}$ is a facet of $\sigma_G^{\vee}$, then the induced subgraph $\G[(U_1 \backslash A) \sqcup (U_2 \backslash N(A))]$ consists of isolated vertices. Thus this induced subgraph is connected if and only if $|A| = m-1$. The supporting hyperplane $\mathcal{H}_A$ associated to $A$ is
\[\{x \in M_{\mathbb{Q}} \ | \ x_1 +...+\widehat{x_i} +... x_m = x_{m+1} + ... + x_{m+n} \}.\]
where $\widehat{x_i}$ indicates that $x_i$ is omitted.
Since the chosen lattice $N = \ZZ^{m+n}/ \overline{(1,-1)}$, we obtain the equality $\mathcal{H}_A \cap \sigma_G^{\vee} = H_{e_i} \cap \sigma_G^{\vee}$. In particular $H_{e_i}$ is a supporting hyperplane, since $H_{e_i}^{+} \supset \sigma_G^{\vee}$ or equivalently $e_i \in \sigma_G$.
\end{proof}

\begin{rem}\label{spanningtwo}
Let $A \in \mathcal{I}_G^{(*)}$ be a two-sided maximal independent set. By the equalities from \hyperref[prop3]{Proposition \ref*{prop3}}, we observe that \[\G[A_1 \sqcup N(A_1)] = \G[(U_1 \backslash N(A_2)) \sqcup (U_2 \backslash A_2)] \] \[ \G[(U_1 \backslash A_1) \sqcup (U_2 \backslash N(A_1))] = \G[A_2 \sqcup N(A_2)]\] hold. In particular, the union of induced subgraphs $\G[A_1 \sqcup N(A_1)]$ and $\G[(U_1 \backslash A_1) \cup (U_2 \backslash N(A_1))]$ is a spanning subgraph. If $A_i=U_i \backslash \{u_i\} \in \mathcal{I}_G^{(*)}$, then the union of the induced subgraphs $\G[A_i \sqcup N(A_i)]$ and $\G[(U_i \backslash A_i) \cup (U_j \backslash N(A_j))] = u_i$ is a spanning subgraph of $G$.
\end{rem}

\begin{lem}\label{lemmaxindep}
If $A=A_1 \sqcup A_2$ is a two-sided independent set and if $\mathcal{H}_{A_1} \cap \sigma_G^{\vee}$ is a facet of $\sigma_G^{\vee}$, then there exists a maximal two-sided independent set $A' = A_1 \sqcup A'_2 \in \mathcal{I}_G^{(*)}$ for some vertex set $A'_2 \supseteq A_2$. 
\end{lem}
\begin{proof}
Assume that the two-sided maximal independent set $A' = A_1 \sqcup A'_2$ is not maximal, i.e.\ there exists a vertex set $A'_1 \supsetneq A_1$ such that $A'' =A'_1 \sqcup A'_2$ is a maximal two-sided independent set. Let $v \in A'_1 \backslash A_1$ be a vertex.  By \hyperref [villaprop] {Proposition \ref*{villaprop}}, since $\mathcal{H}_{A_1} \cap \sigma^{\vee}_G$ is a facet of $\sigma_G^{\vee}$, the induced subgraph $\G[(U_1 \backslash A_1) \sqcup (U_2 \backslash N(A_1))] = \G[(U_1 \backslash A_1) \sqcup A'_2]$ must be connected. However $v$ is an isolated vertex in  $\G[(U_1 \backslash A_1) \sqcup (U_2 \backslash N(A_1))]$ which is a contradiction.
\end{proof}

\begin{rem}\label{egal}
We observe that there is a symmetry for the supporting hyperplanes for a two-sided maximal independent set $A=A_1 \sqcup A_2 \in \mathcal{I}_G^{(*)}$. Recall that the supporting hyperplane associated to a one-sided independent set $A_i\subseteq U_i$ is defined as $$\mathcal{H}_{A_i} = \{x \in M_{\mathbb{Q}} \ | \ \sum_{v_i \in A_i} x_i = \sum_{v_i \in N(A_i)} x_i\}.$$
\noindent Assume that $x \in \mathcal{H}_{A_1} \cap \sigma_G^{\vee}$. By the definition  of $\mathcal{H}_{A_i}$ and since $M_{\QQ} \cong \QQ^{m+n} \cap \overline{(1,-1)}^{\bot}$, it follows that $ \sum_{v_i \in N(A_2)} x_i = \sum_{v_i \in A_2} x_i$ and hence $x\in \mathcal{H}_{A_2} \cap\sigma^{\vee}_G $. Therefore it is enough to consider only one component $A_i$ of the maximal two-sided independent set $A=A_1 \sqcup A_2$ for the associated supporting hyperplane.
\end{rem}

\begin{ex}\label{tekindependentli}
Let $G \subsetneq K_{4,4}$ be the connected bipartite graph with the edge set $E(G) = E(K_{4,4}) \backslash \{\{1,5\},\{2,5\},\{3,5\}\}$. We consider the one-sided independent set $A = \{1,2,3\}$. Since $N(A) = \{ 6,7,8\} \subsetneq U_2$, it is contained in a two-sided independent set which is $\{1,2,3,5\}$. We observe in the figure below that this two-sided independent set forms a facet $\tau$ of $\sigma_G^{\vee}$ and it is maximal. Therefore, one obtains that $\tau = \mathcal{H}_{\{1,2,3\}} \cap \sigma_G^{\vee} = \mathcal{H}_{\{5\}} \cap \sigma_G^{\vee}$.\\

\hspace{1.5cm}
\begin{tikzpicture}[baseline=1cm,scale=1.5,every path/.style={>=latex},every node/.style={draw,circle,fill=white,scale=0.6}]
  \node            (a2) at (0,1.5)  {2};
  \node            (a3) at (0,1)  {3};
  \node            (a4) at (0,0.5)  {4};
\node       [rectangle]     (b8) at (1.5,0.5)  {8};
\node       [rectangle]     (b7) at (1.5,1)  {7}; 
 \node       [rectangle]     (b6) at (1.5,1.5)  {6};
  \node            (a1) at (0,2) {1};
  \node       [rectangle]     (b5) at (1.5,2) {5};

    \draw[-] (a1) edge (b6);
    \draw[-] (a1) edge (b7);
    \draw[-] (a1) edge (b8);

    \draw[-] (a2) edge (b6);
    \draw[-] (a2) edge (b7);
    \draw[-] (a2) edge (b8);

    \draw[-] (a3) edge (b6);
    \draw[-] (a3) edge (b7);
    \draw[-] (a3) edge (b8);

    \draw[-] (a4) edge (b6);
    \draw[-] (a4) edge (b7);
    \draw[-] (a4) edge (b8);
    \draw[-] (a4) edge (b5);
 
\end{tikzpicture}
\hspace{2cm}
\begin{tikzpicture}[baseline=1cm,scale=1.5,every path/.style={>=latex},every node/.style={draw,circle,fill=white,scale=0.6}]
  \node         [fill=yellow]   (a2) at (0,1.5)  {2};
  \node        [fill=yellow]    (a3) at (0,1)  {3};
  \node            [draw= black] (a4) at (0,0.5)  {4};
\node       [rectangle]     (b8) at (1.5,0.5)  {8};
\node       [rectangle]     (b7) at (1.5,1)  {7}; 
 \node       [rectangle]     (b6) at (1.5,1.5)  {6};
  \node    [fill=yellow]        (a1) at (0,2) {1};
  \node        [rectangle]     (b5) at (1.5,2) {5};

    \draw[-] (a1) edge (b6);
    \draw[-] (a1) edge (b7);
    \draw[-] (a1) edge (b8);

    \draw[-] (a2) edge (b6);
    \draw[-] (a2) edge (b7);
    \draw[-] (a2) edge (b8);

    \draw[-] (a3) edge (b6);
    \draw[-] (a3) edge (b7);
    \draw[-] (a3) edge (b8);

\end{tikzpicture}
\hspace{2cm}
\begin{tikzpicture}[baseline=1cm,scale=1.5,every path/.style={>=latex},every node/.style={draw,circle,fill=white,scale=0.6}]
  \node       [fill=yellow]     (a2) at (0,1.5)  {2};
  \node       [fill=yellow]     (a3) at (0,1)  {3};
  \node       [draw=cyan] (a4) at (0,0.5)  {4};
\node       [rectangle]     (b8) at (1.5,0.5)  {8};
\node       [rectangle]     (b7) at (1.5,1)  {7}; 
 \node       [rectangle]     (b6) at (1.5,1.5)  {6};
  \node      [fill=yellow]      (a1) at (0,2) {1};
  \node       [fill=yellow,draw=cyan,rectangle]     (b5) at (1.5,2) {5};
\node [draw=none, fill=none,scale=1.3] (e) at (0.75, -0.2) { };
   \draw[-] (a1) edge (b6);
    \draw[-] (a1) edge (b7);
    \draw[-] (a1) edge (b8);
    \draw[-] (a2) edge (b6);
    \draw[-] (a2) edge (b7);
    \draw[-] (a2) edge (b8);
    \draw[-] (a3) edge (b6);
    \draw[-] (a3) edge (b7);
    \draw[-] (a3) edge (b8);
    \draw[-,thick,cyan] (a4) edge (b5);
\end{tikzpicture}

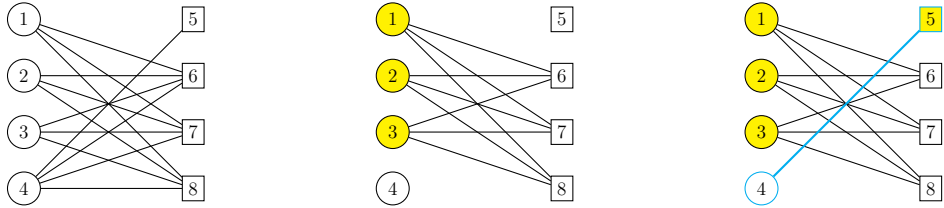
\captionof{figure}{The associated subgraphs $\G\{\{1,2,3\}\}$ and $\G\{\{1,2,3,5\}\}$}\label{benimteo}
\vspace{0.5cm}
Moreover, the independent sets of form $U_i \backslash \{\bullet\} \in \mathcal{I}_G^{(*)}$ other than $A$ give the remaining facets of $\sigma_G^{\vee}$. 
\end{ex}

\noindent  By \hyperref[lemonesid]{Lemma \ref*{lemonesid}} and \hyperref[lemmaxindep]{Lemma \ref*{lemmaxindep}}, we obtain the following theorem.

\begin{thm}\label{halfonetoone}
If $\mathcal{H}_{A_1} \cap \sigma_G^{\vee}$ is a facet of $\sigma_G^{\vee}$, then there exists an independent set $A=A_1 \sqcup A_2 \in \mathcal{I}_G^{(*)}$.
\end{thm}

\noindent For $A=A_1 \sqcup A_2 \in \mathcal{I}_G^{(*)}$, the induced subgraphs $\G[A_1 \sqcup N(A_1)]$ and $\G[A_2 \sqcup N(A_2)]$ might be not connected. In the next example, we observe that $\mathcal{I}_G^{(*)}$ is a necessary but not a sufficient condition to form a facet. This remark and \hyperref[prop239]{Proposition \ref*{prop239}} will be useful for us once we start describing the lower dimensional faces of $\sigma_G^{\vee}$ in \hyperref[facesection]{Section \ref*{facesection}}.

\begin{ex}\label{notconnected}
Let $G \subsetneq K_{4,4}$ be as in the figure below. Consider the two-sided maximal independent set $A=A_1 \sqcup A_2 = \{1,2\}\sqcup \{5,6\}$. We obtain that $N(A_1) = \{7,8\}$ and $N(A_2)=\{3,4\}$. One can observe that although $A=\{1,2,5,6\}$ is a maximal two-sided independent set, the induced subgraph $G[A_1 \sqcup N(A_1)]$ is not a connected graph. \\
\begin{center}
\begin{tikzpicture}[baseline=1cm,scale=1.5,every path/.style={>=latex},every node/.style={draw,circle,fill=white,scale=0.6}]
  \node            (a2) at (0,1.5)  {2};
  \node            (a3) at (0,1)  {3};
  \node            (a4) at (0,0.5)  {4};
\node       [rectangle]     (b8) at (1.5,0.5)  {8};
\node       [rectangle]     (b7) at (1.5,1)  {7}; 
 \node       [rectangle]     (b6) at (1.5,1.5)  {6};
  \node            (a1) at (0,2) {1};
  \node       [rectangle]     (b5) at (1.5,2) {5};

    \draw[-] (a1) edge (b8);

    \draw[-] (a2) edge (b7);

    \draw[-] (a3) edge (b5);
    \draw[-] (a3) edge (b6);
    \draw[-] (a3) edge (b7);
    \draw[-] (a3) edge (b8);

    \draw[-] (a4) edge (b6);
    \draw[-] (a4) edge (b7);
    \draw[-] (a4) edge (b8);
    \draw[-] (a4) edge (b5);
 
\end{tikzpicture}
\hspace{1cm}
\begin{tikzpicture}[baseline=1cm,scale=1.5,every path/.style={>=latex},every node/.style={draw,circle,fill=white,scale=0.6}]
  \node         [fill=yellow]   (a2) at (0,1.5)  {2};
  \node       [draw=cyan]    (a3) at (0,1)  {3};
  \node          [draw=cyan]   (a4) at (0,0.5)  {4};
\node       [rectangle]     (b8) at (1.5,0.5)  {8};
\node       [rectangle]     (b7) at (1.5,1)  {7}; 
 \node       [rectangle, fill=yellow,draw=cyan]     (b6) at (1.5,1.5)  {6};
  \node    [fill=yellow]        (a1) at (0,2) {1};
  \node        [fill=yellow,draw=cyan] [rectangle]     (b5) at (1.5,2) {5};

    \draw[-] (a1) edge (b8);

    \draw[-] (a2) edge (b7);

    \draw[-,cyan] (a3) edge (b6);
    \draw[-,cyan] (a3) edge (b5);
    \draw[-,cyan] (a4) edge (b5);
   \draw[-,cyan] (a4) edge (b6);

\end{tikzpicture}
\end{center}

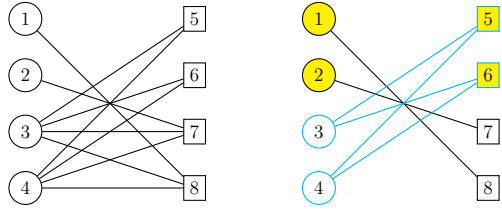
\captionof{figure}{The associated subgraph $\G\{\{1,2,5,6\}\}$}
\end{ex}
\noindent In the next proposition, we examine the case where $\G\{A\}$ has more than two connected components. We show that these independent sets in fact give rise to certain facets of $\sigma_G^{\vee}$.
\begin{prop}\label{prop239} Let $A = A_1 \sqcup A_2 \in \mathcal{I}_G^{(*)}$ be an independent set. Suppose that the induced subgraph $G[A_1 \sqcup N(A_1)]$ consists of $d$ connected bipartite graphs $G_i$ with vertex sets $X_i \subseteq A_1$ and $N(X_i) \subseteq N(A_1)$ and the induced subgraph $G[(U_1 \backslash A_1) \sqcup (U_2 \backslash N(A_1))]$ is connected. Then, for each $i \in [d]$ there exist two-sided maximal independent sets $X_i \sqcup (A_2 \sqcup \bigsqcup_{j \neq i} N(X_j))$ forming facets of $\sigma_G^{\vee}$. \end{prop}
\begin{proof}
We have two cases to examine:\\

\noindent$\textbf{(i)} $ Let $A_1=U_1 \backslash \{u\}$. We obtain the two-sided maximal independent sets $X_i \sqcup (\bigsqcup_{j \neq i} N(X_j)) $. Since $G$ is connected, for each $i \in [d]$, there exists a vertex $x_i \in N(X_i)\subseteq N(A_1)$ such that $\{u,x_i\} \in E(G)$. The  associated subgraphs $\G\{X_i \sqcup (\bigsqcup_{j \neq i} N(X_j))\}$ have therefore two connected components. Thus, these maximal independent sets form facets of $\sigma_G^{\vee}$. \\

\noindent$\textbf{(ii)} $ Let $A=A_1 \sqcup A_2$ be a two-sided maximal independent set. We obtain again the two-sided maximal independent sets $X_i \sqcup (A_2 \sqcup \bigsqcup_{j \neq i} N(X_j))$. Since $G$ is connected, $N(N(A_2)) \supset A_2 \sqcup \bigcup_{i \in [k]} x_i$, where $x_i \in N(X_i)$. Therefore, the associated subgraphs 
$\G\{X_i \sqcup (A_2 \sqcup \bigsqcup_{j \neq i} N(X_j))\}$ have two connected components. Thus, these maximal independent sets form facets of $\sigma_G^{\vee}$. \\

\noindent In particular, if $A_2 \neq \emptyset$, one can state the proposition symmetrically with $G[A_2 \sqcup N(A_2)]$ having $d$ connected components and $\G[A_1 \sqcup N(A_1)]$ being connected.
\end{proof}

\begin{ex}\label{sameexamplefig}
Consider the graph $G \subsetneq K_{4,4}$ from \hyperref[notconnected]{Example \ref*{notconnected}} and the maximal two-sided independent set $A=\{1,2,5,6\}$. The induced subgraph $\G[A_2 \sqcup N(A_2)]$ is connected. The two connected bipartite graphs of $\G[A_1 \sqcup N(A_1)]$ have the vertex sets $X_1 \sqcup N(X_1) := \{1\}\sqcup \{8\}$ and $X_2 \sqcup N(X_2):=\{2\}\sqcup \{7\}$. Hence, we obtain the following two-sided maximal independent sets forming facets of $\sigma_G^{\vee}$: $A':=X_1 \sqcup A_2\sqcup N(X_2) = \{1,5,6,7\}  \text{ and } A'':=X_2 \sqcup A_2\sqcup N(X_1) = \{2,5,6,8\}.$ We see in the figure below that their associated subgraphs have two connected components while $\G\{A\}$ has three connected components. In particular, we will observe in \hyperref[faceexampless]{Example \ref*{faceexampless}} that $(\mathcal{H}_{A_1} \cap \sigma_G^{\vee})^{*}$ is actually a two-dimensional face of $\sigma_G$.

\begin{center}
\colorbox{white}{
\begin{minipage}[b]{0.27\linewidth}
\begin{tikzpicture}[baseline=1,scale=1.5,every path/.style={>=latex},every node/.style={draw,circle,fill=white,scale=0.6}]
  \node            (a2) at (0,1.5)  {2};
  \node            (a3) at (0,1)  {3};
  \node            (a4) at (0,0.5)  {4};
\node       [rectangle]     (b8) at (1.5,0.5)  {8};
\node       [rectangle]     (b7) at (1.5,1)  {7}; 
 \node       [rectangle]     (b6) at (1.5,1.5)  {6};
  \node            (a1) at (0,2) {1};
  \node       [rectangle]     (b5) at (1.5,2) {5};
\node [draw =none, fill=none,scale=1.6](e) at (0.75, -0.4) {$G$};

    \draw[-] (a1) edge (b8);

    \draw[-] (a2) edge (b7);

    \draw[-] (a3) edge (b5);
    \draw[-] (a3) edge (b6);
    \draw[-] (a3) edge (b7);
    \draw[-] (a3) edge (b8);

    \draw[-] (a4) edge (b6);
    \draw[-] (a4) edge (b7);
    \draw[-] (a4) edge (b8);
    \draw[-] (a4) edge (b5);
 
\end{tikzpicture}
\end{minipage}

\begin{minipage}[b]{0.24\linewidth}
\begin{tikzpicture}[baseline=1,scale=1.5,every path/.style={>=latex},every node/.style={draw,circle,fill=white,scale=0.6}]
  \node         [fill=yellow]   (a2) at (0,1.5)  {2};
  \node      [draw=cyan]     (a3) at (0,1)  {3};
  \node         [draw=cyan]    (a4) at (0,0.5)  {4};
\node       [rectangle]     (b8) at (1.5,0.5)  {8};
\node       [rectangle]     (b7) at (1.5,1)  {7}; 
 \node       [rectangle, fill=yellow,draw=cyan]     (b6) at (1.5,1.5)  {6};
  \node    [fill=yellow]        (a1) at (0,2) {1};
  \node        [fill=yellow,draw=cyan] [rectangle]     (b5) at (1.5,2) {5};
\node [draw =none, fill=none,scale=1.6](e) at (0.75, -0.4) {$\G\{A\}$};

    \draw[-] (a1) edge (b8);

    \draw[-] (a2) edge (b7);

    \draw[-,cyan] (a3) edge (b6);
    \draw[-,cyan] (a3) edge (b5);
    \draw[-,cyan] (a4) edge (b5);
   \draw[-,cyan] (a4) edge (b6);

\end{tikzpicture}
\end{minipage}

\begin{minipage}[b]{0.24\linewidth}
\begin{tikzpicture}[baseline=1,scale=1.5,every path/.style={>=latex},every node/.style={draw,circle,fill=white,scale=0.6}]
  \node           (a2) at (0,1.5)  {2};
  \node       [draw=cyan]    (a3) at (0,1)  {3};
  \node         [draw=cyan]   (a4) at (0,0.5)  {4};
\node       [rectangle]     (b8) at (1.5,0.5)  {8};
\node       [fill=yellow, rectangle,draw=cyan]     (b7) at (1.5,1)  {7}; 
 \node       [fill=yellow, rectangle,draw=cyan]     (b6) at (1.5,1.5)  {6};
  \node      [fill=yellow]      (a1) at (0,2) {1};
  \node       [fill=yellow,rectangle,draw=cyan]     (b5) at (1.5,2) {5};
\node [draw =none, fill=none,scale=1.6](e) at (0.75, -0.4) {$\G\{A'\}$};

    \draw[-] (a1) edge (b8);

    \draw[-] (a2) edge (b7);

    \draw[-,cyan] (a3) edge (b6);
    \draw[-,cyan] (a3) edge (b7);
    \draw[-,cyan] (a3) edge (b5);

  \draw[-,cyan] (a4) edge (b6);
    \draw[-,cyan] (a4) edge (b7);
    \draw[-,cyan] (a4) edge (b5);

\end{tikzpicture}
\end{minipage}

\begin{minipage}[b]{0.24\linewidth}
\begin{tikzpicture}[baseline=1,scale=1.5,every path/.style={>=latex},every node/.style={draw,circle,fill=white,scale=0.6}]
  \node         [fill=yellow]   (a2) at (0,1.5)  {2};
  \node           [draw=cyan](a3) at (0,1)  {3};
  \node       [draw=cyan]  (a4) at (0,0.5)  {4};
\node       [rectangle,fill=yellow,draw=cyan]     (b8) at (1.5,0.5)  {8};
\node       [rectangle]     (b7) at (1.5,1)  {7}; 
 \node       [rectangle, fill=yellow,draw=cyan]     (b6) at (1.5,1.5)  {6};
  \node       [draw=cyan]    (a1) at (0,2) {1};
  \node        [fill=yellow,draw=cyan] [rectangle]     (b5) at (1.5,2) {5};
\node [draw =none, fill=none,scale=1.6](e) at (0.75, -0.4) {$\G\{A''\}$};

    \draw[-,cyan] (a1) edge (b8);

    \draw[-] (a2) edge (b7);

\draw[-,cyan] (a3) edge (b5);
    \draw[-,cyan] (a3) edge (b6);

    \draw[-,cyan] (a3) edge (b8);

\draw[-,cyan] (a4) edge (b5);
    \draw[-,cyan] (a4) edge (b6);
  
    \draw[-,cyan] (a4) edge (b8);

\end{tikzpicture}
\end{minipage}

 }

\end{center}

\end{ex}

\noindent With the motivation of \hyperref[prop239]{Proposition \ref*{prop239}} and \hyperref[firstindependentsets]{Definition \ref*{firstindependentsets}}, we are ready to prove \hyperref[11thm]{Theorem \ref*{11thm}}.

\begin{proof} [Proof of Theorem \ref{11thm}]
By \hyperref[halfonetoone]{Theorem \ref*{halfonetoone}} and \hyperref[villaprop]{Proposition \ref*{villaprop}}, the map $\pi$ is surjective. Let $A=A_1\in \mathcal{I}_G^{(1)}$ be a one-sided and $B\in \mathcal{I}_G^{(1)}$ be a two-sided first independent set. Suppose that we have the equality $ (\mathcal{H}_{A_1} \cap \sigma_G^{\vee})^* = (\mathcal{H}_{B_1} \cap \sigma_G^{\vee})^*$. Then $\mathfrak{a} = e_i$ for some $i \in [m]$ by \hyperref[lemonesid]{Lemma \ref*{lemonesid}} with $A_1 = U_1\backslash \{i\}$. Since $A_1$ is not contained in any two-sided independent set, $B_1 \neq A_1$. Assume that $B_1 \supseteq \{i\}$ or $B_1 \subsetneq A_1$, then $e^i + f^j \in \mathcal{H}_{B_1} \cap \sigma_G^{\vee}$ but not in $\mathcal{H}_{A_1} \cap \sigma_G^{\vee}$ where $m+j \in N(B_1)$ or $m+j \in B_2$ respectively. Thus $A$ and $B$ are either both one-sided or both two-sided. If they are both one-sided, then $A=B$. Let both of them be two-sided and assume that we have $A_1= N(B_2)$ and $B_2 = N(A_1)$ and that $N(N(B_2))= B_2$ and $N(N(A_1)) = A_1$. This means that $G$ is not connected. Therefore, we obtain that $A=B$.
\end{proof}

\noindent One may also describe the supporting hyperplane for the extremal ray $\mathfrak{a}=\pi(A)$ by the degree sequence of the graph $G\{A\}$. 

\begin{defn}
The {\it degree (valency) sequence} of a graph $G \subseteq K_{m,n}$ is the $(m+n)$-tuple of the degrees (valencies) of its vertices. Let $A \in \mathcal{I}_G^{(1)}$ be a first independent set. We denote the degree sequence of the associated subgraph $\G\{A\}$ by $\Val(A)$ $\in \sigma_G^{\vee} \cap M$. We denote the supporting hyperplane of $\sigma_G \subseteq N_{\QQ}$ for $m \in \sigma_G^{\vee}$ as 
$$H_m :=\{x \in N_{\QQ} \ | \ \langle m,x \rangle =0\}.$$
\end{defn}

\begin{prop}\label{hilbert}
The extremal ray generators of the cone $\sigma_G^{\vee}$ for a bipartite graph $G$ form the Hilbert basis of $\sigma_G^{\vee}$. 
\end{prop}

\begin{proof}
See [\cite{valenciavillareal}, Lemma 3.10]. The notation $\mathbb{R}_{+}\mathcal{A}$ used in this paper is $\sigma^{\vee}_G \subseteq M \otimes_{\ZZ}{\mathbb{R}}$ in our context. Also, $\mathbb{N} \mathcal{A}$ stands for the semigroup generated by the generators of $\sigma^{\vee}_G$ with nonnegative integer coefficients. 
\end{proof}

\begin{thm}\label{thmgibi}
Let $A \in \mathcal{I}_G^{(1)}$ be a first independent set. Then, the extremal ray generators of the facet $\pi(A)^{*}=\mathfrak{a}^* \prec \sigma_G^{\vee}$ are exactly the extremal ray generators of $\sigma_{\G\{ A \}}^{\vee}$. Moreover, one obtains that $\mathfrak{a}=(\mathcal{H}_{A_i} \cap \sigma_G^{\vee})^* = H_{\Val(A)} \cap \sigma_G$.
\end{thm}

\begin{proof}
Let $\mathfrak{a}^* = \mathcal{H}_{A_1} \cap \sigma_G^{\vee}\prec \sigma_G^{\vee}$ be the facet associated to the first independent set $A$. Since the extremal rays of $\sigma_G^{\vee}$ form the Hilbert basis by \hyperref[hilbert]{Proposition \ref*{hilbert}},  the facet $\mathfrak{a}^*$ is generated by the extremal rays of $\sigma_{G'}^{\vee}$, where $G'$ is a subgraph of $G$. By the definition of the supported hyperplane $\mathcal{H}_{A_1}$, the extremal rays of $\sigma_{\G\{A_1\}}^{\vee}$ are in the set of extremal ray generators of $\mathfrak{a}^*$. If $A$ is two-sided, then $\sigma_{\G\{A_2\}}^{\vee}$ is also included in $\mathfrak{a}^*$. These are the only extremal ray generators of $\mathfrak{a}^*$. To show this, we examine the edges in $E(G) \backslash E(G\{A\})$ in two cases: \\

\noindent $\bullet$ If $A=U_1 \backslash \{i\}$ is one-sided, then for $j \in [m ]$, $e^i + f^j \in M$ is not in the generator set of $\mathfrak{a}^*$.  \\
\noindent $\bullet$ If $A=A_1 \sqcup A_2$ is two-sided, then the remaining rays $e^i + f^j$ for $i \in N(A_2)$ and $j \in N(A_1)$ with $\{i,j\} \in E(G)$ are not in the generator set of $\mathfrak{a}^*$.\\

\noindent By construction, $\Val(A) \in \sigma^{\vee}_G \cap M$. We have $\mathfrak{a} = H_{\Val(A)} \cap \sigma_G$ if and only if $\Val(A) \in \Relint(\mathfrak{a}^{*})$. Since we chose $\Val(A) \in \sigma_G^{\vee}$ to be the sum of the generators of the facet $\mathfrak{a}^*$, we obtain that $\Val(A) \in \Relint(\mathfrak{a}^{*})$. 
\end{proof}

\begin{ex}
Let us recall the graph $G$ from \hyperref[tekindependentli]{Example \ref*{tekindependentli}}. The one-sided first independent sets are in form $U_i \backslash \{v\}$ for $i \in \{1,2\}$ with $v \neq 4$. Consider the only two-sided first independent set $A=\{1,2,3,5\}$. The facet $\mathfrak{a}^*$ is generated by the generators of $\sigma^{\vee}_{\G\{A\}}$ where $\G\{A\}$ is the third bipartite graph in \hyperref[benimteo]{Figure \ref*{benimteo}}. We calculate $\Val(A)=[3,3,3,1,1,3,3,3] \in \sigma_G^{\vee}$ and obtain that $\mathfrak{a}=H_{\Val(A)} \cap \sigma_G = e_4 - f_1 \in \sigma_G$.  
\end{ex}

\subsection{Description of the faces of an edge cone} \label{facesection}
In \hyperref[deformationtheory]{Section \ref*{deformationtheory}}, we remark that one needs to study the compact edges and compact 2-faces of the cross-cut $Q(R)$ for the homogeneous piece $T^1_{\TV(G)}(-R)$ of the vector space of first order deformations of the toric variety $\TV(G)$. Therefore we investigate the combinatorial description of the faces of $\sigma_G$ in terms of graphs in this section. We introduce the technique to find the faces of $\sigma_G$ by intersecting the induced subgraphs $G\{A\}$ associated to the first independent sets.
\begin{lem}\label{dimensionlemma}
Let $G\subseteq K_{m,n}$ be a bipartite graph with $k$ connected components. Then $\dim(\sigma_G)=m+n-1$ and $\dim(\sigma_G^{\vee})=m+n-k$.
\end{lem}

\begin{proof}
Recall that by {\hyperref[bubaslangic]{Proposition \ref*{bubaslangic}}}, if $G$ is connected, then the rank of the incidence matrix $A_G$ is $m+n-1$. Suppose that $G$ has $k$ connected components $G_i$. Then the incidence matrix $A_G$ is
\[\begin{bmatrix}
    A_{G_1}       &0 & 0 & \dots & 0 \\
    0           & A_{G_2}& 0 & \dots & 0 \\
   \vdots& 0& \ddots & & 0\\
   \vdots&\vdots&& \ddots& \vdots \\
    0      & 0 & 0& \dots & A_{G_k}
\end{bmatrix}\]
Therefore the rank of $A_G$, i.e.\ dimension of the dual edge cone is $m+n-k$. Furthermore, since $\sigma_G^{\vee}$ contains no linear subspace, the edge cone $\sigma_G\subseteq N_{\QQ}$ is full dimensional and hence $\dim(\sigma_G)=m+n-1$.  
\end{proof}

\begin{thm}\label{facetheorem}
Let $ S \subseteq \mathcal{I}_G^{(1)}$ be a subset of $d$ first independent sets and let $\pi$ be the bijection from \hyperref[11thm]{Theorem \ref*{11thm}}. The extremal ray generators $\pi(S)$ span a face of dimension $d$ if and only if the dimension of the dual edge cone of the spanning subgraph $\G[S]:=\bigcap_{A \in S} \G\{A\}$ is $m+n-d-1$, i.e.\ $\G[S]$ has $d+1$ connected components. In particular, the face is equal to $H_{\Val_S} \cap \sigma_G$ where $\Val_S$ is the degree sequence of the graph $\G[S]$.
\end{thm}

\begin{proof} 

By \hyperref[11thm]{Theorem \ref*{11thm}}, if $A \in S$, then the associated facet $a^* \preceq \sigma_G^{\vee}$ is generated by the extremal ray generators of $\sigma_{\G\{A\}}^{\vee}$. Hence, intersecting these induced subgraphs $\G\{A\}$ is equivalent to intersecting the extremal ray generators of the facets. Every face of $\sigma_G^{\vee}$ is the intersection of the facets it is
contained in. This intersection forms a face $\tau$ of $\sigma_G$ (and therefore a face of $\sigma^{\vee}_G$) since we have:\\
\begin{eqnarray*}
 \tau^* = \bigcap_{\mathfrak{a} \in \pi(S)} \mathfrak{a}^{*} =  \bigcap_{A \in S} (H_{\Val(A)} \cap \sigma_G)^{*} =  (H_{\Val_S} \cap \sigma_G)^*   
 \end{eqnarray*}
where $\Val_S \in \Relint( \sigma^{\vee}_{\G[S]}) \subsetneq \sigma_G^{\vee}$. By \hyperref[dimensionlemma]{Lemma \ref*{dimensionlemma}}, $\di (\sigma_G^{\vee}) = \di (\sigma_G) = m+n-1$. Thus, the dimension of $\tau$ is $d$ if and only if the dimension of $\tau^{*}$ is $m+n-d-1$. Hence, this means that the dimension of the cone $\sigma_{\G[S]}^{\vee}$ is $m+n-d-1$. 
\end{proof}

\begin{cor}\label{tatli}

Let $\tau:=H_{\Val_S}\cap \sigma_G \preceq \sigma_G$ be a face of dimension $d$ which is given by the intersection of subgraphs formed by a subset $S \subsetneq \mathcal{I}_G^{(1)}$ where $|S| \geq d$. If $G[S] \subset G\{A'\}$ for some $A' \in \mathcal{I}_G^{(1)}\backslash S$, then the associated extremal ray generator $\mathfrak{a}'$ is also included in the extremal ray generators of the face $\tau$. 

\end{cor}

\begin{proof}
It follows from \hyperref[facetheorem]{Theorem \ref*{facetheorem}} by dropping the condition of $S$ consisting of $d$ elements and from the fact that every face is an intersection of facets it is contained in.
\end{proof}

\begin{prop}\label{prop2310}
The maximal independent sets of  \hyperref[prop239]{Proposition \ref*{prop239}} form $d$-dimensional faces $\tau \preceq \sigma_G$. Moreover $\tau = (\mathcal{H}_{A_1} \cap \sigma_G^{\vee})^* $. 
\end{prop}

\begin{proof}
Let $C^i$ denote the two-sided maximal independent sets $X_i \sqcup (A_2 \sqcup \bigsqcup_{j \neq i} N(X_j))$ for $i \in [d]$. By \hyperref[facetheorem]{Theorem \ref*{facetheorem}}, the dual edge cone of the intersection subgraph $\bigcap_{i \in [d]} \G\{C^i\}$ is $m+n-d-1$. Furthermore, since $\bigcap_{i \in [d]} \G\{C^i\} = \G\{A\}$, one obtains that $$\langle \mathfrak{c}^1, \ldots , \mathfrak{c}^d \rangle_{\QQ} = ((\mathcal{H}_{C_1^1} \cap \sigma_G^{\vee}) \cap \ldots \cap (\mathcal{H}_{C_1^d} \cap \sigma_G^{\vee}) )^{*} = (\mathcal{H}_{A_1} \cap \sigma_G^{\vee})^{*}.$$\end{proof}

\begin{prop}\label{nays}
Let $A$ be an independent set of $V(G)$.\ Then $\tau = H_{\Val(A)} \cap \sigma_G$ is a $d$-dimensional face of $\sigma_G$ where $m+n-d-1=\dim(\sigma^{\vee}_{\G\{A\}})$.
\end{prop}

\begin{proof}
It follows from \hyperref[villaprop]{Proposition \ref*{villaprop}} and \hyperref[facetheorem]{Theorem \ref*{facetheorem}}.
\end{proof}

\begin{ex}\label{faceexampless}
We examine the two and three-dimensional faces of $\sigma_G$ for $G \subsetneq K_{4,4}$ from \hyperref[notconnected]{Example \ref*{notconnected}}. We use the notation from \hyperref[11thm]{Theorem \ref*{11thm}}. The edge cone $\sigma_G$ is generated by the extremal ray generators $e_1,e_2,e_3,e_4,f_1,f_2,\mathfrak{a}', \mathfrak{a}''$. From the figure in \hyperref[sameexamplefig]{Example \ref*{sameexamplefig}}, we observe that $\G\{A\} = \G\{A'\} \cap \G\{A''\}$, thus $(\mathcal{H}_{A_1} \cap \sigma_G^{\vee})^{*}$ is a two-dimensional face spanned by $\mathfrak{a}'$ and $\mathfrak{a}''$. Furthermore, we see that the intersection of the associated subgraphs $\G\{A'\} \cap \G\{A''\}$ with another associated subgraph to an extremal ray of $\sigma_G$ has four connected components. The only pair of extremal rays which does not span a two-dimensional face of $\sigma_G$ is $\{e_3,e_4\}$. One can infer this in \hyperref[figure32]{Figure \ref*{figure32}} below: The intersection $\G\{U_1 \backslash \{3\}\} \cap \G\{U_1 \backslash \{4\}\}$ has the edge set consisting of only two edges $\{1,8\}$ and $\{2,7\}$. This implies that any triple of extremal ray generators containing $\{e_3,e_4\}$ does not span a three-dimensional face of $\sigma_G$. In particular by \hyperref[nays]{Proposition \ref*{nays}}, for the independent set $\{1,2\}$ we obtain a five-dimensional face of $\sigma_G$, since $\G\{\{1,2\}\}$ has six connected components as seen in the figure. More precisely this five-dimensional face is spanned by $e_3,e_4,f_1,f_2,\mathfrak{a}'$ and $\mathfrak{a}''$, as $\G\{\{1,2\}\}$ contains the associated subgraphs of these extremal rays. Lastly, a computation on the intersection of associated subgraphs shows that any triple not containing both  $e_3$ and $e_4$ spans a three-dimensional face of $\sigma_G$.\\

\begin{minipage}[b]{0.27\linewidth}
\begin{tikzpicture}[baseline=1,scale=1.5,every path/.style={>=latex},every node/.style={draw,circle,fill=white,scale=0.6}]
  \node            (a2) at (0,1.5)  {2};
  \node         (a3) at (0,1)  {3};
  \node            (a4) at (0,0.5)  {4};
\node       [rectangle]     (b8) at (1.5,0.5)  {8};
\node       [rectangle]     (b7) at (1.5,1)  {7}; 
 \node       [rectangle]     (b6) at (1.5,1.5)  {6};
  \node            (a1) at (0,2) {1};
  \node       [rectangle]     (b5) at (1.5,2) {5};
\node [draw =none, fill=none,scale=1.6](e) at (0.75, -0.4) {$G$};

    \draw[-] (a1) edge (b8);

    \draw[-] (a2) edge (b7);

    \draw[-] (a3) edge (b5);
    \draw[-] (a3) edge (b6);
    \draw[-] (a3) edge (b7);
    \draw[-] (a3) edge (b8);

    \draw[-] (a4) edge (b6);
    \draw[-] (a4) edge (b7);
    \draw[-] (a4) edge (b8);
    \draw[-] (a4) edge (b5);
 
\end{tikzpicture}
\end{minipage}
\hspace{0.1cm}
\begin{minipage}[b]{0.24\linewidth}
\begin{tikzpicture}[baseline=1,scale=1.5,every path/.style={>=latex},every node/.style={draw,circle,fill=white,scale=0.6}]
  \node         [fill=yellow]   (a2) at (0,1.5)  {2};
  \node        [draw=cyan]      (a3) at (0,1)  {3};
  \node         [fill=yellow]    (a4) at (0,0.5)  {4};
\node       [rectangle]     (b8) at (1.5,0.5)  {8};
\node       [rectangle]     (b7) at (1.5,1)  {7}; 
 \node       [rectangle]     (b6) at (1.5,1.5)  {6};
  \node      [fill=yellow]   (a1) at (0,2) {1};
  \node        [rectangle]     (b5) at (1.5,2) {5};
\node [draw =none, fill=none,scale=1.6](e) at (0.75, -0.4) {$\G\{U_1 \backslash \{3\}\}$};

    \draw[-] (a1) edge (b8);

    \draw[-] (a2) edge (b7);

    \draw[-] (a4) edge (b6);
    \draw[-] (a4) edge (b7);
    \draw[-] (a4) edge (b5);

  \draw[-] (a4) edge (b8);

\end{tikzpicture}
\end{minipage}
\hspace{0.1cm}
\begin{minipage}[b]{0.24\linewidth}
\begin{tikzpicture}[baseline=1,scale=1.5,every path/.style={>=latex},every node/.style={draw,circle,fill=white,scale=0.6}]
  \node        [fill=yellow]   (a2) at (0,1.5)  {2};
  \node        [fill=yellow]   (a3) at (0,1)  {3};
  \node        [draw=cyan]        (a4) at (0,0.5)  {4};
\node       [rectangle]     (b8) at (1.5,0.5)  {8};
\node       [rectangle]     (b7) at (1.5,1)  {7}; 
 \node       [rectangle]     (b6) at (1.5,1.5)  {6};
  \node      [fill=yellow]      (a1) at (0,2) {1};
  \node       [rectangle]     (b5) at (1.5,2) {5};
\node [draw =none, fill=none,scale=1.6](e) at (0.75, -0.4) {$\G\{U_1 \backslash \{4\}\}$};

    \draw[-] (a1) edge (b8);

    \draw[-] (a2) edge (b7);

    \draw[-] (a3) edge (b6);
    \draw[-] (a3) edge (b7);
    \draw[-] (a3) edge (b5);

  \draw[-] (a3) edge (b8);

\end{tikzpicture}
\end{minipage}
\hspace{0.1cm}
\begin{minipage}[b]{0.24\linewidth}
\begin{tikzpicture}[baseline=1,scale=1.5,every path/.style={>=latex},every node/.style={draw,circle,fill=white,scale=0.6}]
  \node       [fill=yellow]  (a2) at (0,1.5)  {2};
  \node           (a3) at (0,1)  {3};
  \node        (a4) at (0,0.5)  {4};
\node       [rectangle]     (b8) at (1.5,0.5)  {8};
\node       [rectangle]     (b7) at (1.5,1)  {7}; 
 \node       [rectangle]     (b6) at (1.5,1.5)  {6};
  \node        [fill=yellow]   (a1) at (0,2) {1};
  \node       [rectangle]     (b5) at (1.5,2) {5};
\node [draw =none, fill=none,scale=1.6](e) at (0.75, -0.4) {$\G\{\{1,2\}\}$};

    \draw[-] (a1) edge (b8);

    \draw[-] (a2) edge (b7);

\end{tikzpicture}
\end{minipage}


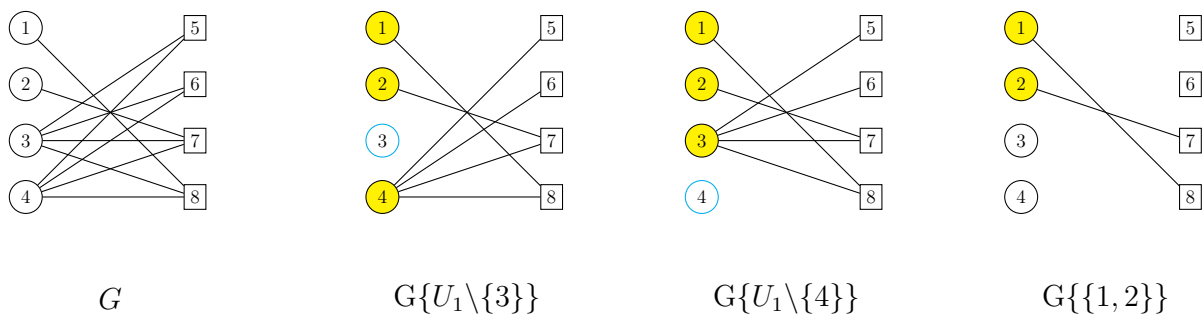
\captionof{figure}{Studying faces of the edge cone via intersecting associated subgraphs to first independent sets.} \label{figure32}
\end{ex}

\section{On the classification of the general case} \label{generalch}

This section starts with preliminaries about deformation theory and in particular the technique which we utilize for affine toric varieties. This motivates us to provide a detailed exposition of two and three-dimensional faces of the edge cone $\sigma_G$ for a connected bipartite graph $G$. Using the tools from \hyperref[facechar]{Section \ref*{facechar}}, we prove that the toric variety $\TV(G)$ is smooth in codimension 2 in \hyperref[smooth2co]{Theorem \ref*{smooth2co}}. Next, we prove that the non-simplicial three-dimensional faces of an edge cone are generated exactly by four extremal ray generators in \hyperref[existenceofadeg]{Theorem \ref*{existenceofadeg}}. Then, we conclude that the toric varieties associated to the edge cones having non-simplicial three-dimensional faces are not rigid in \hyperref[4tuplenonrigid]{Theorem \ref*{4tuplenonrigid}}. Moreover, we characterize the bipartite graphs whose edge cones have only simplicial three-dimensional faces through \hyperref[threefacesection]{Section \ref*{threefacesection}}. In this case, we determine its two and three-dimensional faces and we focus on its non 2-faces pairs and non 3-faces triples. However, in the general setting, it may be challenging to predict the faces of $\sigma_G$. We observe this more detailed in \hyperref[whynot]{Example \ref*{whynot}}.
\subsection{Deformation Theory}\label{deformationtheory}

A deformation of an affine algebraic variety $X_0$ is a flat morphism $\pi \colon \mathcal{X} \longrightarrow S$ with $0 \in S$ such that $\pi^{-1}(0) = X_0$, i.e.\  we have the following commutative diagram.
\begin{center}
\begin{tikzpicture}[every node/.style={midway}]
\matrix[column sep={4em,between origins},
        row sep={2em}] at (0,0)
{ \node(X0)   {$X_0$}  ; & \node(X) {$\mathcal{X}$}; \\
  \node(0) {$0$};&      \node(S) {$S$};             \\};
\draw[->] (X0) -- (0) node[anchor=east]  {$$};
\draw[->] (X) -- (S) node[anchor=west]  {$\pi$};
\draw[right hook->] (X0) -- (X) node[anchor=north]  {$$};
\draw[right hook->] (0)   -- (S) node[anchor=south] {$$};
\end{tikzpicture}
\end{center}

\noindent The variety $\mathcal{X}$ is called the total space and $S$ is called the base space of the deformation. Let $\pi \colon \mathcal{X} \longrightarrow S $ and $\pi' \colon \mathcal{X}' \longrightarrow S$ be two deformations of $X_0$. We say that two deformations are isomorphic if there exists a map $\phi: \mathcal{X} \longrightarrow \mathcal{X}'$ over $S$ inducing the identity on $X_0$. Let $A$ be an Artin ring where $S=\Spec(A)$. For an affine algebraic variety $X_0$, one has a contravariant functor $\Def_{X_0}$ such that $\Def_{X_0} (A)$ is the set of deformations of $X_0$ over $S=\Spec(A)$ modulo isomorphisms. 
\begin{defn}
The map $\pi$ is called a \emph{first order deformation} of $X_0$ if $S=\Spec(\CC [\epsilon]/(\epsilon^2))$. We set $T^1_{X_0} := \Def_{X_0}(\CC [\epsilon] /(\epsilon^2))$.
\end{defn}

\noindent The variety $X_0$ is called rigid if $T_{X_0}^{1} = 0$. This implies that a rigid variety $X_0$ has no nontrivial infinitesimal deformations. This means that every deformation $\pi \in \Def_{X_0}(A)$ over an Artin ring $A$ is trivial i.e.\ isomorphic to the trivial deformation $X_0 \times S \longrightarrow S$. \\

\noindent From now on, let $X_0$ be an affine normal toric variety. We refer to the techniques which are developed in \cite{altmann96} in order to describe the $\CC$-vector space $T_{X_0}^1$. The deformation space $T_{X_0}^1$ is multigraded by the lattice elements of $M$, i.e.\ $T_{X_0}^1 = \bigoplus_{R\in M} T^1_{X_0}(-R)$. We first set some definitions in order to define the homogeneous part $T^1_{X_0}(-R)$. Then, we introduce the formula of $T^1_{X_0}$ for when $X_0$ is smooth in codimension 2.   \\

\noindent Let us call $R\in M$ a deformation degree and let $\sigma\subseteq N_{\QQ}$ be generated by the extremal ray generators $a_1,\ldots,a_n$. We consider the following affine space\[[R=1]:=\{a \in N_\mathbb{Q} \ | \  \langle R, a \rangle =1\} \subseteq N_{\mathbb{Q}}.\] We define the crosscut of $\sigma$ in degree $R$ as the polyhedron $Q(R):=\sigma \cap [R=1]$ in the assigned vector space $[R=0]:=\{a \in N_\mathbb{Q} \ | \  \langle R, a \rangle =0\} \subseteq N_{\mathbb{Q}}$. It has the cone of unbounded directions $Q(R)^{\infty}:= \sigma \cap [R=0]$ and the minimal convex compact part $Q(R)^{c}$ such that $Q(R)^{c} + Q(R)^{\infty} = Q(R)$ holds.  The minimal convex compact part $Q(R)^{c}$ is generated by the vertices $\overline{a_i} := a_i /\langle R, a_i \rangle$ where $\langle R, a_i \rangle \geq 1$. An useful observation is that $\overline{a_i}$ is a lattice vertex in $Q(R)$ if $\langle R, a_i \rangle =1$. Note that in the next definition, there is a choice of an orientation included in regarding edges in $Q(R)$ as elements of $R^{\bot}$.

\begin{defn}
\noindent
\begin{itemize}
\item[(i)] Let $d^1,\hdots,d^N \in R^{\bot} \subset N_{\QQ}$ be the compact edges of $Q(R)$. The vector $\bar{\epsilon} \in \{0, \pm 1\}^{N}$ is called a sign vector assigned to each two-dimensional compact face $\epsilon$ of $Q(R)$ defined as  
\[
    \overline{\epsilon_i}=\left\{
                \begin{array}{ll}
                  \pm 1, \ \text{if} \ d^i \ \text{is an edge of} \ \epsilon\\
                  0\\
                \end{array}
              \right.  \]

\noindent such that $\sum_{i \in [N]} \overline{\epsilon_i} d^i=0$, i.e the oriented edges $\overline{\epsilon_i} d^i$ form a cycle along the edges of the face $\epsilon \preceq Q(R)$. We choose one of both possibilities for the sign of $\overline{\epsilon}$.\\
\item[(ii)] For every deformation degree $R \in M$, the related vector space $V(R)$ is defined as $$V(R)=\{\overline{t}=(t_1,\hdots,t_N)\in \CC^{N} | \sum_{i \in [N]} t_i \overline{\epsilon_i} d^i = 0, \text{ for every compact 2-face }\epsilon \preceq Q(R)\}.$$
\end{itemize}
\end{defn}

\begin{ex}\label{hellodef}
Let us consider the cone over a double pyramid $P$ over a triangle in $N \cong \ZZ^4$ with extremal ray generators $a_1=(0,1,0,1)$, $a_2=(1,0,0,1)$, $a_3=(-1,-1,0,1)$, $a_4=(0,0,1,1)$ and $a_5=(0,0,-1,1)$.  For the deformation degree $R_1=[1,1,1,1] \in M$, we obtain the compact part $Q(R_1)^{c}$ as a two-dimensional face $\epsilon$ generated by $\overline{a_1}$, $\overline{a_2}$, and $\overline{a_4}$. We choose the sign vector $\bar{\epsilon}=(1,1,1)$ to this two-dimensional face and we obtain the elements of $V(R_1)$ as $\bar{t}=(t,t,t)$.  Now let us consider the deformation degree $R_2 =[0,1,0,1] \in M$. Then $Q(R_1)^{c}$ consists of two two-dimensional faces $\epsilon$ and $\epsilon'=\conv(\overline{a_1},\overline{a_2},\overline{a_5})$ where $\overline{\overline{a_1},\overline{a_2}}$ is a common edge of both $\epsilon$ and $\epsilon'$. Up to a fixed labelling of the compact edges, we choose the sign vectors as $\overline{\epsilon} =(1,1,1,0,0)$ and $\overline{\epsilon'} =(0,0,1,1,1)$. Hence we obtain that $V(R_2)$ is a one-dimensional vector space generated by $(1,1,1,1,1)$.
\begin{center}
\begin{tikzpicture}[baseline=1,scale=0.7,every path/.style={>=latex},every node/.style={draw,circle,fill=white,scale=0.7}]
  \node            (a) at (0,0)  {$\overline{a_1}$};
  \node            (b) at (2,4)  {$\overline{a_4}$};
  \node            (c) at (4,0) {$\overline{a_2}$};
    \draw[->] (a) edge (b);
    \draw[<-] (a) edge (c);
    \draw[<-] (c) edge (b);
   
  \node      [draw=none,fill=none,scale=1.5]      (a) at (2,-0.5)  {$t$};
  \node       [draw=none,fill=none,scale=1.5]      (b) at (0.5,2)  {$t$};
  \node        [draw=none,fill=none,scale=1.5]     (c) at (3.5,2) {$t$};
 
\path[draw, fill=green!50] (0.2,0.12)--(2,3.8)--(3.8,0.12)--cycle;

\end{tikzpicture}
\hspace{3cm}
\begin{tikzpicture}[baseline=1,scale=0.4,every path/.style={>=latex},every node/.style={draw,circle,fill=white,scale=0.7}]
  \node            (a) at (0,4)  {$\overline{a_1}$};
  \node            (b) at (2,8)  {$\overline{a_4}$};
    \node            (d) at (2,0)  {$\overline{a_5}$};
  \node            (c) at (4,4) {$\overline{a_2}$};
    \draw[->] (a) edge (b);
    \draw[<-] (a) edge (c);
    \draw[<-] (c) edge (b);
       \draw[<-] (c) edge (d);
          \draw[->] (a) edge (d);
   
  \node      [draw=none,fill=none,scale=1.5]      (a) at (2,3.95)  {$t$};
  \node       [draw=none,fill=none,scale=1.5]      (b) at (0.5,2)  {$t$};
  \node        [draw=none,fill=none,scale=1.5]     (c) at (3.5,2) {$t$};

  \node       [draw=none,fill=none, scale=1.5]      (b) at (0.5,6)  {$t$};
  \node        [draw=none,fill=none,scale=1.5]     (c) at (3.5,6) {$t$};
 
\path[draw, fill=green!50] (0.4,4.32)--(2,7.5)--(3.6,4.32)--cycle;
\path[draw, fill=cyan] (0.4,3.6)--(2,0.6)--(3.6,3.6)--cycle;

\end{tikzpicture}

\end{center}


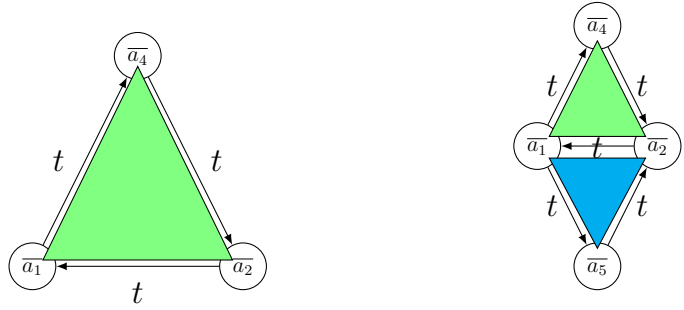
\captionof{figure}{The compact part of the crosscut $Q(R_i)$ and the vector space $V(R_i)$.}

\end{ex}

\begin{thm}[Corollary 2.7, \cite{altmann96}]\label{altmann96}
If the affine normal toric variety $X_0$ is smooth in codimension 2, then $T^1_{X_0}(-R)$ is contained in $V(R)/\mathbb{C}(1,\ldots,1)$. Moreover, it is built by those vectors $\bar{t}$  satisfying $t_{ij} = t_{jk}$ where $\overline{a_j}$ is a non-lattice common vertex in $Q(R)$ of the edges $d^{ij}=\overline{\overline{a_i} \ \overline{a_j}}$ and $d^{jk}=\overline{\overline{a_j} \ \overline{a_k}}$. Thus, $T^1_{X_0}(-R)$ equals the set of equivalence classes of those Minkowski summands of $\RR_{\geq 0}. Q(R)$ that preserve up to homothety the stars of non-lattice vertices of $Q(R)$.
\end{thm}
\noindent Here, a polyhedron $P$ is called a Minkowski summand of $\RR_{\geq 0}.Q(R)$ if there is a $P'$ such that $\RR_{\geq 0} Q(R) = P + P'$, where $P$ and $P'$ have the same cone of unbounded directions. The star of a vertex $\overline{v} \in Q(R)$ is defined as the set of faces having $v$ as a face.  In general, if the toric variety $X_0$ is not smooth in codimension 2, then the homogeneous piece $T^1_{X_0}(-R)$ consists of elements of $V(R) \oplus W(R)/ \mathbb{C}(\underline{1},\underline{1})$ satisfying certain conditions (\cite{altmann96}, Theorem 2.7). Here the vector space $W(R)$ is equal to $\RR^{\# (\text{non-lattice vertices of Q(R)})}$.

\begin{ex}Let us consider the cone $\sigma \subseteq N_{\QQ}$ over $P$ as in \hyperref[hellodef]{Example \ref*{hellodef}}. The two dimensional faces of $\sigma$ are all the pairs of generating rays except $\{a_4,a_5\}$. These are all smooth and hence $\TV(\sigma)$ is smooth in codimension 2. For the deformation degrees $R_1$ and $R_2$, by \hyperref[altmann96]{Theorem \ref*{altmann96}}, we obtain that $T_{\TV(\sigma)}(R_i)=0$ for $i=1,2$. Note that the three dimensional faces of $\sigma$ are all simplicial by construction. In particular one observes that the only cross-cut $Q^{c}(R)$ with $T^1_{\TV(\sigma)(-R)} \neq 0$ consists of compact edges $\overline{\overline a_i,\overline{a_4}}$ and $\overline{\overline{a_i},\overline{a_5}}$ for $i=1,2,3$ and $\overline{a_i}$ is a lattice vertex. Since no such $R\in M$ exists, $\TV(\sigma)$ is a rigid toric variety.
\end{ex}

\begin{rem}\label{garipolay}
Let us consider the cone $\sigma'$ generated by $b_1 = (1,0,0,1)$, $b_2=(1,1,1,-1)$, $b_3=(0,1,0,0)$, $b_4=(0,0,1,0)$ and $b_5=(1,0,0,0)$. It is combinatorially equivalent to $\sigma$ from \hyperref[hellodef]{Example \ref*{hellodef}}, i.e.\ their face lattices are isomorphic. However for the deformation degree $R=[1,0,0,0]$,  $Q(R)^{c}$ consists of compact edges $\overline{\overline{b_1},\overline{b_5}}$ and $\overline{\overline{b_2},\overline{b_5}}$ where $\overline{b_5}$ is a lattice vertex. Hence $T^1_{\TV(\sigma')}(-R) \neq 0$ and $\TV(\sigma')$ is not rigid. Therefore, we emphasize the importance of calculating $Q(R)$ for each deformation degree $R \in M$ than just its combinatorial structure.
\end{rem}

\begin{rem} \label{prooftechnique}
The following two cases in \hyperref[transferringexplained]{Figure \ref*{transferringexplained}} will appear often while we study rigid toric varieties. The first figure was in particular studied in \hyperref[hellodef]{Example \ref*{hellodef}}. For the second figure let $\epsilon^1, \epsilon^2 \preceq Q(R)$ be the compact 2-faces connected by the vertex $\overline{a}$. As in the previous case we obtain that $t_1 = t_2 = t_3$ and $t_4 = t_5 = t_6$. By \hyperref[altmann96]{Theorem \ref*{altmann96}}, if $\overline{a}$ is a non-lattice vertex, then we obtain that $t_3 = t_4$. We note also that there are pairs of vertices of $Q(R)$ where their convex hull is not contained in $Q(R)$. This implies that the corresponding pair of extremal rays do not form a two dimensional face. We call these pairs of extremal rays  \emph{non 2-faces} and we focus on these in the proofs for rigidity.  \\
\begin{center}
\begin{tikzpicture}[baseline=1cm,scale=0.7,every path/.style={>=latex},every node/.style={draw,circle,fill=black,scale=0.2}]
  \node      [draw=black, circle, fill=black, scale=3]      (a) at (0,0)  {};
  \node        [draw=black, circle, fill=black, scale=3]      (b) at (-3,2)  {};
  \node       [draw=black, circle, fill=black, scale=3]        (d) at (3,2) {};
  \node        [draw=black, circle, fill=black, scale=3]       (f) at (0,4) {};

  \node      [draw=none, fill=white, scale=5]      (a1) at (-0.8,1.1)  {$d^1$};
  \node      [draw=none, fill=white, scale=5]      (a4) at (1,1.1)  {$d^5$};
  \node      [draw=none, fill=white, scale=5]      (a7) at (0.35,2)  {$d^3$};
  \node      [draw=none, fill=white, scale=5]      (a8) at (-1,3)  {$d^2$};
  \node      [draw=none, fill=white, scale=5]      (a9) at (1.2,2.7)  {$d^4$};

    \draw[->] (a) edge (b);

 \draw [<-] (f) edge (b);
 \draw [<-] (f) edge (d);
 \draw [->] (f) edge (a);



\draw [->] (a) edge (d);





\end{tikzpicture}
\hspace{2cm}
\begin{tikzpicture}[baseline=1,scale=0.7,every path/.style={>=latex},every node/.style={draw,circle,fill=black,scale=0.2}]
  \node      [draw=black, circle, fill=red, scale=3]      (a) at (0,0)  {$\overline{a}$};
  \node        [draw=black, circle, fill=black, scale=3]      (b) at (-2,2)  {};
  \node         [draw=black, circle, fill=black, scale=3]       (c) at (-3,-1.5) {};
  \node       [draw=black, circle, fill=black, scale=3]        (d) at (2,2) {};
  \node        [draw=black, circle, fill=black, scale=3]      (e) at (3,-1.5) {};

  \node      [draw=none, fill=none, scale=5]       (X) at (-1.5,1)  {$d^1$};
  \node        [draw=none, fill=none, scale=5]       (Y) at (-2.2,0.2)  {$d^2$};
  \node         [draw=none, fill=none, scale=5]      (Z) at (-1.5,-0.28) {$d^3$};
  \node      [draw=none, fill=none, scale=5]       (X1) at (1.3,1) {$d^4$};
  \node       [draw=none, fill=none, scale=5]       (Y1) at (2.15,0.25) {$d^5$};
  \node        [draw=none, fill=none, scale=5]     (Z1) at (1.7,-0.28) {$d^6$};

    \draw[->] (a) edge (b);


\draw [->] (b) edge (c);

\draw [<-] (a) edge (c);
\draw [->] (a) edge (d);

\draw [<-] (a) edge (e);


\draw [->] (d)  edge (e);


\end{tikzpicture}

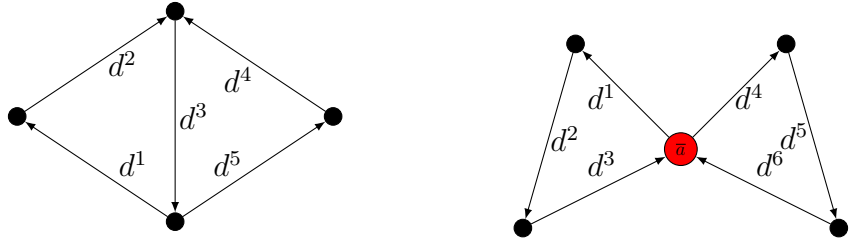
\captionof{figure}{Compact 2-faces sharing an edge or a non-lattice vertex in $Q(R)$}
\label{transferringexplained}
\end{center}

\vspace{0.5cm}

\noindent Moreover we will refer to these two cases by ``$t$ is transferred by an edge or a vertex'' during the investigation of the skeleton of $Q(R)$.
\end{rem}

\subsection{The two-dimensional faces of the edge cone}\label{twofacesection}
We investigate here all possible types of pairs of first independent sets. Our aim is to find necessary and sufficient graph theoretical conditions for the pairs of extremal rays to span a two-dimensional face of $\sigma_G$. We will also use these results to prove that $\TV(G)$ is smooth in codimension 2. We introduce the notation for the tuples of first independent sets forming $d$-dimensional faces analogously to $\Ione$ as in the following definition.

\begin{defn}
A tuple from the first independent set $\mathcal{I}_G^{(1)}$ is said to form a $d$-dimensional face, if their associated tuple of extremal ray generators of $\sigma_G$ under the map $\pi$ of \hyperref[11thm]{Theorem \ref*{11thm}} forms a $d$-dimensional face of $\sigma_G$. We denote the set of these tuples by $\mathcal{I}_G^{(d)}$.  
\end{defn}
\noindent Denote the set of $d$-dimensional faces of $\sigma_G$ as $\sigma_G^{(d)}$. Recall that we have the following one-to-one correspondence by \hyperref[facetheorem]{Theorem \ref*{facetheorem}} as follows:
\begin{eqnarray*}
\mathcal{I}_G^{(d)} &\longrightarrow &\sigma_G^{(d)} \\
(I^1, \ldots, I^{t}) &\mapsto&  (\mathfrak{i}^1,\ldots,\mathfrak{i}^{t})
\end{eqnarray*}
where $t \geq d$ and $\bigcap_{i \in [t]} \G\{I^{i}\}$ has $d+1$ connected components.\\
We label the first independent sets $\Ione$ as in three types: $A=U_1 \backslash \{a\}$, $B=U_2 \backslash \{b\}$ and the two-sided maximal independent set $C= C_1 \sqcup C_2$. 

\subsection*{The pairs of type $(A,A'), (A,B), (A,C)$}

The next proposition follows naturally by \hyperref[dimensionlemma]{Lemma \ref*{dimensionlemma}} and \hyperref[facetheorem]{Theorem \ref*{facetheorem}}.

\begin{prop}\label{aaconditions}

Let $A=U_1 \backslash \{a\}$, $A'=U_1 \backslash \{a'\}$ and $B=U_2 \backslash \{b\}$ be first independent sets.  
\begin{enumerate}
\item[\normalfont{(1)}] $(A,A') \in \Itwo$ if and only if $\G[A \cap A']$ is connected.
\item[\normalfont{(2)}] $(A,B) \in \Itwo$ if and only if $\G[A\sqcup B]$ is connected.
\end{enumerate}

\end{prop}

\begin{prop}\label{acbcconditions}
Let $A=U_1 \backslash \{a\}$ and $C= C_1 \sqcup C_2$ be first independent sets. One obtains that $(A,C) \in \Itwo$ if and only if one of the three following conditions is satisfied.
\begin{enumerate}
\item[\normalfont{(1)}] $A \cap C_1 = \emptyset$ and $C_2 = U_2 \backslash \{\bullet\}$.
\item[\normalfont{(2)}] $C_1 \subsetneq A$ and $\G[C_2 \sqcup (N(C_2) \backslash \{a\})]$ is connected.
\item[\normalfont{(3)}] $N(C_2) \subsetneq A$ and $\G[(C_1 \backslash \{a\}) \sqcup N(C_1)]$ is connected.
\end{enumerate}
\end{prop}

\begin{proof}

Assume $A \cap C_1 = \emptyset$, i.e.\ $C_1 = \{a\}$. Then the graph $\G\{A\} \cap \G\{C\}$ has the isolated vertex set $C_1 \sqcup N(C_1)$. In this case, $(A,C) \in \Itwo$ if and only if $C_2 = U_2 \backslash \{b\}$ for some vertex $b \in U_2$. This implies in particular that $U_2 \backslash \{b\} \notin \Ione$. Now let us consider the case where $A \cap C_1 \neq \emptyset$. Since $A=U_1 \backslash \{a\}$, it is either $C_1 \subsetneq A$ or $N(C_2) \subsetneq A$. We now prove (2), the case $(3)$ follows symmetrically. We require the intersection subgraph $\G[A] \cap \G[C]$ to have three connected components. Since it consists of $\G[C_1] \sqcup (\G[A] \cap \G[C_2])$, and $a$ is an isolated vertex, $\G[C_2 \sqcup (N(C_2) \backslash \{a\})]$ must be connected.  
\end{proof}

\begin{ex}\label{exampleabc}
Let us consider the bipartite graph $G \subset K_{5,4}$ as in \hyperref[figure51]{Figure \ref*{figure51}}. We observe the existence of two two-sided first independent sets $C = \{3\} \sqcup \{6,7\}$ and $C' = \{1,2\} \sqcup \{8,9\}$. Let $A= U_1 \backslash \{4\}$ and $A'= U_1 \backslash \{5\}$. Since $\G[A \cap A']$ has two connected components, $(A,A') \notin \Itwo$. In particular, we obtain that $(A,A',C,C') \in \Ithree$. Since we have that $A \cap C_1 = \emptyset$ and $C_2 = U_2 \backslash \{8,9\}$, $(A,C) \notin \Itwo$. On the other hand $(A',C) \in \Itwo$, since $\{3\} \subset A'$ and the induced subgraph $\G[\{6,7\} \sqcup \{1,2,5\}]$ is connected.

\begin{center}
\begin{tikzpicture}[baseline=1cm,scale=1.5,every path/.style={>=latex},every node/.style={draw,circle,fill=white,scale=0.6}]
  \node            (a2) at (0,1.5)  {2};
  \node            (a3) at (0,1)  {3};
  \node            (a4) at (0,0.5)  {4};
  \node            (a5) at (0,0)  {5};
\node       [rectangle]     (b8) at (1.5,0.5)  {9};
\node       [rectangle]     (b7) at (1.5,1)  {8}; 
 \node       [rectangle]     (b6) at (1.5,1.5)  {7};
  \node            (a1) at (0,2) {1};
  \node       [rectangle]     (b5) at (1.5,2) {6};
\node [draw =none, fill=none,scale=1.6](e) at (0.75, -0.4) {$G$};

    \draw[-] (a1) edge (b5);
  \draw[-] (a1) edge (b6);
  \draw[-] (a2) edge (b5);
  \draw[-] (a2) edge (b6);
  
    \draw[-] (a3) edge (b7);
    \draw[-] (a3) edge (b8);

    \draw[-] (a4) edge (b6);
    \draw[-] (a4) edge (b7);
    \draw[-] (a4) edge (b8);
    \draw[-] (a4) edge (b5);
    \draw[-] (a5) edge (b6);
    \draw[-] (a5) edge (b7);
    \draw[-] (a5) edge (b8);
    \draw[-] (a5) edge (b5);
 
\end{tikzpicture}
\hspace{0.8cm}
\begin{tikzpicture}[baseline=1cm,scale=1.5,every path/.style={>=latex},every node/.style={draw,circle,fill=white,scale=0.6}]
  \node         [draw=cyan]   (a2) at (0,1.5)  {2};
  \node       [fill=yellow]     (a3) at (0,1)  {3};
  \node          [draw=cyan]  (a4) at (0,0.5)  {4};
  \node          [draw=cyan]  (a5) at (0,0)  {5};
\node       [rectangle]     (b8) at (1.5,0.5)  {9};
\node       [rectangle]     (b7) at (1.5,1)  {8}; 
 \node       [fill=yellow,rectangle,draw=cyan]     (b6) at (1.5,1.5)  {7};
  \node         [draw=cyan]   (a1) at (0,2) {1};
  \node       [fill=yellow,rectangle,draw=cyan]     (b5) at (1.5,2) {6};
\node [draw =none, fill=none,scale=1.6](e) at (0.75, -0.4) {$\G\{C\}$};

    \draw[-,cyan] (a1) edge (b5);
  \draw[-,cyan] (a1) edge (b6);
  \draw[-,cyan] (a2) edge (b5);
  \draw[-,cyan] (a2) edge (b6);
  
    \draw[-] (a3) edge (b7);
    \draw[-] (a3) edge (b8);

    \draw[-,cyan] (a4) edge (b6);

    \draw[-,cyan] (a4) edge (b5);
    \draw[-,cyan] (a5) edge (b6);

    \draw[-,cyan] (a5) edge (b5);
 
\end{tikzpicture}
\hspace{0.8cm}
\begin{tikzpicture}[baseline=1cm,scale=1.5,every path/.style={>=latex},every node/.style={draw,circle,fill=white,scale=0.6}]
  \node        [fill=yellow]      (a2) at (0,1.5)  {2};
  \node         [draw=cyan]   (a3) at (0,1)  {3};
  \node      [draw=cyan]      (a4) at (0,0.5)  {4};
  \node     [draw=cyan]      (a5) at (0,0)  {5};
\node       [rectangle,fill=yellow,draw=cyan]     (b8) at (1.5,0.5)  {9};
\node       [rectangle,fill=yellow,draw=cyan]     (b7) at (1.5,1)  {8}; 
 \node       [rectangle]     (b6) at (1.5,1.5)  {7};
  \node    [fill=yellow]        (a1) at (0,2) {1};
  \node       [rectangle]     (b5) at (1.5,2) {6};
\node [draw =none, fill=none,scale=1.6](e) at (0.75, -0.4) {$\G\{C'\}$};

    \draw[-] (a1) edge (b5);
  \draw[-] (a1) edge (b6);
  \draw[-] (a2) edge (b5);
  \draw[-] (a2) edge (b6);
  
    \draw[-,cyan] (a3) edge (b7);
    \draw[-,cyan] (a3) edge (b8);

    \draw[-,cyan] (a4) edge (b8);
    \draw[-,cyan] (a4) edge (b7);

    \draw[-,cyan] (a5) edge (b8);
    \draw[-,cyan] (a5) edge (b7);
 
\end{tikzpicture}
\hspace{0.8cm}
\begin{tikzpicture}[baseline=1cm,scale=1.5,every path/.style={>=latex},every node/.style={draw,circle,fill=white,scale=0.6}]
  \node       [fill=yellow]      (a2) at (0,1.5)  {2};
  \node       [fill=yellow]      (a3) at (0,1)  {3};
  \node            (a4) at (0,0.5)  {4};
  \node            (a5) at (0,0)  {5};
\node       [rectangle]     (b8) at (1.5,0.5)  {9};
\node       [rectangle]     (b7) at (1.5,1)  {8}; 
 \node       [rectangle]     (b6) at (1.5,1.5)  {7};
  \node          [fill=yellow]  (a1) at (0,2) {1};
  \node       [rectangle]     (b5) at (1.5,2) {6};
\node [draw =none, fill=none,scale=1.6](e) at (0.75, -0.4) {$\G\{A\cap A'\}$};

    \draw[-] (a1) edge (b5);
  \draw[-] (a1) edge (b6);
  \draw[-] (a2) edge (b5);
  \draw[-] (a2) edge (b6);
  
    \draw[-] (a3) edge (b7);
    \draw[-] (a3) edge (b8);

\end{tikzpicture}

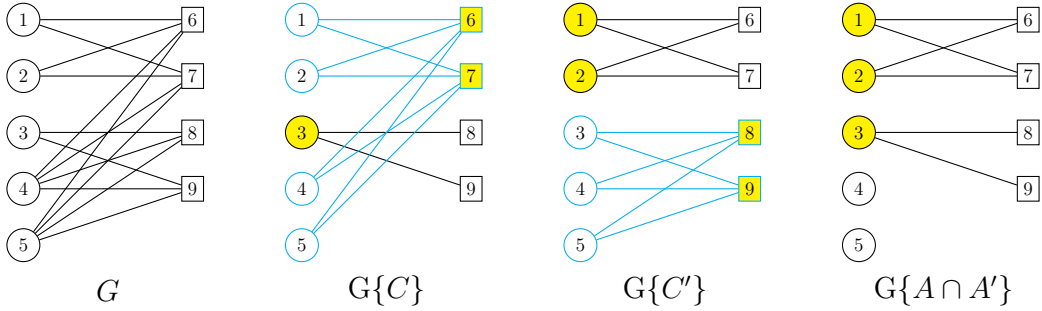
\captionof{figure}{A case where two first independent sets do not form a 2-face of $\sigma_G$.}\label{figure51}
\end{center}

\end{ex}

\subsection*{The pairs of type $(C,C')$}
We would like to consider the possible pairs of two-sided first independent sets, which we denote by $C = C_1 \sqcup C_2$ and $C' = C_1' \sqcup C_2'$.  Suppose that $C_1 \subsetneq C'_1$. Then $C_1 \sqcup C_2  \cup C'_2$ is also a two-sided independent set strictly containing $C$, unless $C'_2 \subsetneq C_2$. By the maximality condition on $C$ and $C'$, it is impossible that $C_1 = C'_1$ or $C_2=C'_2$. By the connectivity assumption on $G$, it is impossible that $C_1 \cup C'_1 = U_1$ and $C_2 \cup C'_2 =U_2$. Consequently, under the conditions where $C_1 \neq C'_1$ or $C_2 \neq C'_2$, and $C \cup C' \neq U_1 \sqcup U_2$, one obtains five types of pairs of $(C,C')$: \\\\
Type $\casei$: $C_1 \subsetneq C'_1$ and $C'_2 \subsetneq C_2$.\\
Type $\caseii$: $C_1 \cap C'_1 = \emptyset$ and $C_2 \cap C'_2=\emptyset$.\\
Type $\caseiii$: $C_1 \cap C'_1 \neq \emptyset$ and $C_2 \cap C'_2 = \emptyset$.\\
Type $\caseiv$: $C_1 \cap C'_1 = \emptyset$ and $C_2 \cap C'_2 \neq \emptyset$.\\
Type $\casev$: $C_1 \cap C'_1 \neq C_1 \neq C'_1 \neq \emptyset$ and $C_2 \cap C'_2 \neq C_2 \neq C'_2 \neq \emptyset$.\\

\noindent We investigate the 2-face conditions by following these types in the following Lemma.

\begin{lem}\label{ccconditions}
Let $C$ and $C'$ be two-sided first independent sets with $C=C_1 \sqcup C_2$ and $C'=C'_1 \sqcup C'_2$. Then $(C,C') \in \mathcal{I}_G^{(2)}$ if and only if it is one of the following types:
\begin{enumerate}
\item[\normalfont{(1)}] $C_1 \subsetneq C'_1$ and $C'_2 \subsetneq C_2$, where $G[(C'_1 \backslash C_1) \sqcup (C_2 \backslash C'_2)]$ is connected.
\item[\normalfont{(2)}] $C_1 \sqcup C'_1 = U_1 \backslash \{\bullet\}$ and $C_2\sqcup C'_2 = U_2$ or $C_2 \sqcup C'_2 = U_2 \backslash \{\bullet\}$ and $C_1 \sqcup C'_1 = U_1$.
\item[\normalfont{(3)}]  $C_1 \cap C'_1 \neq \emptyset$ and $C_2 \cap C'_2 = \emptyset$, where $G[C_1 \cap C'_1]$ is connected with $N(C_1 \cap C'_1)=U_2 \backslash (C_2 \sqcup C'_2)$ and $C_1 \cup C'_1 = U_1.$ 
\item[\normalfont{(4)}] $C_1 \cap C'_1 = \emptyset$ and $C_2 \cap C'_2 \neq \emptyset$, where $G[C_2 \cap C'_2]$ is connected with $N(C_2 \cap C'_2)=U_1 \backslash (C_1 \sqcup C'_1)$ and $C_2 \cup C'_2 = U_2.$ 
\end{enumerate}
\end{lem}

\begin{proof}
The pair $(C,C')$ forms a $2$-face of $\sigma_G$ if and only if $\G\{C\} \cap \G\{C'\}$ has three connected components. We would like to divide the proof into five types which we introduced just before the statement of this Lemma. For the related intersection subgraph $\G\{C\} \cap \G\{C'\}$, we must calculate four intersections:\\\\
$\bf {G_1} $$= \G[(C_1 \cap C'_1 )\sqcup (U_2 \backslash (C_2 \cup C'_2))]$\\
$\bf {G_2} $$= \G[(C_2 \cap C'_2 )\sqcup (U_1 \backslash (C_1 \cup C'_1))]$\\
$\bf {G_3} $$= \G[(C_1 \backslash C'_1) \sqcup (C'_2 \backslash C_2)]$\\
$\bf {G_4} $$= \G[(C'_1 \backslash C_1) \sqcup (C_2 \backslash C'_2)]$\\

\noindent And we have that $\G\{ C\} \cap \G\{C' \} = \bf G_1 \sqcup G_2 \sqcup G_3 \sqcup G_4$.\\

\noindent \underline{Type $\casei$:}  ($C_1 \subsetneq C'_1$ and $C'_2 \subsetneq C_2$). One obtains two connected subgraphs $\bf G_1$$=\G \{C_1\}$ and $\bf G_2 $$= \G \{C'_2 \}$. The graph $\bf G_3$ is empty, since $C_1 \backslash C'_1 = \emptyset$ and $C'_2 \backslash C_2 = \emptyset$. The subgraph $\bf G_4$ is not empty. Assume that $\bf G_4$ has an isolated vertex $u \in C'_1 \backslash C_1$. Then $C_1 \sqcup \{x\} \sqcup C_2$ is an independent set. This contradicts the fact that $C$ is maximal. Similarly, there exists no isolated vertex in $C_2 \backslash C'_2$ of the subgraph $\bf G_4$, otherwise $C'$ is not maximal. However it is possible that $\G[(C'_1 \backslash C_1) \sqcup (C_2 \backslash C'_2)]$ has $k\geq2$ connected components with vertex sets $X_i \subsetneq C'_1 \backslash C_1$ and $Y_i \subsetneq C_2 \backslash C'_2$, for $i \in [k]$. This means in particular that for $I \subsetneq [k]$, there exist first independent sets of form $\mathcal{C}^{I}:=(C_1 \sqcup \bigsqcup_{i \in I} X_i) \sqcup (C_2\backslash \bigsqcup_{i \in I}Y_i)$. We examine this case in \hyperref[magicalcc]{Lemma \ref*{magicalcc}}.\\

\noindent \underline{Type $\caseii$:}  ($C_1 \cap C'_1 = \emptyset$ and $C_2 \cap C'_2=\emptyset$). The subgraphs $\bf G_1$ and $\bf G_2$ are empty. Since $C'_1 \subseteq U_1 \backslash C_1 = N(C_2)$ and $C'_2 \subseteq U_2 \backslash C_2 = N(C_1)$, we obtain that $\bf G_3 $$= \G[C_1 \sqcup C_2']$ and $\bf G_4 $$ =\G[C_2 \sqcup C_1']$. Since we cannot have that $C_1 \sqcup C'_1 \neq U_1$ and $C_2 \sqcup C'_2\neq U_2$, there must exist exactly one isolated vertex $v$ such that $\G\{C\} \cap \G\{C'\} = \bf G_3 \sqcup G_4 \sqcup \{v\}$. For if not, $\G\{C\} \cap \G\{C'\}$ has more than three connected components. Let us suppose for the moment $\{v\} = U_1 \backslash (C_1 \sqcup C'_1)$. Then $\bf G_3 =$$ \G[C_1 \sqcup N(C_1)]$ and $\bf G_4 =$$ \G[C_2\sqcup N(C_2)]$ are connected and therefore $(C,C') \in \mathcal{I}_G^{(2)}$. It follows similarly if $v \in U_2\backslash (C_2 \sqcup C'_2)$. Note that in these cases, $U_i \backslash \{v\} \notin \mathcal{I}_G^{(1)}$.  \\

\noindent \underline{Type $\caseiii$:}($C_1 \cap C'_1 \neq \emptyset$ and $C_2 \cap C'_2 = \emptyset$). The subgraph $\bf G_2$ is empty. Assume that $C_1 \cup C'_1 \neq U_1$. Then the intersection subgraph $\G\{C\} \cap \G\{C'\}$ do not contain $U_1 \backslash (C_1\cup C'_1)$ as a vertex set. This implies that one must have $C_1 \cup C'_1 =U_1$ for otherwise $\G\{C\} \cap \G\{C'\}$ has at least four connected components. The subgraphs $\bf G_3$$ = \G[(C_1\backslash C'_1) \sqcup C'_2]$ and $\bf G_4$$ = \G[(C'_1 \backslash C_1) \sqcup C_2]$ are connected subgraphs.  Consequently, $(C,C') \in \mathcal{I}_G^{(2)}$ if and only if $\G[C_1 \cap C'_1 \sqcup U_2 \backslash (C_2 \sqcup C'_2)]$ is connected and $C_1 \cup C'_1 = U_1$. \underline{Type $\caseiv$} ($C_1 \cap C'_1 = \emptyset$ and $C_2 \cap C'_2 \neq \emptyset$) follows similarly to Type $\caseiii$.\\

\noindent \underline{Type $\casev$:} ($C_1 \cap C'_1 \neq C_1 \neq C'_1 \neq \emptyset$ and $C_2 \cap C'_2 \neq C_2 \neq C'_2 \neq \emptyset$). Assume that $C_1 \cup C'_1 = U_1$. Then $C_2 \cap C'_2 \neq \emptyset$ is an isolated vertex set of $G$. The same holds for the assumption $C_2 \cup C'_2 = U_2$. This means that we must have $C_1 \cup C'_1 \neq U_1$ and $C_2 \cup C'_2 \neq U_2$. But, this implies that $\G \{C\} \cap \G\{C'\}$ has at least four non-empty connected components. 
\end{proof}

\begin{ex}\label{araexample}
Let us consider the first independent sets $C = \{3\} \sqcup \{6,7\}$ and $C'=\{1,2\} \sqcup \{8,9\}$ from \hyperref[exampleabc]{Example \ref*{exampleabc}}. The pair $(C,C')$ is of Type $\caseii$. But we observe that $C_1 \sqcup C'_1 = U_1 \backslash \{4,5\}$. Hence $(C,C') \notin \Itwo$.
\end{ex}

\noindent Now, we utilize the information from \hyperref[ccconditions]{Lemma \ref*{ccconditions}} in order to give a concise proof for the next theorem.

\begin{thm}\label{smooth2co}
Let $G \subseteq K_{m,n}$ be a connected bipartite graph. Then $\TV(G)$ is smooth in codimension 2. 
\end{thm}

\begin{proof}
Recall that $N = \ZZ^{m+n} / \overline{(1,-1)} \cong \ZZ^{m+n-1}$. Let $A,B,C \in  \Ione$ be types of first independent sets as before. The pairs of one-sided first independent sets are the pairs of the canonical basis of $\ZZ^{m+n}$. The extremal rays of $\sigma_G$ associated to two-sided first independent sets are in form of $\mathfrak{c}=\sum_{i \in U_1 \backslash C_1} e_i - \sum_{m+j \in C_2 } f_j \in N$. Consider now the pairs of type $(A,C) \in \Itwo$. Following the conditions from \hyperref[acbcconditions]{Proposition \ref*{acbcconditions}} for any $i \in N(C_2)$ not equal to $a$ and for any $m+j \in N(C_1)$, the set $\{e_1,\ldots,\hat{e_i} ,\ldots,e_m, f_1,\ldots, \hat{f_j}, \ldots, f_{n}\}$ extends the extremal ray $\mathfrak{c}$ to a $\ZZ$-basis of $N$. Note that if $N(C_2)=\{a\}$, then $A \backslash \{a\} \notin \Ione$.

We now consider the pair of two extremal rays $\{\mathfrak{c},\mathfrak{c'}\}$ associated to two-sided first independent sets $C$ and $C'$. By \hyperref[ccconditions]{Lemma \ref*{ccconditions}}, there are four cases we should consider: 
\begin{enumerate}
\item[$(1)$] $\{e_1, \ldots , \hat{e_i}, \ldots, \hat{e_{i'}}, \ldots, e_m, f_1, \ldots, \hat{f_j}, \ldots, f_n\}$ for some $i \in C'_1 \backslash C_1$ and $i' \in N(C'_2)$, and $m+j \in N(C_1)$,
\item[$(2)$] $\{e_1, \ldots , \hat{e_i}, \ldots, e_m, f_1, \ldots, \hat{f_j}, \ldots, \hat{f_{j'}},\ldots,f_n\}$ for $i \in U_1\backslash(C_1 \sqcup C'_1)$ and for some $m+j \in C_2$, and $m+j' \in C'_2$,
\item[$(3)$]  $\{e_1, \ldots , \hat{e_i}, \ldots, \hat{e_{i'}}, \ldots, e_m, f_1, \ldots, \hat{f_j}, \ldots, f_n\}$ for some $i \in C_1 \backslash C'_1$ and $i' \in C'_1 \backslash C_1$, and $m+j \in N(C_1) \cap N(C'_1)$,
\item[$(4)$]  $\{e_1, \ldots , \hat{e_i}, \ldots, e_m, f_1, \ldots, \hat{f_j}, \ldots, \hat{f_{j'}},\ldots, f_n\}$ for some $i \in N(C_2) \cap N(C'_2)$ and $m+j \in C_2 \backslash C'_2$ and $m+j' \in C'_2 \backslash C_2$
\end{enumerate}
extends the pair $\{\mathfrak{c},\mathfrak{c}'\}$ to a $\ZZ$-basis of $N$.
\end{proof}

\noindent In \hyperref[garipolay]{Remark \ref*{garipolay}}, for rigidity of a toric variety $\TV(\sigma)$ we have seen that, it is not sufficient that all 3-faces of $\sigma$ are simplicial. However, from the next proposition we conclude in \hyperref[4tuplenonrigid]{Theorem \ref*{4tuplenonrigid}}, as soon as $\sigma_G$ has a non-simplicial three-dimensional face, $\TV(G)$ is not rigid. 

\begin{prop}\label{nonrigid}
Let $\TV(\sigma)$ be an affine normal toric variety. Assume that $\tau$ is a face of $\sigma$ and $\TV(\tau)$ is not rigid. Then $\TV(\sigma)$ is also not rigid.
\end{prop}

\begin{proof}
Let $m \in \sigma^{\vee}$ and let $\tau = H_m \cap \sigma$ be a face of $\sigma$. Since $\TV(\tau)$ is not rigid, there exists a deformation degree $R \in M$ such that $T_{\TV(\tau)}^1(-R)\neq 0$. Let us set another deformation degree $R'= R- k.m \in M$ for some positive integer $k\gg0$. Since $-m \in R$ evaluates negative on $\sigma \backslash \tau$, we obtain that the compact part of $Q(R')$ consists of the face $\tau$. Therefore $T_{\TV(\sigma)}^1(-R') = T_{\TV(\tau)}^1(-R) \neq 0$. 
\end{proof}
\subsection{The three-dimensional faces of the edge cone}\label{threefacesection}

Since the toric variety $\TV(G)$ is smooth in codimension 2, we can now apply \hyperref[altmann96]{Theorem \ref*{altmann96}} to pursue our investigation on the rigidity of $\TV(G)$. Also with the motivation of \hyperref[nonrigid]{Proposition \ref*{nonrigid}} we first investigate the non-simplicial three-dimensional faces $\tau \preceq\sigma_G$. There exists a pair of extremal ray generators of $\tau$ which does not form a two-dimensional face. Therefore, we treat the pairs of first independent sets which do not form a two-dimensional face and which are contained in the set of extremal ray generators of a three dimensional face. By using \hyperref[tatli]{Corollary \ref*{tatli}} and the 2-face conditions from \hyperref[twofacesection]{Section \ref*{twofacesection}}, we conclude that non-simplicial three-dimensional faces of $\sigma_G$ are generated exactly by four extremal ray generators in \hyperref[existenceofadeg]{Theorem \ref*{existenceofadeg}}. In addition in \hyperref[magicalcc]{Lemma \ref*{magicalcc}} we prove that if the pair of first independent sets of Type (i) does not form a 2-face, then the associated toric variety is not rigid.

\begin{lem}\label{3faceAA}
Let $A = U_1 \backslash \{a\}\in \mathcal{I}_G^{(1)}$ and $A' = U_1 \backslash \{a'\}\in \mathcal{I}_G^{(1)}$. Assume that $\{\mathfrak{a},\mathfrak{a'}\}$ forms part of the extremal ray generators of a three-dimensional face of $\sigma_G$. 
\begin{enumerate}
\item[\normalfont{(1)}] If $(A,A') \in \Itwo$, then the three-dimensional face is simplicial. 
\item[\normalfont{(2)}] If $(A,A') \notin \Itwo$, then either
\begin{enumerate}
\item[\normalfont (i)] $(A,A',B,C) \in \Ithree$, where $B=U_2 \backslash \{b\} \in \Ione$ and $C= (A\cap A') \sqcup \{b\} \in \Ione$ or
\item[\normalfont (ii)] $(A,A',C,C')\in \Ithree$, where $C_1 \sqcup C'_1 = A \cap A'$ and $C_2 \sqcup C'_2 =U_2$. 
\end{enumerate}
\end{enumerate}
\end{lem}

\begin{proof}
$(1)$ The subgraph $G\{A \cap A'\}$ has three connected components.  Let $B=U_2 \backslash \{b\}$. We first investigate the intersection subgraph $\G\{A\} \cap \G\{A'\} \cap\G\{B\}$. By assumption, the dimension of its dual edge cone must be $m+n-4$. Therefore, the intersection subgraph has four connected components with three isolated vertices $a,a'\text{ and } b$. Hence $(A,B), (A',B) \in \Itwo$. The fact that $(A,A',A'') \in \Ithree$ is similarly obtained. Let $C \in \Ione$. We next investigate the intersection subgraph $\G\{A\} \cap \G\{A'\} \cap \G\{C\}$. It has by assumption four connected components with at least two isolated vertices, $a \text{ and } a'$. If $C_1 =\{a\}$ and $C_2=U_2\backslash \{\bullet\}$ then the intersection subgraph has three isolated vertices. In this case by \hyperref[acbcconditions]{Proposition \ref*{acbcconditions}}, $(A,C),(A',C) \in \Itwo$. For the other cases  $\G\{A\} \cap \G\{C\}$ and $\G\{A'\} \cap \G\{C\}$ have three connected components with one isolated vertex, i.e.\ $(A,C), (A',C) \in \Itwo$. Therefore $(A,A',C) \in \Ithree$.\\

\noindent $(2)$ The intersection subgraph $\G\{A\} \cap \G\{A'\}$ has four connected components. Since $a,a' \in U_1$ are isolated vertices of this graph, the proof falls naturally into two parts:\\
\indent (i) $\G\{A\} \cap \G\{A'\}$ has an isolated vertex $b \in U_2$ and $\G[(A\cap A') \sqcup (U_2 \backslash \{b\})]$ is connected. Since $G$, $\G[A \sqcup N(A)]$, and $\G[A' \sqcup N(A')]$ are connected, we obtain the following first independent set $C:=(A \cap A' )\sqcup \{b\} \in \mathcal{I}_G^{(1)}$. Also, since $\G[(A\cap A') \sqcup (U_2 \backslash \{b\})]$ is connected, then $\G[U_2 \sqcup B]$ is connected, i.e.\ $B:=U_2 \backslash \{b\} \in \Ione$. We observe in particular that $(A,B), (A',B), (A,C), (A',C) \in \Itwo$. Hence, we obtain $(A,A',B,C) \in \mathcal{I}_G^{(3)}$. In particular, in the case where $G = K_{2,2}$, the first independent set $C = \{b\}$ and therefore we obtain the edge cone $\sigma_{K_{2,2}}$ as the non-simplicial 3-face.  \\
\indent(ii) $\G[(A\cap A')  \sqcup U_2]$ has two connected components with no isolated vertices. Let us denote the vertex sets as $X_1 \sqcup X_2 = A\cap A'$ and $Y_1, \sqcup Y_2 \subsetneq U_2$ where $G[X_1 \sqcup Y_1]$ and $G[X_2 \sqcup Y_2]$ are connected. Since $\G\{A\}$ and $\G\{A'\}$ have two connected components, there exist edges $\{a,y_1\},\{a,y_2\},\{a',y'_1\},\{a',y'_2\} \in E(G)$ for some vertices $y_1,y'_1 \in Y_1$ and $y_2,y'_2 \in Y_2$. Thus $\mathcal{C}:= X_1 \sqcup Y_2 \in \Ione$ and $\mathcal{C}':=X_2 \sqcup Y_1 \in \Ione$. By \hyperref[ccconditions]{Lemma \ref*{ccconditions}} $(2)$, we know that $(C,C') \notin \Itwo$ and by \hyperref[acbcconditions]{Lemma \ref*{acbcconditions}} $(2)$, we know that $(A,C), (A,C'), (A',C), (A',C') \in \Itwo$. Hence, we obtain that $(A,A',C,C') \in \mathcal{I}_G^{(3)}$.
\end{proof}

\begin{lem}\label{3faceAB}
Let $A = U_1 \backslash \{a\}\in \mathcal{I}_G^{(1)}$ and $B = U_2 \backslash \{b\}\in \mathcal{I}_G^{(1)}$. Assume that $\{\mathfrak{a},\mathfrak{b}\}$ forms part of the extremal generators of a three-dimensional face of $\sigma_G$. 
\begin{enumerate}
\item[\normalfont{(1)}] If $(A,B) \in \Itwo$ then the three-dimensional face is either
\begin{enumerate}
\item[\normalfont{(i)}] the non-simplicial one from \hyperref[3faceAA]{ Lemma \ref*{3faceAA}} $(2)$$\casei$ or
\item[\normalfont{(ii)}] $(A,B,C,C') \in \Ithree$, with $C_1\backslash C'_1 = \{a\}$ and $C'_2 \backslash C_2 = \{b\}$ or $C'_1\backslash C_1 = \{a\}$ and $C_2 \backslash C'_2 = \{b\}$ or
\item[\normalfont{(iii)}] simplicial. 
\end{enumerate}
\item[\normalfont{(2)}] If $(A, B) \notin \Itwo$, then $(A,B,C,C') \in \Ithree$, where $C_1 \sqcup C'_1 = A$ and $C_2 \sqcup C'_2=B$.

\end{enumerate}
\end{lem}

\begin{proof}

$(1)$ The intersection subgraph $\G\{A\} \cap \G\{B\}$ has three connected components with two isolated vertices $a$ and $b$. Analysis similar to the proof of \hyperref[3faceAA]{ Lemma \ref*{3faceAA}} $(1)$ shows that $(A,A',B) \in \Ithree$ and $(A,B,B') \in \Ithree$. We investigate now the intersection $\G\{A\} \cap \G\{B\} \cap \G\{C\}$. If $\{a\} = C_1$ and $b \in N(C_1) $ with $N(C_1) \geq 3 $, then we have that $(A,B,C) \in \Ithree$ unless $\{b\} \sqcup C_1 \backslash \{a\}$ is an independent set. In this case, we obtain a first independent set $C' \in \Ione$ with $C_1 \backslash C'_1 = \{a\}$ and $C'_2 \backslash C'_2 = \{b\}$. If $N(C_1) =2 $, this gives rise to the case $(2)$ $\caseii$ from \hyperref[3faceAA]{ Lemma \ref*{3faceAA}} where $(A,A',B,C) \in \Ithree$. In the other cases similar to proof of \hyperref[3faceAA]{Lemma \ref*{3faceAA}}, we obtain that $(A,B,C) \in \Ithree$. \\

\noindent $(2)$ The intersection $\G\{A\} \cap \G\{B\}$ has of four connected components. This intersection subgraph cannot have four isolated vertices, because this means that we have that $G\subseteq K_{2,2}$. We studied these cases in \hyperref[kucuk]{Example \ref*{kucuk}} and in \hyperref[rigidcomplete]{Theorem \ref*{rigidcomplete}}. Assume that the intersection subgraph has three isolated vertices $\{a,a',b\}$ and one connected component. This means that $a' \cap U_2 \backslash \{b\}$ is an independent set. But this contradicts the fact that $B \in \Ione$. The case with three isolated vertices $\{a,b,b'\}$ is similarly impossible, because $A \in \Ione$. Assume lastly that the intersection has two isolated vertices $\{a,b\}$ and two connected graphs with vertex sets $X_1 \sqcup X_2 = A$ and $Y_1 \sqcup Y_2 = B$. Since $\G[A \sqcup N(A)]$ and $\G[B \sqcup N(B)]$ are connected, we obtain that $C: = X_1 \sqcup Y_2$ and $C':= X_2\sqcup Y_1$ of Type $\caseii$ and $(C,C') \notin \Itwo$. We conclude that $(A,C), (A,C'), (B,C), (B,C') \in \Itwo$ and $(A,B,C,C') \in \Ithree$. 
\end{proof}

\begin{ex}\label{exampleaaab}
Consider the first independent sets $A,A',C$ and $C'$ from \hyperref[exampleabc]{Example \ref*{exampleabc}}. Since $(A,A') \notin \Itwo$ and $\G\{A \cap A'\}$ has four connected components, we observe that $(A,A',C,C') \in \Ithree$ and it is the case from \hyperref[3faceAA]{Lemma \ref*{3faceAA}} (2)(ii). Let $B = U_2 \backslash \{9\} \in \Ione$. Then we observe that $B$ forms a 2-face with every first independent set except $B' = U_2 \backslash \{8\}$. In that case we have the symmetrical case to \hyperref[3faceAA]{Lemma \ref*{3faceAA}} (2)(i), namely $(B,B',A'',C) \in \Ithree$ with $A'' = U_1 \backslash \{3\}$. In particular, we observe that $(A'',B) \in \Itwo$ and $(A'',B') \in \Itwo$.
\end{ex}

\noindent The calculation of an intersection of subgraphs associated to three two-sided independent sets can easily become heavily combinatorial. Therefore, by using \hyperref[ccconditions]{Lemma \ref*{ccconditions}}, we would like to eliminate some cases of these two-sided independent sets resulting in a non-rigid toric variety. This will simplify the calculations for three-dimensional faces in \hyperref[3faceCC]{Lemma \ref*{3faceCC}}.

\begin{lem}\label{magicalcc}
Let $C = C_1 \sqcup C_2 \in \mathcal{I}_G^{(1)}$ and $C' = C'_1 \sqcup C'_2 \in \mathcal{I}_G^{(1)}$. If $(C,C') \notin \Itwo$ is of Type $\casei$, then $\TV(G)$ is not rigid.
\end{lem}

\begin{proof}
Recall that $(C,C')$ of Type $\casei$ means that $C_1 \subsetneq C'_1$ and $C'_2 \subsetneq C_2$. By \hyperref[ccconditions]{Lemma \ref*{ccconditions}} $(1)$, we infer that if $(C,C') \notin \Itwo$, then $\G[(C'_1 \backslash C_1) \sqcup (C_2 \backslash C'_2)]$ has $k\geq2$ connected components without isolated vertices. Denote the vertex sets $X_i \subsetneq C'_1 \backslash C_1$ and $Y_i \subsetneq C_2 \backslash C'_2$, for $i \in [k]$. Since $C \in \Ione$, we know that $\G[C_2 \sqcup N(C_2)]$ is connected. Thus, for each $i \in [k]$, we obtain that $N( Y_i) = X_i \sqcup Z_i$ where $Z_i \subseteq N(C'_2)$. We can use the connectivity argument of $\G[C'_1 \sqcup N(C'_1)]$ symmetrically for each neighborhood vertex set $N(X_i)$. This implies that for a subset $I \subsetneq [k]$, there exist first independent sets of form
\[ \mathcal{C}^{I}:=(C_1 \sqcup \bigsqcup_{i \in I} X_i) \sqcup (C_2\backslash \bigsqcup_{i \in I}Y_i). \]
Now let $i , j \in [k]$ and consider the pair $(\mathcal{C}^{i}, \mathcal{C}^{j}) \notin \Itwo$ of Type $\casev$. We calculate the intersection subgraph $\G \{\mathcal{C}^{i}\} \cap \G \{\mathcal{C}^{j}\}$ as
\[\G[C_1 \sqcup N(C_1)] \sqcup \G[C_2 \backslash (Y_i \sqcup Y_j) \sqcup U_1 \backslash (C_1 \sqcup X_i \sqcup X_j)] \sqcup \G[X_i \sqcup Y_i] \sqcup \G[X_j \sqcup Y_j] \]
and conclude that it has four connected components. This means that $\{\mathfrak{c}^i , \mathfrak{c}^j\}$ is contained in the extremal generator set of a 3-face of $\sigma_G$. By \hyperref[tatli]{Corollary \ref*{tatli}}, we search for first independent sets such that the intersection subgraph $\G \{\mathcal{C}^{i}\} \cap \G \{\mathcal{C}^{j}\}$ is a subgraph of their associated subgraph. We observe that $\G\{C\}$ and $\G\{\mathcal{C}^{i,j}\}$ satisfy this condition. Moreover $(C,\mathcal{C}^i),(C, \mathcal{C}^j), (\mathcal{C}^i ,\mathcal{C}^{i,j}), (\mathcal{C}^{j}, \mathcal{C}^{i,j}) \in \Itwo$ of Type $\casei$. Hence we obtain the non-simplicial 3-face $(C,\mathcal{C}^i, \mathcal{C}^j, \mathcal{C}^{i,j}) \in \Ithree$. Let $\alpha \in N(C'_2)$ and $\beta \in N(C_1)$ be two vertices and let $R = e_{\alpha} + f_{\beta} \in M$ be a deformation degree. Since the associated extremal rays to the tuple  $(C,\mathcal{C}^i, \mathcal{C}^j, \mathcal{C}^{i,j})$ are all lattice vertices in $Q(R)$, by \hyperref[nonrigid]{Proposition \ref*{nonrigid}}, we conclude that $\TV(G)$ is not rigid.
\end{proof}

\begin{prop}\label{3faceCC}
Let $C = C_1 \sqcup C_2 \in \mathcal{I}_G^{(1)}$ and $C' = C'_1 \sqcup C'_2 \in \mathcal{I}_G^{(1)}$. Assume that $(C,C') \notin \Itwo$ and $\{\mathfrak{c}, \mathfrak{c'}\}$ forms part of the extremal generators of a three-dimensional face of $\sigma_G$. 
\begin{enumerate}
\item[\normalfont{(1)}] If $(C,C')$ is of Type $\caseii$, then one obtains the three-dimensional face either from \hyperref[3faceAA]{Lemma \ref*{3faceAA}} $(2)$$\caseii$ or from \hyperref[3faceAB]{Lemma \ref*{3faceAB}} $(2)$.
\item[\normalfont{(2)}] If $(C,C')$ is of Type $\caseiii$, then one obtains either one of the following 
\begin{enumerate}
\item[\normalfont(i)] $(A,C,C',C'') \in \Ithree$, where $C'' = (C_1 \cap C'_1) \sqcup (C_2 \sqcup C'_2) \in \Ione$ and $A = C_1 \cup C'_1 \in \Ione$. 
\item[\normalfont(ii)] $(C,C', \mathcal{C}, \mathcal{C}') \in \Ithree$, where $C_1 \cup C'_1 =U_1$, $\mathcal{C}_1 \sqcup \mathcal{C}'_1= C_1 \cap C'_1$, $\mathcal{C}_2 \cap \mathcal{C}'_2 = C_2 \sqcup C'_2$, and $\mathcal{C}_2 \sqcup \mathcal{C}'_2 = U_2$.  
\item[\normalfont(iii)]  $(B,C,C',C'')\in \Ithree$, where $C_1 \cup C'_1 =U_1$, $C'' = (C_1 \cap C'_1) \sqcup (C_2 \sqcup C'_2) \sqcup \{b\} \in \Ione$, and $B = U_2 \backslash \{b\} \in \Ione$.
\end{enumerate}
\item[\normalfont{(3)}]  If $(C,C')$ is of Type $\casev$, then there exist the first independent sets $\mathcal{C}:= (C_1 \cap C'_1) \sqcup (C_2\cup C'_2) \in \Ione$, $\mathcal{C}' := (C_1 \cup C_1) \sqcup (C_2 \cap C'_2) \in \Ione$ and one obtains that $(C,C', \mathcal{C}, \mathcal{C}') \in \Ithree$. 
\end{enumerate}

\end{prop}

\begin{proof}

By the assumption, the intersection subgraph $\G\{C\} \cap \G\{C'\}$ has four connected components.\\
\noindent $(1)$ The intersection subgraph $\G\{C\} \cap \G\{C'\}$ has the following isolated vertices: 
 \[(N(C_2) \cap N(C'_2))\sqcup (N(C_1) \sqcup N(C'_1)). \]  
The number of isolated vertices can be at most two. If there is exactly one isolated vertex, we concluded in \hyperref[ccconditions]{Lemma \ref*{ccconditions}} that $(C,C') \in \Itwo$. Hence, we conclude that there are two isolated vertices. Assume that $N(C_2) \cap N(C'_2) = \{a,a'\}$ and $C_2 \sqcup C'_2 =U_2$. Since $\G[C_1 \sqcup N(C_1)]$ and $\G[C'_1 \sqcup N(C'_1)]$ are connected, we have that $A, A' \in \Ione$. We observe that $(A,A') \notin \Itwo$ and therefore it is the case that we examined in  \hyperref[3faceAA]{Lemma \ref*{3faceAA}}$(2)$$\caseii$. Assume now that $N(C_2) \cap N(C'_2) = \{a\}$ and $N(C_1) \cap N(C'_1) = \{b\}$. Similarly to the previous investigation, we have that $A,B \in \Ione$ and it is the case that we examined in  \hyperref[3faceAB]{Lemma \ref*{3faceAB}} $(2)$.\\
 
\noindent $(2)$ It is impossible that $C_2 \sqcup C'_2 = U_2$, because then $C_1 \cap C'_1$ is a set of isolated vertices in $G$. We also conclude that $U_1 \backslash (C_1 \cup C'_1)$ has at most one vertex. Assume first that $C_1 \cup C'_1 = U_1$. In the intersection subgraph $\G\{C\} \cap \G\{C'\}$, there cannot be isolated vertices in $C_1 \cap C'_1$, because this implies that these are isolated vertices in $G$. Since $\G[C_2 \sqcup N(C_2)]$ and $\G[C'_2 \sqcup N(C'_2)]$ are connected, there are two possibilities for the subgraph $\G[(C_1 \cap C'_1) \sqcup (N(C_1) \sqcup N(C'_1))]$: \\

\indent $\bullet$ The subgraph $\G[(C_1 \cap C'_1)\sqcup U_2 \backslash (C_2 \sqcup C'_2 \sqcup \{b\})]$ is connected. This implies that there exist first independent sets $C'' := (C_1 \cap C'_1) \sqcup C_2 \sqcup C'_2 \sqcup \{b\}$ and $B= U_2 \backslash \{b\}$.  Moreover $(C,C'') \in \Itwo$ and $(C',C'') \in \Itwo$ are of Type $\casei$ and $(B,C'') \notin \Itwo$. \\

\indent $\bullet$ The subgraph has two connected components and no isolated vertices. Let us denote their vertex sets as $X_i \subsetneq C_1 \cap C'_1$ and $Y_i \subsetneq U_2 \backslash \{C_2 \sqcup C'_2\}$. Then there exist two first independent sets:
\[\mathcal{C} : = X_1 \sqcup C_2 \sqcup C'_2 \sqcup Y_2 \in \Ione\]
\[\mathcal{C}': = Y_1 \sqcup C_2 \sqcup C'_2 \sqcup Y_1 \in \Ione\]
We observe that $(C,\mathcal{C}), (C,\mathcal{C}'), (C',\mathcal{C}), (C', \mathcal{C}') \in \Itwo$ of Type $\casei$. In particular, $(\mathcal{C}, \mathcal{C}') \notin \Itwo$ of Type $\caseiv$. \\

\noindent Assume now that $U_1 \backslash C_1 \cup C'_1 = \{a\}$. Then the subgraphs $\bf {G_1} $, $\bf {G_3} $, and $\bf {G_4} $ must be connected. Moreover, there exist two first independent sets $C'':= (C_1 \cap C'_1) \sqcup C_2 \sqcup C'_2 \in \Ione$ and $A =U_1 \backslash \{a\}$. We observe that $(A,C), (A,C') \in \Itwo$ and the pairs $(C,C''), (C',C'') \in \Itwo$ are of Type $\casei$.    \\

\noindent $(3)$ One cannot have that $C_1 \cup C'_1 = U_1$ or $C_2 \cup C'_2 =U_2$, because otherweise $G$ has isolated vertices. Also, the subgraph $\bf {G_i} $ must be connected for each $i \in [4]$. We thus observe that there exist two first independent sets 
\[\mathcal{C}: = (C_1 \cap C'_1 ) \sqcup (C_2 \cup C'_2) \in \Ione\]
\[\mathcal{C}':= (C_1 \cup C'_1 ) \sqcup (C_2 \cap C'_2) \in \Ione\]
of Type $\casei$. Moreover we have that $(C, \mathcal{C}), (C', \mathcal{C}), (C, \mathcal{C}'), (C', \mathcal{C}') \in \Itwo$, but $(C,C') \notin \Itwo$. 
\end{proof}

\noindent Consider the pair $(B,C) \in \Itwo$ such that $\{\mathfrak{b},\mathfrak{c}\}$ forms part of the extremal ray generators of a three-dimensional face. We covered all possible triples of the forms $(A,B,C)$ and $(B,B,C)$ in \hyperref[3faceAA]{Lemma \ref*{3faceAA}} and \hyperref[3faceAB]{Lemma \ref*{3faceAB}}. For the triples of form $(B,C,C')$, we finished studying the cases where $(C,C') \notin \Itwo$ in \hyperref[3faceCC]{Proposition \ref*{3faceCC}}. We are left with the task of determining the cases where $(C,C') \in \Itwo$.
\begin{lem}
Let $B = U_2 \backslash \{b\}\in \mathcal{I}_G^{(1)}$, $C = C_1 \sqcup C_2 \in \mathcal{I}_G^{(1)}$ and $C' = C'_1 \sqcup C'_2 \in \mathcal{I}_G^{(1)}$. Assume that $(B,C) \in \Itwo$, $(C,C') \in \Itwo$ and $\{ \mathfrak{b}, \mathfrak{c}, \mathfrak{c'}\}$ forms part of the extremal generators of a three-dimensional face of $\sigma_G$. Then the three-dimensional face is either
\begin{enumerate}
\item[\normalfont{(1)}] $(A,B,C,C') \in \Ithree$ from \hyperref[3faceAB]{Lemma \ref*{3faceAB}} $(1) \caseii$ or
\item[\normalfont{(2)}] simplicial. 
\end{enumerate}
\end{lem}

\begin{proof}
Consider the intersection $\G\{C\} \cap \G\{C'\}$. If $(C,C')$ is of Type $\casei$, without loss of generality let us assume that $C'_1 \subsetneq C_1$ and $C_2 \subsetneq C'_2$. For each type of $(C,C') \in \Itwo$, the induced subgraph $\G[C_2 \sqcup N(C_2)]$ is not empty. If $b \in C_2$, then we obtain that $(B,C') \in \Itwo$. For the rest, we divide the proof into the four types of the pair $(C,C') \in \Itwo$: \\ 

\noindent \underline{Type $\casei$}:  Let $b \in C'_2 \backslash C_2$. The triple $(B,C,C') \notin \Itwo$ if and only if $C_1 \backslash C'_1 = \{a\}$. This is the case from \hyperref[3faceAB]{Lemma \ref*{3faceAB}} (1)$\caseii$. Let $b \in N(C'_1)$. We conclude that $\G[C'_1 \sqcup N(C'_1) \backslash \{b\}]$ is connected and therefore $(B,C') \in \Itwo$.\\

\noindent \underline{Type $\caseii$}: Let $b \in C'_2$. Then $\G[C'_2 \backslash \{b\} \sqcup N(C'_2)]$ is connected, since otherwise $B \notin \Ione$. Hence $(B,C') \in \Itwo$. Note that we cannot have that $C_2 \sqcup C'_2 \sqcup \{b\} = U_2$, since otherwise $B \in \Ione$.\\

\noindent \underline{Type $\caseiii$}: Let $b \in C'_2$. Then $\G[(C'_2 \backslash\{b\}) \sqcup (U_1 \backslash C'_1)]$ is connected and therefore $(B,C') \in \Itwo$. If $b \in U_2 \backslash (C_2 \sqcup C'_2)$, we conclude similarly that $(B,C,C') \in \Ithree$. Note that as in the case of Type $\caseii$, $C_2 \sqcup C'_2 \sqcup \{b\} \neq U_2$.\\  

\noindent \underline{Type $\caseiv$}: Let $b \in C'_2 \backslash C_2$. Since $B \in \Ione$, the induced subgraph $\G[C'_2\backslash \{b\} \sqcup N(C'_2)]$ is connected. Hence $(B,C') \in \Itwo$. 
\end{proof}

\begin{cor}\label{3faceAC}
Let $B = U_2 \backslash \{b\}\in \mathcal{I}_G^{(1)}$ and $C = C_1 \sqcup C_2 \in \mathcal{I}_G^{(1)}$. Assume that $(B,C) \notin \Itwo$ and $\{ \mathfrak{b}, \mathfrak{c}\}$ forms part of the extremal generators of a three-dimensional face of $\sigma_G$. Then one obtains the non-simplicial three-dimensional face in \hyperref[3faceAA]{Lemma \ref*{3faceAA}} $(2)$$\casei$ or in \hyperref[3faceAB]{Lemma \ref*{3faceAB}} or in \hyperref[3faceCC]{Proposition \ref*{3faceCC}} $(2)$$\casei$ and $\caseiii$.

\end{cor}

\begin{proof}
We only need to show that there exists no three-dimensional face containing the extremal rays $\{\mathfrak{b}, \mathfrak{c}, \mathfrak{c'}\}$ where $(C,C') \in \Itwo$ and $(B,C') \notin \Itwo$. Consider the intersection $\G\{B\} \cap \G\{C\}$ which has four connected components. Since we want to have another $\mathfrak{c}'$ in the generator set, we have two possibilities: \\
$\bullet$ If $b \in C_2$, there exist two first independent sets $\mathcal{C}^1$ and $\mathcal{C}^2$ such that $C_1 \cup \mathcal{C}_1^1 \cup \mathcal{C}_1^2 = U_1$ and $\mathcal{C}_2^1 \sqcup \mathcal{C}_2^2 \sqcup \{b\} = C_2$. \\
$\bullet$ If $b \in N(C_1)$, there exist two first independent sets $\mathcal{C}^1$ and $\mathcal{C}^2$ such that $\mathcal{C}_1^1 \sqcup \mathcal{C}_1^2 = C_1$ and $C_2 \cup \mathcal{C}_2^1 \cup \mathcal{C}_2^2 \sqcup \{b\} = N(C_1)$. \\
However, these have been examined in  \hyperref[3faceCC]{Proposition \ref*{3faceCC}} $(2)$$\casei$ and $\caseiii$.
\end{proof}

\begin{ex}
Consider the first independent sets $B = U_2 \backslash \{9\}$ and $C' = \{1,2\} \sqcup \{8,9\}$ from the bipartite graph in \hyperref[exampleaaab]{Example \ref*{exampleabc}}. We observe that $(B,C') \in \Itwo$ and if $\{\mathfrak{b}, \mathfrak{c'}\}$ forms part of the extremal generators of a three dimensional face, then this 3-face is simplicial by \hyperref[3faceCC]{Proposition \ref*{3faceCC}}. Moreover as studied in \hyperref[araexample]{Example \ref*{araexample}}, $(C,C') \notin \Itwo$ is of Type (ii) and we obtain the case from \hyperref[3faceCC]{Proposition \ref*{3faceCC}} (1) or equivalently from \hyperref[3faceAA]{Lemma \ref*{3faceAA}} $(2)$$\caseii$ for the 3-faces containing $\{\mathfrak{c}, \mathfrak{c'}\}$.
\end{ex}

\noindent Finally, we are left with characterizing the triples $(C,C',C'') \in \Ithree$. The next result, follows by the recent calculations and \hyperref[ccconditions]{Lemma \ref*{ccconditions}}.
\begin{cor}
Let $C$, $C'$ and $C''$ be three first independent sets of $G$. Assume that $(C,C',C'') \in \Ithree$ forms a three-dimensional face of $\sigma_G$. Then its two-dimensional faces are one of the following type:
\begin{enumerate}
\item[$\bullet$] $(\casei, x, x)$, $x\in \{\casei,\caseii,\caseiii,\caseiv\}$.
\item[$\bullet$] $(\casei, \caseii, \caseiii)$, $(\casei, \caseii, \caseiv)$, $(\casei, \caseiii, \caseiv)$, $(\caseiii, \caseiii, \caseiii)$, $(\caseiv, \caseiv, \caseiv)$.
\end{enumerate} 
\end{cor}

\subsection{Non-rigidity for toric varieties with non-simplicial three-dimensional faces}
Using the compilations from \hyperref[threefacesection]{Section \ref*{threefacesection}}, we characterize the non-simplicial three-dimensional faces of $\sigma_G$ and show that in this case $\TV(G)$ is not rigid. After that, we are reduced to proving the rigidity for the toric varieties whose edge cone $\sigma_G$ admits only simplicial three-dimensional faces. We classified this type of edge cones explicitly in \hyperref[threefacesection]{Section \ref*{threefacesection}}.
 
\begin{thm}\label{existenceofadeg}
Let $G \subseteq K_{m,n}$ be a connected bipartite graph and let $\tau \preceq \sigma_G$ be a three-dimensional non-simplicial face of the edge cone $\sigma_G$. Then $\tau$ is spanned by four extremal rays. 
\end{thm}

\begin{proof}
It follows by \hyperref[3faceAA]{ Lemma \ref*{3faceAA}} $(2)$$\casei, \caseii$, \hyperref[3faceAB]{ Lemma \ref*{3faceAB}} $(1) \caseii$ and $(2)$, \hyperref[magicalcc]{ Lemma \ref*{magicalcc}}, and \hyperref[3faceCC]{Proposition \ref*{3faceCC}}.
\end{proof}

\begin{thm}\label{4tuplenonrigid}
Let $G \subseteq K_{m,n}$ be a connected bipartite graph. Assume that the edge cone $\sigma_G$ admits a three-dimensional non-simplicial face. Then $\TV(G)$ is not rigid.
\end{thm}

\begin{proof}
Let $\G[S] \subsetneq G$ be the intersection subgraph associated to the non-simplicial three-dimensional face $\tau$ generated by $\pi(S)$ as defined in  \hyperref[facetheorem]{ Lemma \ref*{facetheorem}}. We have also proven that $\tau = H_{\Val_S} \cap \sigma_G$, where $\Val_{S}$ is the degree sequence of the graph $\G[S]$. Since $\Val_{S} \in \sigma^{\vee}_G$, the lattice point $-\Val_S$ evaluates negative on every extremal ray except the extremal ray generators of $\tau$. Hence, we consider the deformation degree $R'= R + k (-\Val_S) \in M$ for $k\gg0$ where $[R,\tau] \geq 1$. Thus, the compact part of the crosscut $Q(R')$ consists of $\tau$. We are now reduced to examine the non-simplicial 3-faces from  \hyperref[3faceAA]{ Lemma \ref*{3faceAA}} $(2)$$\casei, \caseii$, \hyperref[3faceAB]{ Lemma \ref*{3faceAB}} $(2)$, and \hyperref[3faceCC]{Proposition \ref*{3faceCC}}. For each case, by \hyperref[nonrigid]{Proposition \ref*{nonrigid}}, it is sufficient to show that there exists a deformation degree $R \in M$ such that the extremal rays $\pi(S)$ are lattice vertices in $R$. We find such deformation degrees as following: \\
\noindent \underline{\hyperref[3faceAA]{ Lemma \ref*{3faceAA}}}\\
$(2) \casei$: $e^a + e^{a'} + f^{b} + f^{b'}$, where $b\neq b'$.\\
$(2) \caseii$: $e^a + e^{a'} + f^{b} + f^{b'}$, where $b \in C_2$ and $b' \in C'_2$.\\ 
\noindent \underline{\hyperref[3faceAB]{ Lemma \ref*{3faceAB}}}\\
$(1) \caseii$:  $e^a + e^{a'} + f^{b} + f^{b'}$, where $b \in N(C_2)$ and $b' \in N(C'_1)$. \\
$(2)$: $e^a + f^b$.\\
\noindent \underline{\hyperref[3faceCC]{ Proposition \ref*{3faceCC}}}\\
$(2) \casei$: $e^a +f^b$, where $b \in U_2 \backslash (C_2 \sqcup C'_2)$.\\
$(2) \caseii$: $e^a + e^{a'} +f^b + f^{b'}$, where $a \in N(C'_2)$, $a' \in C'_1$, $b \in \mathcal{C}_2 \backslash (C_2 \cup C'_2)$, and $b' \in \mathcal{C}' \backslash (C_1 \cup C'_1)$. \\
$(2) \caseiii$: $e^a + e^{a'} +f^b + f^{b'}$, where $a \in N(C'_2)$, $a' \in N(C_2)$, and $b' \in U_2 \backslash C''$. \\
$(3)$: $e^a  +f^b$, where $a \in N(\mathcal{C}'_2)$ and $b \in N(\mathcal{C}_1)$.
\end{proof}

\subsection{Examples of pairs of first independent sets not spanning a two-dimensional face}

As noted in \hyperref[prooftechnique]{Remark \ref*{prooftechnique}}, it is crucial to examine the non 2-faces of $\sigma_G$ while investigating the rigidity of $\TV(G)$. In the next example, we explain the challenge about using this technique for a general bipartite graph $G$.

\begin{ex}\label{whynot}
Let $G \subsetneq K_{m,n}$ be a connected bipartite graph and let $A=U_1 \backslash \{a\}$ and $A' = U_1 \backslash \{a'\}$ be two first independent sets. Assume that $(A,A') \notin \Ione$ and the edge cone $\sigma_G$ does not have any non-simplicial three-dimensional face. By \hyperref[aaconditions]{Proposition \ref*{aaconditions}} and \hyperref[3faceAA]{Lemma~\ref*{3faceAA}}, the induced subgraph $\G[(U_1 \backslash \{a,a'\}) \sqcup U_2 ]$ has $k$ connected components where $k \geq 3$. If this induced subgraph has isolated vertices, say the set $Y \subsetneq U_2$ as in the first figure, then we obtain the maximal independent set $(A \cap A' ) \sqcup Y$. This maximal independent set is not a first independent set, unless $\G[A\cap A' \sqcup (U_2 \backslash Y)]$ is connected. However, even if this induced subgraph is connected, there might exist another first independent set, say $C \in \Ione$ with $C_1 \subsetneq A\cap A'$ and $C_2 \subsetneq U_2 \backslash Y$. This possibility makes the investigation iterative and hard to control. 

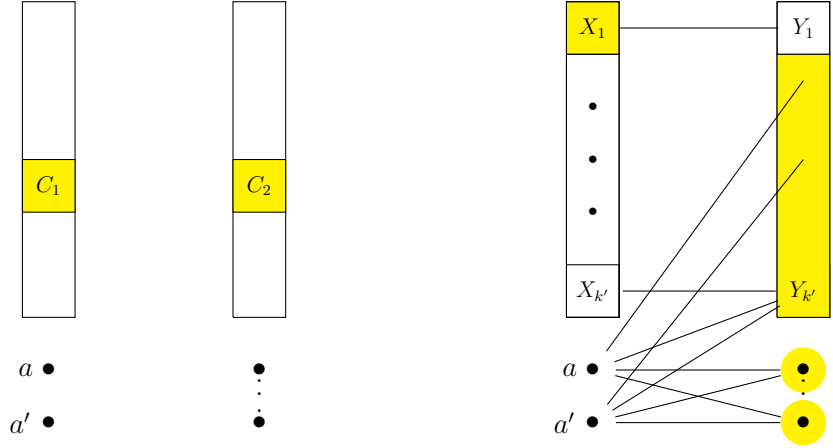
\begin{center}
\begin{minipage}[b]{0.27\linewidth}
\begin{tikzpicture}[baseline=1cm,scale=0.7,every path/.style={>=latex},every node/.style={draw,circle,fill=black,scale=0.6}]
\draw (0,2) -- (0,8) -- (1,8) -- (1,2) -- (0,2) ;
\filldraw[fill=yellow] (0,5) -- (1,5) -- (1,4) -- (0,4) -- (0,5) ;

\node       [fill=none,scale =1.2,draw=none]    (x1) at (0.5,4.5)  {$C_1$}; 
\node       [left,draw=none,circle,fill=none,scale =1.5]    (a) at (0.5,1)  {$a$}; 
\node      [left,draw=none,circle,fill=none,scale =1.5]     (b) at (0.5,0)  {$a'$}; 
\node       [draw=none,circle,fill=none,scale =1.5]    (a1) at (0.5,1)  {$\bullet$}; 
\node      [draw=none,circle,fill=none,scale =1.5]     (b1) at (0.5,0)  {$\bullet$}; 

\draw (4,2) -- (4,8) -- (5,8) -- (5,2) -- (4,2) ;
\filldraw[fill=yellow] (4,5) -- (5,5) -- (5,4) -- (4,4) -- (4,5) ;

\node       [fill=none,scale =1.2,draw=none]    (y1) at (4.5,4.5)  {$C_2$}; 
\node       [draw=none,fill=none,scale=1.5]      (bul1) at (4.5,1)  {$\bullet$}; 
\node       [draw=none,fill=none,scale=1.5]      (bul2) at (4.5,0.8)  {.}; 
\node       [draw=none,fill=none,scale=1.5]      (bul3) at (4.5,0.55)  {.}; 
\node       [draw=none,fill=none,scale=1.5]      (bul4) at (4.5,0.23)  {.}; 
\node       [draw=none,fill=none,scale=1.5]      (bul5) at (4.5,0)  {$\bullet$}; 
\end{tikzpicture}
\end{minipage}
\hspace{2.5cm}
\begin{minipage}[b]{0.27\linewidth}
\begin{tikzpicture}[baseline=1cm,scale=0.7,every path/.style={>=latex},every node/.style={draw,circle,fill=black,scale=0.6}]
\filldraw[fill=yellow] (0,8) -- (1,8) -- (1,7) -- (0,7) -- (0,8) ;
\draw (0,3) -- (1,3) -- (1,2) -- (0,2) -- (0,2) ;
\draw (0,2) -- (0,8) -- (1,8) -- (1,2) -- (0,2) ;
\draw (0,2) -- (0,7) -- (1,7) -- (1,2) -- (0,2) ;

\node       [fill=none,scale =1.2,draw=none]    (x1) at (0.5,7.5)  {$X_1$}; 
\node       [fill=none,scale =1.2,draw=none]    (x2) at (0.5,2.5)  {$X_{k'}$}; 
\node      [left,draw=none,fill=none,scale =1.5]   (a1) at (0.5,1)  {$a$}; 
\node      [left,draw=none,fill=none,scale =1.5]     (b1) at (0.5,0)  {$a'$}; 
\node      [draw=none,circle,fill=none,scale =1.5]   (a) at (0.5,1)  {$\bullet$}; 
\node      [draw=none,circle,fill=none,scale =1.5]     (b) at (0.5,0)  {$\bullet$}; 
\node       [draw=none,fill=none]     (p1) at (0.5,4)  {$\bullet$}; 
\node       [draw=none,fill=none]     (p2) at (0.5,5)  {$\bullet$}; 
\node       [draw=none,fill=none]     (p3) at (0.5,6)  {$\bullet$}; 

\draw (4,8) -- (5,8) -- (5,7) -- (4,7) -- (4,8) ;
\filldraw [draw=black] (4,3) -- (5,3) -- (5,2) -- (4,2) -- (4,2) ;
\draw (4,2) -- (4,8) -- (5,8) -- (5,2) -- (4,2) ;
\filldraw[fill=yellow] (4,2) -- (4,7) -- (5,7) -- (5,2) -- (4,2) ;

\node       [fill=none,scale =1.2,draw=none]    (y1) at (4.5,7.5)  {$Y_1$}; 
\node       [fill=none,scale =1.2,draw=none]    (y2) at (4.5,2.5)  {$Y_{k'}$}; 
\node       [draw=none,fill=yellow,scale=1.5]      (bul1) at (4.5,1)  {$\bullet$}; 
\node       [draw=none,fill=none,scale=1.5]      (bul2) at (4.5,0.8)  {.}; 
\node       [draw=none,fill=none,scale=1.5]      (bul3) at (4.5,0.55)  {.}; 
\node       [draw=none,fill=none,scale=1.5]      (bul4) at (4.5,0.23)  {.}; 
\node       [draw=none,fill=yellow, scale=1.5]      (bul5) at (4.5,0)  {$\bullet$}; 

  \draw[-] (x1) edge (y1);
 \draw[-] (x2) edge (y2);
  \draw[-] (a) edge (4.5,6.5);
  \draw[-] (a) edge (y2);
  \draw[-] (b) edge (4.5,5);
  \draw[-] (b) edge (y2);
      \draw[-] (a) edge (bul1);
  \draw[-] (a) edge (bul5);
    \draw[-] (b) edge (bul1);
        \draw[-] (b) edge (bul5);

\end{tikzpicture}
\end{minipage}
\end{center}
\vspace{0.2cm}
\noindent Another possibility is that $\G[A\cap A' \sqcup (U_2 \backslash Y)]$ has more than 2 connected components. This means that there exist disjoint vertex sets $X_i \subsetneq A\cap A'$ and $Y_i \subsetneq U_2 \backslash Y$ where $\G[X_i \sqcup Y_i]$ is connected as illustrated in the second figure. Since $\G\{A\}$ and $\G\{A'\}$ have two connected components, we obtain the first independent sets $C^i := X_i \sqcup (U_2 \backslash Y_i)$. A pair $(C^i, C^j)$ is of Type $\caseiv$ and does not form a 2-face. Let $R = e^a + e^{a'} - e^{x_i} - e^{x_j} \in M$ where $x_i \in X_i$ and $x_j \in X_j$ and we consider the crosscut $Q(R)$. Although $\G[X_i \sqcup Y_i]$ is  connected, as in the previous situation there might exists an independent set $D$ with $D^i_1 \subsetneq X_i$, $D^i_2 \subsetneq Y_i$ and $x_i \in D^i_1$. Moreover there might exist first independents set $\mathcal{D}^i := D^{i}_1 \sqcup D^{i}_2 \sqcup C_2^i$. We observe that $(\mathcal{D}^i,C^i)$ of Type $\casei$ forms a 2-face, otherwise by \hyperref[magicalcc]{Lemma \ref*{magicalcc}}, $\sigma_G$ has non-simplicial three-dimensional faces. However $(\mathcal{D}^i,C^j)$ is of Type $\caseiv$ and does not form a 2-face. Furthermore, there cannot exist any first independent set containing both $X_i$ and $Y_i$. Hence we obtain that $T^1(-R) \neq 0$ for this possibility. However, for rigidity, one needs to examine all non 2-face pairs, e.g.\ $(\mathcal{D}^i,C^j)\notin \Itwo$.     \\
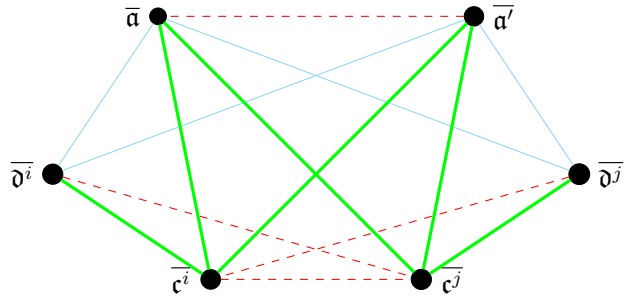
\begin{center}
\begin{tikzpicture}[baseline=1cm,scale=0.7,every path/.style={>=latex},every node/.style={draw,circle,fill=black,scale=0.3}]
  \node      [left,draw=none, circle, fill=none, scale=3]     (d1) at (-5,-1)  {$\overline{\mathfrak{d}^i}$};
    \node      [right,draw=none, circle, fill=none, scale=3]      (d2) at (5,-1)  {$\overline{\mathfrak{d}^j}$};
      \node        [draw=black, circle, scale=1]    (d1) at (-5,-1)  {$\overline{\mathfrak{d}^i}$};
    \node       [draw=black, circle, scale=1]      (d2) at (5,-1)  {$\overline{\mathfrak{d}^j}$};
  \node        [left,draw=none, circle, fill=none, scale=3]      (b) at (-3,2)  {$\overline{\mathfrak{a}}$};
    \node        [draw=black, circle, scale=1]      (b) at (-3,2)  {$\overline{\mathfrak{a}}$};
  \node       [right,draw=none, circle, fill=none, scale=3]        (d) at (3,2) {$\overline{\mathfrak{a}'}$};
  \node     [draw=black, circle, scale=1]          (d) at (3,2) {$\overline{\mathfrak{a}'}$};
  \node      [left,draw=none, circle, fill=none, scale=3]       (f) at (-2,-3) {$\overline{\mathfrak{c}^i}$};
  \node       [right,draw=none, circle, fill=none, scale=3]      (g) at (2,-3) {$\overline{\mathfrak{c}^j}$};
  \node        [draw=black, circle, scale=1]      (f) at (-2,-3) {$\overline{\mathfrak{c}^i}$};
  \node         [draw=black, circle, scale=1]      (g) at (2,-3) {$\overline{\mathfrak{c}^j}$};

     \draw[-,cyan!40] (d1) edge (b);
         \draw[-,very thick, green] (d1) edge (f);
     \draw[-,cyan!40] (d1) edge (d);
      \draw[-,cyan!40] (d2) edge (b);
      \draw[-,cyan!40] (d2) edge (d);
           \draw[-,very thick, green] (d2) edge (g);
                      
    \draw[dashed,red] (d) edge (b);
    \draw[dashed,red] (d1) edge (g);
    \draw[dashed,red] (d2) edge (f);

    \draw[-,very thick, green] (b) edge (g);
    \draw[-,very thick, green] (d) edge (g);
    \draw[dashed,red] (f) edge (g);
  \draw [-,very thick, green] (f) edge (b);
 \draw [-,very thick, green] (f) edge (d);

\end{tikzpicture}
\vspace{0.3cm}
\end{center}

\end{ex}

\vspace{1cm}

\noindent We observe that as long as we know more information about the bipartite graph $G$, it is more probable that we are able to determine the rigidity of $\TV(G)$. In this manner, we study the edge cones associated to so-called toric matrix Schubert varieties in \cite{portakal2}. After examining their face structure, we are able to classify the rigid toric matrix Schubert varieties.

\section {Rigidity of complete bipartite graphs with multiple edge removals}

In this section, we would like to apply the results from \hyperref[generalch]{Section \ref*{generalch}} to the complete bipartite graphs with multiple edge removals. In particular, this section generalizes the results in \cite{herzog2015}, where the complete bipartite graphs with one edge removal were studied. First we prove in \hyperref[rigidcomplete]{Theorem \ref*{rigidcomplete}} that
$\TV(K_{m,n})$ is rigid except for $m=n = 2$.  Next, we consider two vertex sets $C_1 \subsetneq U_1$ and $C_2 \subsetneq U_2$ of the complete bipartite graph $K_{m,n}$ and we remove all the edges between these two sets. This means that we obtain a two-sided first independent set $C:=C_1 \sqcup C_2 \in \Ione$. Without loss of generality, we assume that $C_1= \{1,\ldots,t_1\}$ and $C_2=\{m+1,\ldots,m+t_2\}$ and therefore $\pi(C)=\mathfrak{c} = \sum_{i>t_1} e_i - \sum_{j\leq t_2} f_j$ under the map from \hyperref[11thm]{Theorem \ref*{11thm}}. In \hyperref[onetwosided]{Theorem \ref*{onetwosided}}, we prove that $\TV(G)$ is rigid except the cases where $|C_1|=1$ and $|C_2| = n-2$ or $|C_1|=m-2$ and $|C_2|=1$.
\subsection{Complete bipartite graphs}
Let us first study the case with no edge removals i.e.\ the determinantal singularity $\TV(K_{m,n})$.  The toric variety $\TV(K_{m,n})$ is the affine cone over a Segre variety which is the image of the embedding $\mathbb{P}^{m-1} \times \mathbb{P}^{n-1} \longrightarrow \mathbb{P}^{mn-1}$. It is a famous result by Thom, Grauert-Kerner and Schlessinger as in \cite{segreler} that it is rigid whenever $m\geq 2$ and $n \geq 3$. Note that $\TV(K_{m,n})$ is $\mathbb{Q}$-Gorenstein and $\QQ$-factorial in codimension 3 for $m=n$. By Corollary 6.5.1 in \cite{minkowskialtmann}, it follows that $\TV(K_{m,m})$ is rigid. We prove this classical result combinatorially with graphs also for $m\neq n$. Note that if $m=1$ or $n=1$ then $K_{m,n}$ is a tree and hence $\TV(K_{m,n})$ is smooth and rigid. Therefore we consider the cases with $m\geq2$ and $n\geq 2$.\\

\noindent First, we collect some facts about the faces of the edge cone $\sigma_{K_{m,n}}$. 
\begin{prop}
The edge cone $\sigma_{K_{m,n}} \subset N_{\QQ}$ is generated by the extremal ray generators $e_1, \ldots, e_m, f_1, \ldots, f_n $. 
\end{prop}

\begin{proof}
The complete bipartite graph has no edge removals, therefore it has no two-sided first independent set. The associated subgraph $\G \{U_i \backslash \{u\}\}$ has two connected components for each $u\in U_i$ and $i=1,2$. 
\end{proof}

\begin{ex}\label{smallcompletes}
Let us study the small examples $K_{2,2}$, $K_{2,3}$, and $K_{3,3}$ which will be excluded in \hyperref[completefaces]{Proposition \ref*{completefaces}}. The three-dimensional edge cone $\sigma_{K_{2,2}}$ is generated by the extremal rays $e_1,e_2,f_1,f_2$  where $(e_1,e_2)$ and $(f_1,f_2)$ do not span a 2-face. For $K_{2,3}$ we observe that the intersection graphs $\G\{U_1 \backslash \{1\}\} \cap \G\{U_1 \backslash \{2\}\}$ and $\G\{U_2 \backslash \{3\}\} \cap \G\{U_2 \backslash \{4\}\} \cap \G\{U_2 \backslash \{5\}\}$ have five isolated vertices and therefore $(e_1,e_2)$ does not span a 2-face and $(f_1,f_2,f_3)$ does not span a 3-face. The second figure is the combinatorial representation of the four dimensional cone $\sigma_{K_{2,3}}$.

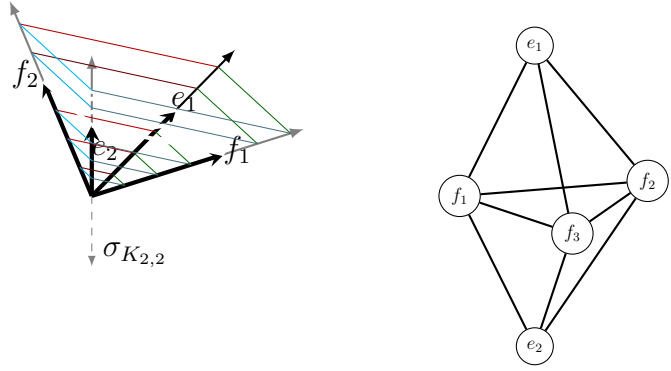
\begin{center}

\begin{tikzpicture}[baseline=1,scale=1.5,every path/.style={>=latex},every node/.style={draw,circle,fill=white,scale=0.6}]
\tdplotsetmaincoords{390}{-20}
\tdplotsetrotatedcoords{0}{0}{50}

\begin{scope}[scale=2.5,tdplot_main_coords,every node/.style={draw=none,fill=none},xshift=45, yshift=0cm]


\coordinate (O) at (0,0,0);
\coordinate (a) at (0.5,0,0) ;
\coordinate (b) at (0,0.5,0);
\coordinate (c) at (0.5,0.5,-0.5);
\coordinate (d) at (0,0,0.5);




\draw[thick,->,gray] (0,0,0) -- (0.8,0,0) node[anchor=north east]{};
\draw[thick,->,gray] (0,0,0) -- (0,0.85,0) node[anchor=north west]{};
\draw[dashed,->,gray] (0,0,0) -- (0,0,-0.5) node[anchor=south]{};
\draw[thick,->,gray] (0,0,0) -- (0,0,1) node[anchor=south]{};
\draw[thick,->,black] (0,0,0) -- (0.85,0.85,-0.85) node[anchor=south]{};

\tdplotsetrotatedcoords{60}{30}{10}
\draw[-stealth,ultra thick, color=black] (O) -- (a);
\draw[-stealth,ultra thick, color=black] (O) -- (b);
\draw[-stealth, ultra thick, color=black] (O) -- (c);
\draw[-stealth,ultra thick, color=black] (O) -- (d);

\node (O1) at (0,0,0) {};
\node (a1) at (0.55,0,0) {$f_1$};
\node (b1) at (-0.05,0.55,0) {$f_2$};
\node (c1) at (0.55,0.55,-0.55) {$e_1$};
\node (d1) at (0,-0.14,0.55) {$e_2$};
\node (e1) at (0.2,0.1,-0.7) {\textbf{$\sigma_{K_{2,2}}$}};

\draw [white,ultra thick] (a) -- (c) node [above,pos=.5,opacity=1] {};
\draw [black!60!green, very thin] (0.25,0,0)--(0.25,0.25,-0.25);
\draw [black!60!green, very thin] (0.375,0,0)--(0.375,0.375,-0.375);
\draw [black!60!green, very thin] (0.125,0,0)--(0.125,0.125,-0.125);
\draw [black!60!green, very thin] (0.625,0,0)--(0.625,0.625,-0.625);
\draw [black!60!green, very thin] (0.75,0,0)--(0.75,0.75,-0.75);

\draw [white,ultra thick] (c)--(d);
\draw [white,ultra thick] (0.5,0.25,-0.25)--(0.25,0,0.25);
\draw [white,ultra thick] (0.5,0.375,-0.375)--(0.125,0,0.375);
\draw [white,ultra thick] (0.25,0.5,-0.25)--(0,0.25,0.25);
\draw [white,ultra thick] (0.375,0.5,-0.375)--(0,0.125,0.375);

\draw [white,ultra thick] (b) -- (c) node [above,pos=.5,opacity=1] {};
\draw [black!30!red] (0,0.25,0) -- (0.25,0.25,-0.25) ;
\draw [black!60!red] (0,0.125,0) -- (0.125,0.125,-0.125) ;
\draw [black!30!red] (0,0.375,0) -- (0.375,0.375,-0.375) ;
\draw [black!60!red] (0,0.625,0) -- (0.625,0.625,-0.625) ;
\draw [black!30!red] (0,0.75,0) -- (0.75,0.75,-0.75) ;

\draw [white,ultra thick] (a) -- (d) node [above,pos=.5,opacity=1] {};
\draw [black!60!cyan, ultra thin] (0.25,0,0) -- (0,0,0.25)  ;
\draw [black!60!cyan, ultra thin] (0.125,0,0) -- (0,0,0.125)  ;
\draw [black!60!cyan, ultra thin] (0.375,0,0) -- (0,0,0.375)  ;
\draw [black!60!cyan, ultra thin] (0.625,0,0) -- (0,0,0.625)  ;
\draw [black!60!cyan, ultra thin] (0.75,0,0) -- (0,0,0.75)  ;

\draw [white,ultra thick] (d) -- (b) node [above,pos=.5,opacity=1] {};
\draw [white!10!cyan] (0,0,0.375) -- (0,0.375,0)  ;
\draw [white!10!cyan] (0,0,0.25) -- (0,0.25,0)  ;
\draw [black!10!cyan] (0,0,0.125) -- (0,0.125,0)  ;
\draw [black!10!cyan] (0,0,0.625) -- (0,0.625,0)  ;
\draw [white!10!cyan] (0,0,0.75) -- (0,0.75,0)  ;

\end{scope}

\end{tikzpicture}
\hspace{1.5cm}
\begin{tikzpicture}[baseline=1,scale=1,every path/.style={>=latex},every node/.style={draw,circle,fill=black,scale=0.2}]
  \node      [draw=black, circle, fill=none, scale=3]      (a1) at (0,0)  {$f_1$};
  \node      [draw=black, circle, fill=none, scale=3]      (a2) at (1,2)  {$e_1$};
  \node      [draw=black, circle, fill=none, scale=3]      (a3) at (2.5,0.2)  {$f_2$};
  \node      [draw=black, circle, fill=none, scale=3]      (a4) at (1.5,-0.5)  {$f_3$};
  \node      [draw=black, circle, fill=none, scale=3]      (a5) at (1,-2)  {$e_2$};

 \node      [draw=none, fill=none, scale=5]      (a6) at (2,-1.5)  {};

 \draw[-,thick] (a1) edge (a3);
 \draw[-,thick] (a4) edge (a3);
 \draw[-,thick] (a4) edge (a1);
 \draw[-,thick] (a1) edge (a2);
 \draw[-,thick] (a4) edge (a2);
 \draw[-,thick] (a3) edge (a2);
 \draw[-,thick] (a1) edge (a5);
 \draw[-,thick] (a4) edge (a5);
 \draw[-,thick] (a3) edge (a5);

\end{tikzpicture}
\end{center}

\noindent Finally, consider the complete bipartite graph $K_{3,3}$. Similar to the calculation on $K_{2,3}$, we observe that $(e_1,e_2,e_3)$ and $(f_1,f_2,f_3)$ do not span 3-faces. Any other triple of extremal ray generators spans a 3-face.
\end{ex}

\begin{prop} \label{completefaces}
The two-dimensional faces of $\sigma_{K_{m,n}}$ are
\begin{enumerate}
\item[\normalfont{(1)}] all pairs except $(e_1,e_2)$, if $m=2$ and $n\geq 3$.
\item[\normalfont{(2)}] all pairs of extremal rays, if $m \geq 3$ and $n \geq 3$.
\end{enumerate} 
The three-dimensional faces of $\sigma_{K_{m,n}}$ are
\begin{enumerate}
\item[\normalfont{(1)}] all triples of extremal rays not containing both $e_1$ and $e_2$, if $m=2$, $n\geq 4$.
\item[\normalfont{(2)}] all triples of extremal rays except $(e_1,e_2,e_3)$, if $m=3$ and $n \geq 4$.
\item[\normalfont{(3)}] all triples of extremal rays, if $m \geq 4$ and $n\geq 4$.
\end{enumerate}
\end{prop}

\begin{proof}
The characterization of the two-dimensional faces follows by \hyperref[aaconditions]{Proposition \ref*{aaconditions}}. Since there exists no two-sided first independent set, the characterization of the three-dimensional faces follows by \hyperref[3faceAA]{Lemma \ref*{3faceAA}} (1) and by \hyperref[3faceAB]{Lemma \ref*{3faceAB}} (1)(iii). 
\end{proof}

\begin{ex}\label{minirigidi}
Let us study the deformation space for the complete bipartite graphs from \hyperref[smallcompletes]{Example \ref*{smallcompletes}}. For $K_{2,2}$ and the deformation degree $R=[1,1,1,1] \in M$, the vertices of $Q(R)$ are all lattice vertices. This implies that $T_{\TV(K_{2,2})}^1(-R) \neq 0$. Next, let us consider the edge cone $\sigma_{K_{2,3}}$. It does not have any non-simplicial 3-face. It suffices to check the cases where the non 2-face pair $(e_1,e_2)$ or non 3-face triple $(f_1,f_2,f_3)$ is in the compact part of crosscut $Q(R)$. Suppose that $\overline{f_1},\overline{f_2},\overline{f_3}$ are vertices in $Q(R)$, for a deformation degree $R = [R_1, \ldots, R_5] \in M$. Then we obtain that $R_1 + R_2 \geq 3$. This means that there exists a non-lattice vertex $\overline{e_i} \in Q(R)$. Now suppose that $\overline{e_1}$ and $\overline{e_2}$ are vertices in $Q(R)$. Then we have that $R_3 + R_4 + R_5 \geq 2$ and thus there exists a non-lattice vertex $\overline{f_j}$ or there exist two lattice vertices $\overline{f_{k}}$ and $\overline{f_l}$ in $Q(R)$. In \hyperref[kucukrigidexs]{Figure \ref*{kucukrigidexs}}, these cases and their vector space $V(R)$ are illustrated. 
\begin{center}
\begin{tikzpicture}[scale=0.9,every path/.style={>=latex},every node/.style={draw,circle,fill=black,scale=0.4}]
  \node      [circle]      (a) at (-1,0)  {};
  \node      [right,draw=none, circle, fill=none,scale=2]      (at) at (-1,0)  {$\overline{f_3}$};
  \node      [draw=none, fill=none, scale=1]      (a1) at (-0.9,0.3)  {};
  \node      [draw=none, fill=none, scale=1]      (a2) at (-1.2,-1.3)  {};
  \node        [circle]      (b) at (-1.5,2)  {};
  \node        [left,draw=none, circle, fill=none,scale=2]    (bt) at (-1.5,2)  {$\overline{f_1}$};
  \node      [draw=none, fill=none, scale=1]      (b1) at (-1.2,1.7)  {};
  \node      [draw=none, fill=none, scale=1]      (b2) at (-1.7,0.5)  {};
  \node       [ circle]        (d) at (1.5,1.5) {};
  \node      [right,draw=none, circle, fill=none,scale=2]         (dt) at (1.5,1.5) {$\overline{f_2}$};
  \node      [draw=none, fill=none, scale=1]      (d1) at (1,1.4)  {};
  \node      [draw=none, fill=none, scale=1]      (d2) at (1.3,0)  {};
  \node       [circle]        (c) at (0.2,3.5) {};
  \node       [above,draw=none, circle, fill=none,scale=2]        (ct) at (0.2,3.5) {$\overline{e_i}$};

  \node      [draw=none, fill=none, scale=2]      (t1) at (-1.4,1)  {$t$};
  \node      [draw=none, fill=none, scale=2]      (t2) at (0.6,0.6)  {$t$};
  \node      [draw=none, fill=none, scale=2]      (t3) at (-0.3,1.3)  {$t$};
  \node      [draw=none, fill=none, scale=2]      (t4) at (-1,2.7)  {$t$};
  \node      [draw=none, fill=none, scale=2]      (t5) at (1.2,2.5)  {$t$};
  \node      [draw=none, fill=none, scale=2]      (t6) at (0,2)  {$t$};

 \draw[->,thick] (b) edge (b2);
 \draw[->,thick] (d) edge (d2);
 \draw[->,thick] (a) edge (a2);

    \draw[thick,green] (c) edge (b);
    \draw[thick,green] (a) edge (c);
    \draw[thick,green] (d) edge (c);

    \draw[-,green] (a) edge (b);
 \draw[dashed,very thick,red] (a1) edge (b1);
 \draw[dashed,very thick,red] (d1) edge (b1);
 \draw[dashed,very thick,red] (a1) edge (d1);

    \draw[-,thick,green] (d) edge (b);



\draw [-,green] (a) edge (d);


\path[draw, fill=green!20, opacity=.4] (-1.5,2)--(0.2,3.5)--(-1,0)--cycle;
\path[draw, fill=green!20, opacity=.5] (-1.5,2)--(1.5,1.5)--(0.2,3.5)--cycle;
\path[draw, fill=green!20, opacity=.4] (-1,0)--(1.5,1.5)--(0.2,3.5)--cycle;


\end{tikzpicture}
\hspace{1.5cm}
\begin{tikzpicture}[scale=1,every path/.style={>=latex},every node/.style={draw,circle,fill=black,scale=0.4}]

  \node      [draw=black, circle]      (a2) at (0.8,1.2)  {};
  \node      [right, draw=none, fill=none, scale=2]      (a2tt) at (0.8,1.2)  {$\overline{e_2}$};
  \node      [draw=none, fill=none, scale=3]      (a2b) at (-0.3,1.6)  {};
  \node      [draw=black, circle, fill=green]      (a4) at (1.5,-0.3)  {};
  \node      [right,draw=none, fill=none, scale=2]      (a4tt) at (1.5,-0.3)  {$\overline{f_j}$};
  \node      [draw=none, circle, fill=white, scale=2]      (a4t) at (1.3,0.5)  {$t$};
  \node      [draw=black, circle]      (a5) at (1,-1.5)  {};
  \node      [right,draw=none, fill=none, scale=2]      (a5tt) at (1,-1.5)  {$\overline{e_1}$};
  \node      [draw=none, circle, fill=white, scale=2]      (a5t) at (1.44,-0.8)  {$t$};
  \node      [draw=none,  fill=none, scale=3]      (a5b) at (0.1,-2)  {};

 \draw[-,thick] (a4) edge (a2);
 \draw[->,thick] (a2) edge (a2b);
 \draw[->,thick] (a5) edge (a5b);
 \draw[dashed,thick,red] (a5) edge (a2);

 \draw[-,thick] (a4) edge (a5);

\end{tikzpicture}
\hspace{1.5cm}
\begin{tikzpicture}[scale=1,every path/.style={>=latex},every node/.style={draw,circle,fill=black,scale=0.4}]
  \node      [draw=black, circle]      (a1) at (0,0)  {};
\node [left,draw=none, circle, fill=none,scale=2]    (a1t) at (0,0)  {$f_l$};
  \node      [draw=black, circle]      (a2) at (1,2)  {};
\node [above,draw=none, circle, fill=none,scale=2]      (a2t) at (1,2)  {$e_1$};
  \node      [draw=black, circle]      (a3) at (2.5,0.2)  {};
\node [right,draw=none, circle, fill=none,scale=2]     (a3t) at (2.5,0.2)  {$f_k$};
  \node      [draw=black, circle]      (a5) at (1,-2)  {};
\node [below,draw=none, circle, fill=none,scale=2]    (a5t) at (1,-2)  {$e_2$};

\node [draw=none, fill=none,scale=2]    (t1) at (2,1)  {$t$};
\node [draw=none, fill=none,scale=2]    (t2) at (0.2,0.9)  {$t$};
\node [draw=none, fill=none,scale=2]    (t3) at (1.2,0.3)  {$t$};
\node [draw=none, fill=none,scale=2]    (t4) at (0.3,-1)  {$t$};
\node [draw=none, fill=none,scale=2]    (t5) at (1.9,-0.9)  {$t$};

 \draw[-,thick] (a1) edge (a3);
 \draw[-,thick] (a1) edge (a2);
 \draw[-,thick] (a3) edge (a2);
 \draw[-,thick] (a1) edge (a5);
 \draw[-,thick] (a3) edge (a5);
 \draw[dashed,red,thick] (a2) edge (a5);

\path[draw, fill=green!20, opacity=.4] (0,0)--(1,2)--(2.5,0.2)--cycle;
\path[draw, fill=green!20, opacity=.4] (0,0)--(2.5,0.2)--(1,-2)--cycle;

\end{tikzpicture}

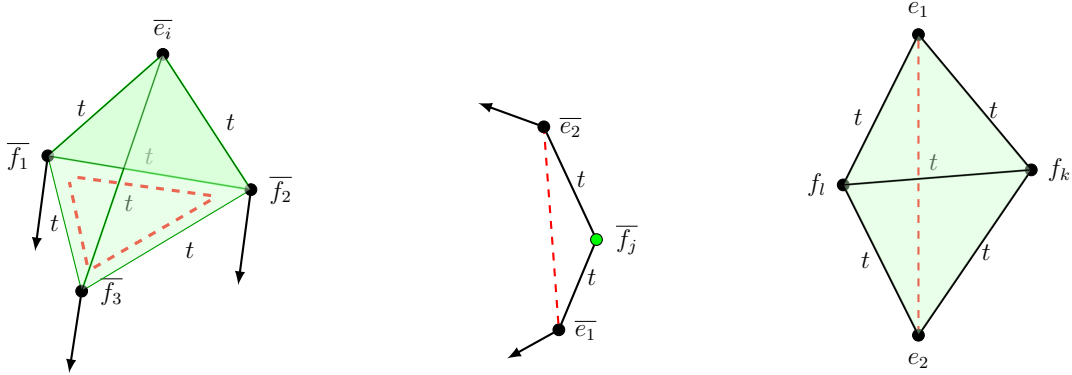
\captionof{figure}{Some crosscut pictures of the edge cone $\sigma_{K_{2,3}}$}\label{kucukrigidexs}
\end{center}
\end{ex}

\noindent Finally, we consider the edge cone of $K_{3,3}$. Similar to $\sigma_{K_{2,3}}$, if $\overline{f_1}, \ \overline{f_2}$ and $\overline{f_3}$ are vertices in $Q(R)$, then there exists a non-lattice vertex $\overline{e_i}$ in $Q(R)$. The same follows  symmetrically for the vertices $\overline{e_1},\ \overline{e_2}$ and $\overline{e_3}$.

\begin{thm}\label{rigidcomplete}
$\TV(K_{m,n})$ is rigid except for $m=n = 2$. 
\end{thm}

\begin{proof}
We have shown in \hyperref[minirigidi]{Example \ref*{minirigidi}} that $\TV(K_{2,2})$ is not rigid and $\TV(K_{2,3})$ and $\TV(K_{3,3})$ are rigid. By \hyperref[completefaces]{Proposition \ref*{completefaces}}, it remains to prove three cases: \\

\noindent [$m=2$ and $n\geq 4$]: The 2-faces are all pairs except $(e_1,e_2)$ and the 3-faces are all triples which do not contain both $e_1$ and $e_2$. Assume that there exists a deformation degree $R\in M$ such that $\overline{e_1}$ and $\overline{e_2}$ are vertices in $Q(R)$ and $\overline{f_j}$ is a lattice vertex in $Q(R)$ for some $j \in [n]$. Then we obtain that 
\[R_3 + \ldots + R_{j+1} + R_{j+3} + \ldots + R_{n+2} \geq 1.\]
Thus there exists a vertex $\overline{f_{j'}} \in Q(R)$ with $j' \neq j$. Hence we conclude that $T_{K_{m,n}}^1(-R) = 0$, since $(e_1,f_j,f_{j'})$ and $(e_2,f_j,f_{j'})$ are 3-faces and $t$ is transfered by the edge $\overline{\overline{f_j} \ \overline{f_{j'}}}$ as explained in \hyperref[prooftechnique]{Remark \ref*{prooftechnique}}.

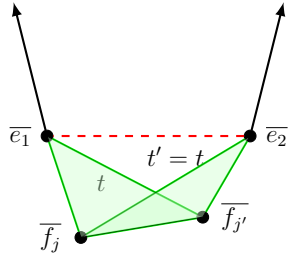
\begin{center}

\begin{tikzpicture}[scale=0.9,every path/.style={>=latex},every node/.style={draw,circle,fill=black,scale=0.4}]

  \node      [circle]      (a) at (-1,0)  {};
  \node      [left,draw=none, circle, fill=none,scale=2]      (at) at (-1,0)  {$\overline{f_j}$};
  
    \node        [circle]      (b) at (-1.5,1.5)  {};
  \node        [left,draw=none, circle, fill=none,scale=2]      (bt) at (-1.5,1.5)  {$\overline{e_1}$};
  
    \node       [circle]        (d) at (1.5,1.5) {};
  \node       [right,draw=none, circle, fill=none,scale=2]        (dt) at (1.5,1.5) {$\overline{e_2}$};

\node      [circle]      (c) at (0.8,0.3)  {};
\node      [right,draw=none, circle, fill=none,scale=2]      (ct) at (0.8,0.3)  {$\overline{f_{j'}}$};
 
  \node      [draw=none, fill=white, scale=2]      (a1) at (-0.7,0.8)  {$t$};

  \node      [draw=none, fill=white, scale=2]      (a4) at (0.4,1.2)  {$t'=t$};
  \node      [draw=none, fill=none, scale=0.25]      (c1) at (-2,3.5)  {};
\node      [draw=none, fill=none, scale=0.25]      (c2) at (2,3.5)  {};

    \draw[->,thick] (b) edge (c1);
    \draw[->,thick] (d) edge (c2);
    \draw[green,thick] (a) edge (b);
    \draw[dashed,thick,red] (d) edge (b);



\draw [thick,green] (a) edge (d);
\draw [thick,green] (c) edge (a);
\draw [thick,green] (c) edge (b);
\draw [thick,green] (c) edge (d);


\path[draw, fill=green!20, opacity=.4] (-1.5,1.5)--(0.8,0.3)--(-1,0)--cycle;
\path[draw, fill=green!20, opacity=.4] (1.5,1.5)--(0.8,0.3)--(-1,0)--cycle;


\end{tikzpicture}
\hspace{2cm}

\captionof*{figure}{The 2-faces are colored in green and the vertex $\overline{f_j}$ is a lattice vertex in $Q(R)$.}

\end{center}
\vspace{0.5cm}
[$m=3$ and $n \geq 4$]: The 2-faces are all pairs and the 3-faces are all triples except $(e_1,e_2,e_3)$. We just need to check the case where the non 3-face $(e_1,e_2,e_3)$ is in the compact part of $Q(R)$. In this case, we obtain that $\sum_{i=4}^{n+3} R_i \geq 3$. This implies that there exists a vertex $\overline{f_j}$ for some $j \in [n]$. Thus $t$ is transfered by the 2-faces $(f_j,e_1)$, $(f_j,e_2)$ and $(f_j,e_3)$. 
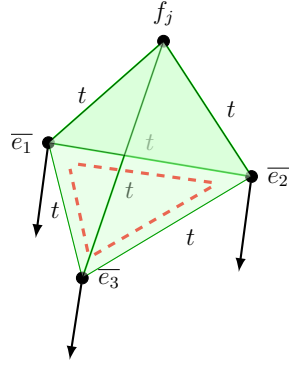
\begin{center}
\begin{tikzpicture}[scale=0.9,every path/.style={>=latex},every node/.style={draw,circle,fill=black,scale=0.4}]
  \node      [circle]      (a) at (-1,0)  {};
  \node      [right,draw=none, circle, fill=none,scale=2]      (at) at (-1,0)  {$\overline{e_3}$};
  \node      [draw=none, fill=none, scale=1]      (a1) at (-0.9,0.3)  {};
  \node      [draw=none, fill=none, scale=1]      (a2) at (-1.2,-1.3)  {};
  \node        [circle]      (b) at (-1.5,2)  {};
  \node        [left,draw=none, circle, fill=none,scale=2]    (bt) at (-1.5,2)  {$\overline{e_1}$};
  \node      [draw=none, fill=none, scale=1]      (b1) at (-1.2,1.7)  {};
  \node      [draw=none, fill=none, scale=1]      (b2) at (-1.7,0.5)  {};
  \node       [ circle]        (d) at (1.5,1.5) {};
  \node      [right,draw=none, circle, fill=none,scale=2]         (dt) at (1.5,1.5) {$\overline{e_2}$};
  \node      [draw=none, fill=none, scale=1]      (d1) at (1,1.4)  {};
  \node      [draw=none, fill=none, scale=1]      (d2) at (1.3,0)  {};
  \node       [circle]        (c) at (0.2,3.5) {};
  \node       [above,draw=none, circle, fill=none,scale=2]        (ct) at (0.2,3.5) {$\overline{f_j}$};

  \node      [draw=none, fill=none, scale=2]      (t1) at (-1.4,1)  {$t$};
  \node      [draw=none, fill=none, scale=2]      (t2) at (0.6,0.6)  {$t$};
  \node      [draw=none, fill=none, scale=2]      (t3) at (-0.3,1.3)  {$t$};
  \node      [draw=none, fill=none, scale=2]      (t4) at (-1,2.7)  {$t$};
  \node      [draw=none, fill=none, scale=2]      (t5) at (1.2,2.5)  {$t$};
  \node      [draw=none, fill=none, scale=2]      (t6) at (0,2)  {$t$};

 \draw[->,thick] (b) edge (b2);
 \draw[->,thick] (d) edge (d2);
 \draw[->,thick] (a) edge (a2);

    \draw[thick,green] (c) edge (b);
    \draw[thick,green] (a) edge (c);
    \draw[thick,green] (d) edge (c);

    \draw[-,green] (a) edge (b);
 \draw[dashed,very thick,red] (a1) edge (b1);
 \draw[dashed,very thick,red] (d1) edge (b1);
 \draw[dashed,very thick,red] (a1) edge (d1);

    \draw[-,thick,green] (d) edge (b);



\draw [-,green] (a) edge (d);


\path[draw, fill=green!20, opacity=.4] (-1.5,2)--(0.2,3.5)--(-1,0)--cycle;
\path[draw, fill=green!20, opacity=.5] (-1.5,2)--(1.5,1.5)--(0.2,3.5)--cycle;
\path[draw, fill=green!20, opacity=.4] (-1,0)--(1.5,1.5)--(0.2,3.5)--cycle;


\end{tikzpicture}
\end{center}

\captionof*{figure}{The dashed red area means that $( e_1,e_2,e_3)$ do not span a 3-face.}
\vspace{0.5cm}

[$m \geq 4$ and $n\geq 4$]: All pairs are 2-faces and all triples are 3-faces. Hence the associated toric variety is rigid. 
\end{proof}

\subsection{Complete bipartite graphs with multiple edge removals}
Recall again that $C \in \Ione$ is two-sided with $C_1= \{1,\ldots,t_1\}$ and $C_2=\{m+1,\ldots,m+t_2\}$. We denote $\pi(C)$ as $\mathfrak{c} = \sum_{i>t_1} e_i - \sum_{j\leq t_2} f_j$ under the map from \hyperref[11thm]{Theorem \ref*{11thm}}.
\begin{prop}\label{obvious}
Let $G \subset K_{m,n}$ be a connected bipartite graph with exactly one two-sided first independent set $C \in \Ione$.
\begin{enumerate}
\item The pair $(f_{n-1},f_n)$ does not span a two-dimensional face if and only if $|C_2|=n-2$. Moreover, no simplicial three-dimensional face contains both $f_{n-1}$ and $f_n$.
\item The pair $(\mathfrak{c},e_1)$ does not span a two-dimensional face if and only if $|C_1|=1$ and $|C_2|\neq n-1$, Moreover, no simplicial three-dimensional face contains both $\mathfrak{c}$ and $e_1$.
\item If $|C_1|=1$ and $|C_2|=n-2$, then $\TV(G)$ is not rigid.

\end{enumerate}
\end{prop}

\begin{proof}
By \hyperref[3faceAA]{Lemma \ref*{3faceAA}} (2)(i), we obtain the non-simplicial 3-faces $(\mathfrak{c}, e_1, f_{n-1},f_n)$ in (3).  It results to a non-rigid toric variety $\TV(G)$ by \hyperref[4tuplenonrigid]{Theorem \ref*{4tuplenonrigid}}. In (1), by \hyperref[aaconditions]{Proposition \ref*{aaconditions}} (1), $(f_{n-1}, f_n)$ does not span a two-dimensional face. Since we have exactly one two-sided first independent set $C$, by \hyperref[3faceAA]{Lemma \ref*{3faceAA}} (2)(i), the only three-dimensional face containing $(f_{n-1},f_n)$ is again $(\mathfrak{c}, e_1, f_{n-1},f_n)$ as in (3). Similarly for (2), by \hyperref[acbcconditions]{Proposition \ref*{acbcconditions}}(1) an (3) $(\mathfrak{c},e_1)$ does not span a two-dimensional face if and only if $|C_1|=1$ and $|C_2|\neq n-1$. 
\end{proof}

\noindent Note that the cases where $|C_2|=1$ and $|C_1|=m-2$ can be studied symmetrically. In the next proposition, we examine the three-dimensional faces of $\sigma_G$. These statements can also be studied symmetrically.

\begin{prop} \label{non3}
Let $G \subset K_{m,n}$ be a connected bipartite graph with exactly one two-sided first independent set $C \in \Ione$.
\begin{enumerate}
\item The triple $(f_{n-2},f_{n-1},f_n)$ does not span a three-dimensional face if and only if $|C_2|=n-3$.
\item The triple $(\mathfrak{c},e_1,e_2)$ does not span a three-dimensional face if and only if $|C_1|=2$ and $|N(C_1)|\neq 1$.
\item The triple $(e_1, f_{n-1}, f_{n})$ does not span a three-dimensional face if and only if $|C_2|=n-2$
\item The triple $(\mathfrak{c},e_1,f_1)$ does not span a three-dimensional face if and only if $|C_1|=1$ or $|C_2|=1$, except for $G \subset K_{2,2}$.
\end{enumerate}
\end{prop}

\begin{proof}
For (1), the intersection subgraph $\G\{U_2 \backslash \{m+n-2\}\}\cap \G\{U_2 \backslash \{m+n-1\}\} \cap \G\{U_2 \backslash \{m+n\}\}$ has more than four connected components if and only if $|C_2|=n-3$.  For (2), the intersection subgraph $\G\{C\} \cap \G\{U_1 \backslash \{1\}\} \cap \G\{U_1 \backslash \{2\}\}$ has more than four connected components if and only if $|C_1|=2$. In particular, if $|N(C_1)|=1$, $(\mathfrak{c}, e_1,e_2)$ spans a 3-face. For (3) we refer to the proof of \hyperref[obvious]{Proposition \ref*{obvious}} (1) and (3). For (4), the intersection subgraph $\G\{C\} \cap \G\{U_1 \backslash \{1\}\}$ has more than three connected components if $|C_1|$ or $|C_2|$ is equal to one. In particular the graph $G \subset K_{2,2}$ has been examined in \hyperref[kucuk]{Example \ref*{kucuk}}: $\sigma_G$ is generated by $(\mathfrak{c},e_1,f_1)$ and $\TV(G)$ is rigid.
\end{proof}

\begin{thm}\label{onetwosided}
Let $G \subsetneq K_{m,n}$ be a connected bipartite graph with exactly one two-sided first independent set $C \in \Ione$. Then 
\begin{enumerate}
\item $\TV(G)$ is not rigid, if $|C_1|=1$ and $|C_2| = n-2$ or if $|C_1|=m-2$ and $|C_2|=1$.
\item $\TV(G)$ is rigid, otherwise.
\end{enumerate} 
\end{thm}

\begin{proof}
The first case follows from \hyperref[obvious]{Proposition \ref*{obvious}} (3). For the other cases, we study the non 2-faces and 3-faces appearing in the compact part of $Q(R)$ utilizing the previous two propositions.  First of all, note that there exists no case such as two 2-faces connected only by a common lattice vertex in $Q(R)$. This is because, it would mean that there exist four non 2-faces and this is impossible for our bipartite graph $G$.\\

\noindent $\bullet$ Assume that $|C_2| =n-2$. We consider the non 2-face $(f_{n-1},f_n)$ in $Q(R)$. This means that $R_{n-1} \geq 1$ and $R_{n} \geq 1$. This implies that either there exists $i \in [m]$ such that $R_i \geq 1$ or there exists $m+j \in C_2$ such that $R_{m+j} \leq -1$ i.e.\ $\mathfrak{c} \in Q(R)$.
\begin{enumerate}
\item{$R_i \geq 1$:} Suppose that $R$ evaluates zero or negative on all other extremal rays except $e_i , f_{n-1}$ and $f_n$. Then $\overline{e_i}$ is not a lattice vertex in $Q(R)$ and $(e_i,f_{n-1})$ and $(e_i, f_{n})$ are 2-faces. If $e_i$ is not an extremal ray, i.e.\ if $|C_1| = m-1$, then $\overline{\mathfrak{c}}$ is not a lattice vertex in $Q(R)$ and $(\mathfrak{c},f_{n-1})$ and $(\mathfrak{c}, f_{n})$ are 2-faces by \hyperref[obvious]{Proposition \ref*{obvious}} (2). Suppose now that there exists another $i' \in [m] \backslash \{i\}$ such that $R_{i'} \geq 1$. If $\overline{e_i}$ and $\overline{e_{i'}}$ are not lattice vertices, we are done. If at least one of them is a lattice vertex, then we check if $(e_i, e_{i'})$ spans a 2-face. If it does span a 2-face, then we obtain the 3-faces $(e_i, e_{i'},f_{n-1})$ and $(e_i, e_{i'},f_{n})$. If it does not span a 2-face, then $|C_1|= m-2$ and let $e_i = e_{n-1}$ and $e_{i'} = e_n$ by \hyperref[obvious]{Proposition \ref*{obvious}} (1). In that case, $\overline{\mathfrak{c}}$ is a non-lattice vertex and we obtain the 3-faces $(\mathfrak{c}, e_{i}, f_j)$ where $i \in \{m-1, m\}$ and $j \in \{n-1,n\}$ as in the figure below.  \\
\item{$R_{m+j} \leq -1$:} We need to examine the case where $\overline{\mathfrak{c}}$ is a lattice vertex. Then there exists $i \in C_1$ such that $R_i \geq 1$, i.e.\ $e_i \in \sigma_G^{(1)}$. By \hyperref[obvious]{Proposition \ref*{obvious}} (2), $(\mathfrak{c},e_i)$ spans a 2-face. By \hyperref[non3]{Proposition \ref*{non3}} (4), $(\mathfrak{c}, e_i, f_{n-1})$ and $(\mathfrak{c}, e_i , f_n)$ span 3-faces, since $G$ is connected and thus $n\geq 3$.
\end{enumerate}
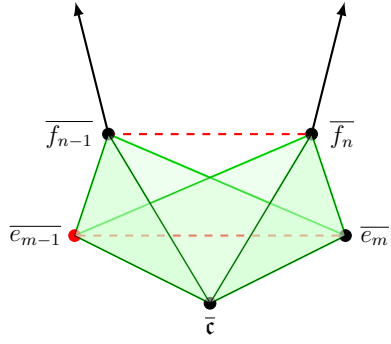
\begin{center}
\begin{tikzpicture}[scale=0.9,every path/.style={>=latex},every node/.style={draw,circle,fill=black,scale=0.4}]
  \node      [left,draw=none, circle, fill=none, scale=2]      (a) at (-2,0)  {$\overline{e_{m-1}}$};
  \node        [left,draw=none, circle, fill=none, scale=2]      (b) at (-1.5,1.5)  {$\overline{f_{n-1}}$};
  \node       [right,draw=none, circle, fill=none, scale=2]        (d) at (1.5,1.5) {$\overline{f_{n}}$};
  \node       [below,draw=none, circle, fill=none, scale=2]        (e) at (0,-1) {$\overline{\mathfrak{c}}$};
  
    \node      [draw=red, circle, fill=red, scale=1]      (a) at (-2,0)  {};
  \node        [draw=black, circle, fill=black, scale=1]      (b) at (-1.5,1.5)  {};
  \node      [draw=black, circle, fill=black, scale=1]       (d) at (1.5,1.5) {};
  \node    [draw=black, circle, fill=black, scale=1]       (e) at (0,-1) {};

\node  [right,draw=none, circle, fill=none, scale=2]         (c) at (2,0)  {$\overline{e_{m}}$};
\node     [draw=black, circle, fill=black, scale=1]       (c) at (2,0)  {};

  \node      [draw=none, fill=none, scale=0.5]      (c1) at (-2,3.5)  {};
\node      [draw=none, fill=none, scale=0.5]      (c2) at (2,3.5)  {};

    \draw[->,thick] (b) edge (c1);
    \draw[->,thick] (d) edge (c2);
    \draw[green,thick] (a) edge (b);
    \draw[dashed,thick,red] (d) edge (b);



\draw [thick,green] (a) edge (e);
\draw [thick,green] (b) edge (e);
\draw [thick,green] (c) edge (e);
\draw [thick,green] (d) edge (e);

\draw [thick,green] (a) edge (d);
\draw [dashed,thick,red] (c) edge (a);
\draw [thick,green] (c) edge (b);
\draw [thick,green] (c) edge (d);

\path[draw, fill=green!20, opacity=.3] (0,-1)--(2,0)--(-1.5,1.5)--cycle;
\path[draw, fill=green!20, opacity=.6] (0,-1)--(2,0)--(1.5,1.5)--cycle;
\path[draw, fill=green!20, opacity=.3] (0,-1)--(-2,0)--(1.5,1.5)--cycle;
\path[draw, fill=green!20, opacity=.6] (0,-1)--(-2,0)--(-1.5,1.5)--cycle;



\end{tikzpicture}
\end{center}

\noindent $\bullet$ Assume that $|C_1|=1$. We know that $(\mathfrak{c},e_1)$ is a non 2-face by \hyperref[obvious]{Proposition \ref*{obvious}} (2). Assume that there exists $R \in M$ that evaluates on the extremal rays $\mathfrak{c}$ and $e_1$ bigger than or equal to one. Then there exists $m+j \in N(C_1)$ such that $R_{m+j} \geq 1$. Assume that $\overline{f_j}$ is a lattice vertex, then there exists $m+j' \in N(C_1)$ such that $\overline{f_{j'}} \in Q(R)$. Now, we must examine if $(f_{j'}, f_j)$, $(f_{j'}, e_1)$ and $(f_{j'}, \mathfrak{c})$ are 2-faces.  Since we excluded the case where $|C_2| = n-2$ and we have that $\{m+j, m+j'\} \in N(C_1) = U_2 \backslash C_2$, by \hyperref[non3]{Proposition \ref*{non3}}, $(\mathfrak{c},f_{j}, f_{j'})$ and $(e_1, f_j, f_{j'})$ span 3-faces. 

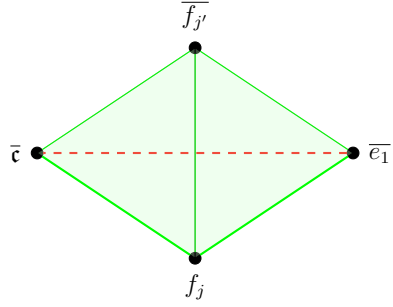
\begin{center}

\begin{tikzpicture}[baseline=1cm,scale=0.7,every path/.style={>=latex},every node/.style={draw,circle,fill=black,scale=0.4}]
  \node     [below,draw=none, circle, fill=none, scale=2]   (a) at (0,0)  {$f_j$};
  \node       [left,draw=none, circle, fill=none, scale=2]     (b) at (-3,2)  {$\overline{\mathfrak{c}}$};
  \node     [right,draw=none, circle, fill=none, scale=2]       (d) at (3,2) {$\overline{e_1}$};
  \node       [above,draw=none, circle, fill=none, scale=2]      (f) at (0,4) {$\overline{f_{j'}}$};

  \node      [draw=black, circle, fill=black, scale=1]      (a) at (0,0)  {};
  \node        [draw=black, circle, fill=black, scale=1]      (b) at (-3,2)  {};
  \node       [draw=black, circle, fill=black, scale=1]         (d) at (3,2) {};
  \node        [draw=black, circle, fill=black, scale=1]        (f) at (0,4) {};

    \draw[dashed,thick,red] (d) edge (b);
\path[draw, fill=green!20, opacity=.3] (-3,2)--(0,0)--(0,4)--cycle;
\path[draw, fill=green!20, opacity=.3] (0,4)--(0,0)--(3,2)--cycle;

    \draw[thick,green] (a) edge (b);

 \draw [-,green] (f) edge (b);
 \draw [-,green] (f) edge (d);
 \draw [-,green] (f) edge (a);

\draw [thick,green] (a) edge (d);

\end{tikzpicture}

\end{center}

\begin{enumerate}
\item $|C_2|=1$: Consider the non 2-face $(\mathfrak{c},f_1)$ in $Q(R)$. Then there exists $i \in [m] \backslash \{1\}$ such that $\overline{e_i} \in Q(R)$. Similarly to the case of non 2-face $(\mathfrak{c}, e_1)$, either $\overline{e_i}$ is a non-lattice point or $\overline{e_i}$ is a lattice point and there exists another lattice point $e_{i'}$ in $Q(R)$. Since $|C_1| \neq m-2$ and $|C_1| \neq 2$, $(e_i, e_{i'},f_1)$ and $(\mathfrak{c}, e_{i}, e_{i'})$ span 3-faces. In particular if both non 2-faces $(\mathfrak{c},e_1)$ and $(\mathfrak{c},f_1)$ appear in $Q(R)$ then $(e_i,f_j,f_{j'})$ spans a 3-face and it suffices to conclude this part of the proof. Note that this is the case which was studied in \cite{herzog2015} for $m=n$. 
\item $|C_1|=m-2=1$: Consider the non 2-face $(e_2,e_3)$ in $Q(R)$. The pairs $(e_1,e_3)$, $(e_3, \mathfrak{c})$, $(e_1,e_2)$ and $(\mathfrak{c},e_2)$ do span 2-faces. Furthermore we have $R_4 + \ldots + R_{n+3} \geq 3$. This means that there exists $ j \in [n]$ such that $R_{m+j} \geq 1$. The ray $f_j$ is an extremal ray generator, otherwise $(\mathfrak{c}, e_1)$ spans a 2-face. Since we excluded the case where $|C_2|=1$, $(\mathfrak{c},f_j)$, $(e_1,f_j)$, $(e_2,f_j)$ and $(e_3,f_j)$ span 2-faces. Additionally $(\mathfrak{c},e_2,f_j)$, $(e_1,e_2,f_j)$, $(\mathfrak{c},e_3,f_j)$ and $(e_1,e_3,f_j)$ span 3-faces.
\end{enumerate}

\noindent $\bullet$ Assume that $|C_2|=n-3$. For the non 3-face $(f_{n-2},f_{n-1},f_n)$, we refer to second case of the proof of \hyperref[rigidcomplete]{Theorem \ref*{rigidcomplete}}. There is a small detail here that one needs to pay attention to. If $|C_1|=m-1$, then $e_m$ is not an extremal ray generator of $\sigma_G$. But then the deformation degree $R \in M$ with $R_m = R_{m+n-2} + R_{m+n-1} + R_ {m+n}$ evaluates bigger than or equal to one on $\mathfrak{c} \in \sigma^{(1)}_G$.  The triples $(\mathfrak{c},f_j, f_k)$ are 3-faces where $j ,k \in \{n-1,n-2,n\}$.  \\

\noindent $\bullet$ Assume that $|C_1|=2$. For the non 3-face $(\mathfrak{c},e_1,e_2)$, we have $R_1 \geq 1$, $R_2 \geq 1$ and $R_3 + \ldots + R_m - R_{m+1} \ldots - R_{m+t_2} \geq 1$. This implies that $R_{m+t_2+1} + \ldots + R_{m+n} \geq 3$ where $t_2 = |C_2|$ as before. Then there exists $j \in N(C_1)$ such that $R_{m+j} \geq 1$. Note that if $f_j$ is not an extremal ray generator then  $(\mathfrak{c},e_1,e_2)$ spans a 3-face. Otherwise, $(e_1,e_2,f_j)$ is always a 3-face. The pair $(c,f_j)$ is not a 2-face if and only if $j \in C_2$ and $|C_2|=1$, which is impossible.
\end{proof}

\begin{ex}
Let $G \subset K_{4,5}$ be the connected bipartite graph constructed by removing two edges connected to the vertex $\{1\}$ in $U_1$. This means that there exists a two-sided first independent set $C=C_1 \sqcup C_2 \in \Ione$ with $|C_1|=1$ and $|C_2| = 2$. By \hyperref[non3]{Proposition \ref*{non3}}, $(f_3, f_{4},f_5)$ does not span a three-dimensional face and by \hyperref[obvious]{Proposition \ref*{obvious}} (2), $(e_1, \mathfrak{c})$ does not span a two-dimensional face. In particular we observe that in \hyperref[rep]{Figure \ref*{rep}} the second graph is the intersection subgraph associated to the extremal ray set $(f_3, f_{4},f_5)$ and $(\mathfrak{c},e_1)$. Let us for example consider the compact part of crosscut $Q(R)$ for $R=[2,0,0,0,0,-1,1,1,1] \in M$ as in the figure below. Except from the triples $(\mathfrak{c},e_1,f_3)$ and $(f_3,f_4,f_5)$, all triples in this figure span 3-faces. Therefore $T^1_{\TV(G)}(-R) = 0$.
 \begin{center}
\begin{tikzpicture}[baseline=1,scale=1.6,every path/.style={>=latex},every node/.style={draw,circle,fill=white,scale=0.7}]
  \node       [fill =yellow]     (1) at (0,2.6)  {1};
  \node            (2) at (0,2.1)  {2};
  \node          (m-1) at  (0,1.6) {3};
  \node           (m) at (0,1.1) {4};

  \node       [fill=yellow,rectangle]     (t1) at (1.2,2.6)  {5};
  \node       [fill=yellow,rectangle]      (t2) at (1.2,2.1)  {6};
  \node       [rectangle]      (t3) at (1.2,1.6) {7};

  \node           [rectangle]  (m+n-2) at (1.2,1.1) {8};
  \node          [rectangle]   (m+n-1) at  (1.2,0.6) {9};

 \draw[-] (2) edge (t1);
 \draw[-] (m-1) edge (t1);
 \draw[-] (m) edge (t1);

 \draw[-] (2) edge (t2);
 \draw[-] (m-1) edge (t2);
 \draw[-] (m) edge (t2);

 \draw[-] (2) edge (t3);
 \draw[-] (m-1) edge (t3);
 \draw[-] (m) edge (t3);

 \draw[-] (2) edge (m+n-2);

 \draw[-] (m-1) edge (m+n-2);
 \draw[-] (m) edge (m+n-2);

 \draw[-] (m+n-2) edge (1);
 \draw[-] (m+n-1) edge (1);
 \draw[-] (t3) edge (1);

 \draw[-] (m+n-1) edge (2);
 \draw[-] (m+n-1) edge (m-1);
 \draw[-] (m+n-1) edge (m);

\end{tikzpicture}
\hspace{1cm}
\begin{tikzpicture}[baseline=1,scale=1.6,every path/.style={>=latex},every node/.style={draw,circle,fill=white,scale=0.7}]
  \node       [fill =none]     (1) at (0,2.6)  {1};
  \node            (2) at (0,2.1)  {2};
  \node             (m-1) at  (0,1.6) {3};
  \node           (m) at (0,1.1) {4};

  \node       [fill=none,rectangle]     (t1) at (1.2,2.6)  {5};
  \node       [fill=none,rectangle]      (t2) at (1.2,2.1)  {6};
  \node       [rectangle]      (t3) at (1.2,1.6) {7};
  \node           [rectangle]  (m+n-2) at (1.2,1.1) {8};
  \node          [rectangle]   (m+n-1) at  (1.2,0.6) {9};

 \draw[-] (2) edge (t1);
 \draw[-] (m-1) edge (t1);
 \draw[-] (m) edge (t1);

 \draw[-] (2) edge (t2);
 \draw[-] (m-1) edge (t2);
 \draw[-] (m) edge (t2);

\end{tikzpicture}
\hspace{1cm}
\begin{tikzpicture}[scale=0.7,every path/.style={>=latex},every node/.style={draw,circle,fill=black,scale=0.4}]
  \node      [above,draw=none, circle, fill=none, scale=2]      (a) at (0,0)  {$f_3$};
    \node      [draw=red, circle, fill=red, scale=1]    (a) at (0,0)  {};
  \node        [left, draw=none, circle, fill=none, scale=2]      (b) at (-3,2)  {$\overline{\mathfrak{c}}$};
    \node        [draw=black, circle, fill=black, scale=1]    (b) at (-3,2)  {};
  \node        [right, draw=none, circle, fill=none, scale=2]        (d) at (3,2) {$\overline{e_1}$};
    \node       [draw=black, circle, fill=black, scale=1]    (d) at (3,2) {};
  \node         [left, draw=none, circle, fill=none, scale=2]      (f) at (-2,-3) {$\overline{f_5}$};
    \node    [draw=black, circle, fill=black, scale=1]      (f) at (-2,-3) {};
  \node      [right, draw=none, circle, fill=none, scale=2]       (g) at (2,-3) {$\overline{f_4}$};
    \node        [draw=black, circle, fill=black, scale=1]       (g) at (2,-3) {};

    \draw[dashed,red] (d) edge (b);
    \draw[-,green] (b) edge (g);
    \draw[-,green] (d) edge (g);
    \draw[-,green] (f) edge (g);
    \draw[-,green] (a) edge (g);
\draw[dashed,red] (0,-0.3) edge (1.7,-2.8);
\draw[dashed,red] (-1.7,-2.8) edge (1.7,-2.8);
\draw[dashed,red] (-1.7,-2.8) edge (0,-0.3);
    \draw[-,green] (f) edge (a);
    \draw[-,green] (a) edge (b);

  \draw [-,green] (f) edge (b);
 \draw [-,green] (f) edge (d);
 \draw [dashed,green] (f) edge (a);

\draw [-,green] (a) edge (d);
\path[draw, fill=green!20, opacity=.65] (0,0)--(-2,-3)--(3,2)--cycle;
\path[draw, fill=green!20, opacity=.65] (2,-3)--(0,0)--(-3,2)--cycle;
\path[draw, fill=green!20, opacity=.28] (-2,-3)--(0,0)--(-3,2)--cycle;
\path[draw, fill=green!20, opacity=.28] (2,-3)--(0,0)--(3,2)--cycle;
\path[draw, dashed, fill=red!20, opacity=.0] (2,-3)--(0,0)--(-2,-3)--cycle;
\end{tikzpicture}


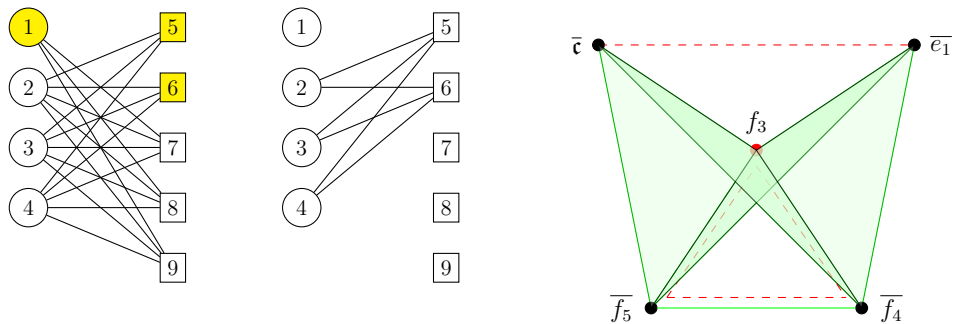
\captionof{figure}{Examining the rigidity through bipartite graphs.}\label{rep}
\end{center}

\end{ex}
\section*{Acknowledgement}
\noindent This paper is part of author's Ph.D. thesis which has been published online at Freie Universit\"at Berlin Library. The author wishes to express her gratitude to her advisor Klaus Altmann for suggesting the problem and many helpful conversations. The author also thanks Berlin Mathematical School for the financial support.

\end{document}